\documentclass[10pt,reqno]{amsart}

\usepackage[english]{babel}

\usepackage{graphicx}
\usepackage{subfig}
\DeclareGraphicsExtensions{.pdf,.png,.jpg}

\usepackage[latin1]{inputenc}
\usepackage[english]{babel}
\usepackage{float}

\if{
\DeclareGraphicsExtensions{.pdf,.png,.jpg}
\usepackage{fullpage}
\usepackage[mathscr]{eucal}
\usepackage{amsfonts}
\usepackage{graphicx}
\usepackage{subfig}
\usepackage{amsmath, amsthm} 
\usepackage[latin1]{inputenc}
\usepackage{t1enc}
\usepackage[english]{babel}
\usepackage{ulem}
\usepackage{float}
}\fi

 \newtheorem{rem}{Remark}[section]
 \newtheorem{theorem}{Theorem}[section]
 \newtheorem{defi}{Definition}[section]
 \newtheorem{lemma}{Lemma}[section]
 \newtheorem{proposition}{Proposition}[section]
 \newtheorem{corol}[proposition]{Corollary}

\newcommand{\intt}{\int\limits}   
    
\newcommand{\ms}{\sigma}
\newcommand{\mO}{\Omega} 
\newcommand{\N}{\mathbb{N}}
\newcommand{\R}{\mathbb{R}}

\newcommand{\E} {\mathbb{E}}
\newcommand{\bP}{\mathbb{P}}
\newcommand{\cA}{\mathcal{A}}
\newcommand{\cC}{\mathcal{C}}
\newcommand{\cF}{\mathcal{F}}

\usepackage{epsf,epsfig}
\usepackage{graphicx, colordvi,amsmath,amsfonts,amssymb}
\usepackage{color}
\usepackage{tikz}

\usepackage{array}

\usepackage[latin1]{inputenc}

\newcommand{\ximud}{x_{i-\frac{1}{2}}}
\newcommand{\xipud}{x_{i+\frac{1}{2}}}

\newcommand{\be}{\begin{eqnarray}}
\newcommand{\ee}{\end{eqnarray}}
\newcommand{\beno}{\begin{eqnarray*}}
\newcommand{\eeno}{\end{eqnarray*}}
\newcommand{\barr}[1]{\begin{array}{#1}}
\newcommand{\earr}{\end{array}}
\newcommand{\VECT}[1]{\left(\barr{c} #1 \earr\right)}
\newcommand{\MAT}[2]{\left(\barr{#1} #2 \earr\right)}
\newcommand{\eps}{\varepsilon}

\newcommand{\dt}{{\Delta t}}
\newcommand{\sdt}{{\sqrt{\dt}}}
\newcommand{\sth}{\sqrt{\theta}}
\newcommand{\dx}{{\Delta x}}

\newcommand{\ma}{\alpha}
\newcommand{\mb}{\beta}
\newcommand{\mg}{\gamma}
\newcommand{\ml}{\lambda}
\newcommand{\mt}{\theta}
\newcommand{\disp}{\displaystyle}

\newcommand{\mS}{\Sigma}

\newcommand{\cO}{\mathcal O}
\newcommand{\cS}{\mathcal S}
\newcommand{\cE}{\mathcal E}
\newcommand{\cT}{\mathcal T}

\newcommand{\xmin}{x_{min}}
\newcommand{\xmax}{x_{max}}

\newcommand{\xumin}{x_{1,min}}
\newcommand{\xumax}{x_{1,max}}
\newcommand{\xdmin}{x_{2,min}}
\newcommand{\xdmax}{x_{2,max}}

\newcommand{\converge}{\rightarrow}

\newcommand{\ud}{\frac{1}{2}}
\newcommand{\wcT}{\widetilde\cT}

\newcount\Comments  
\usepackage{color}
\definecolor{darkgreen}{rgb}{0,0.5,0}
\definecolor{purple}{rgb}{1,0,1}
\newcommand{\kibitz}[2]{\ifnum\Comments=1\textcolor{#1}{#2}\fi}

\newcommand{\Olivier}[1]{\kibitz{blue}     {\bf [Olivier: #1]}}
\newcommand{\comment}[1]{\kibitz{darkgreen}     {\centerline{\fbox{#1}} }}

\newcommand{\MODIF}[1]{{#1}}

\newcommand{\TOREMOVE}[1]{{\color{green} #1}}
\newcommand{\REMOVE}[1]{}

\renewcommand{\comment}[1]{}

\newcommand{\scriptdt}{{\scriptstyle\Delta t}}
\newcommand{\scriptmdt}{{\scriptstyle -\Delta t}}
\newcommand{\<}{\langle}
\renewcommand{\>}{\rangle}

\newcommand{\Extwo}    {Example 1}  
\newcommand{\Exfour}   {Example 2}  
\newcommand{\Exfive}   {Example 3}  
\newcommand{\Exsix}    {Example 4}  
\newcommand{\Exhei}    {Example 5}  
\newcommand{\Exsev}    {Example 6}  
\newcommand{\Exnin}    {Example 7}  
\newcommand{\Exten}    {Example 8}  

\newcommand{\FULL}[1]{} \newcommand{\SHORT}[1]{#1}  

\title[SLDG scheme for first- and second-order PDEs]
{Semi-Lagrangian discontinuous Galerkin schemes for some first- and second-order partial differential equations}

\author{Olivier Bokanowski}
\address{Laboratoire Jacques-Louis Lions,
Universit\'e Paris-Diderot (Paris 7),
75205 Paris CEDEX 13, France\\ 
and 
Unit\'e de Math\'ematiques Appliqu\'ees, ENSTA ParisTech, 91120 Palaiseau, France.}
\email{boka@math.jussieu.fr}
\author{Giorevinus Simarmata}
\address{Finance RI Department - Rabobank International, Europalaan 44, 3526 KS, Utrecht, The Netherlands}
\email{giorevinus.simarmata@rabobank.com}
\thanks{}


\begin{document}

\begin{abstract}
Explicit, unconditionally stable, high-order schemes for the approximation of some first- and 
second-order linear, time-dependent partial differential equations (PDEs) are proposed.
The schemes are based on a weak formulation of a semi-Lagrangian scheme using discontinuous Galerkin (DG) elements.
It follows the ideas of the recent works of Crouseilles, Mehrenberger and Vecil (2010), Rossmanith and Seal (2011),
for first-order equations, based on exact integration, quadrature rules, and splitting techniques for the treatment of two-dimensional
PDEs. For second-order PDEs the idea of the scheme 
is a blending between weak Taylor approximations and projection on a DG basis.
New and sharp error estimates are obtained for the fully discrete schemes and for variable coefficients.
In particular we obtain high-order schemes, unconditionally stable and convergent, 
in the case of linear first-order PDEs, or linear second-order PDEs with constant coefficients. 
In the case of non-constant coefficients, we construct, in some particular cases,
"almost" unconditionally stable second-order schemes and give precise convergence results.
The schemes are tested on several academic examples. 
\end{abstract}

\maketitle

\section{Introduction}

In this paper we consider equations of the form
\be\label{eq:main}
  u_t - \frac{1}{2}Tr (\sigma \sigma^T D^2 u) + b\cdot \nabla u + r u = 0, \quad x\in \Omega,\ t\in (0,T),
\ee
where $\Omega\subset \R^d$ is a box
(with some boundary conditions on $\partial \Omega$),  $\sigma$ (matrix), $b$ (vector) and $r$ (scalar) may be $x$-dependent,
at least Lipschitz continuous,
together with an initial condition
\be\label{eq:mainu0}
  u(0,x)=u_0(x), \quad x\in\Omega,
\ee
with $u_0\in L^2(\Omega)$.
The matrix $\sigma$ may be zero or positive semidefinite.
Unless otherwise stated,
we will in general assume periodic boundary conditions for \eqref{eq:main} 
in order to avoid difficulties on the boundary.
We will assume sufficient regularity on the data in order to have 
existence and uniqueness of weak solutions of \eqref{eq:main}-\eqref{eq:mainu0}, and so that
$t\converge u(t,.)$ is in $C^0([0,T], L^2(\Omega))$.

We study and propose new semi-Lagrangian Discontinuous Galerkin schemes, also abbreviated 
"SLDG" in this work, in order to approximate the solutions of \eqref{eq:main}-\eqref{eq:mainu0}.

The semi-Lagrangian (SL) approach (see~\cite{fal-fer-1998}, or the textbook~\cite{fal-fer-14}),
is based on the approximation of the "method of characteristics". 
By considering a weak formulation of this principle, an explicit SLDG scheme is obtained.
In the case of first-order PDEs with constant coefficient,  our approach is based on a similar method as in
the recent works of Crouseilles, Mehrenberger and Vecil~\cite{cro-meh-vec-2010} (for the Vlasov equation in plasma physics),
Rossmanith and Seal~\cite{ros-sea-2011}. However our approach seems not to have been considered for variable coefficients.
It is slightly different from the work of Qiu and Shu~\cite{Qiu_Shu_2011} (see also Restelli et al~\cite{Res_Bon_Sac_06}),
where first a weak formulation of the PDE is considered, and then quadrature formulae are used 
(see also \cite{richtmyer1967} for the original approach).
Here we will furthermore introduce new SLDG schemes
for second-order PDEs for which we prove stability and convergence results, and obtain
higher-orders of accuracy when possible.

First, in Section 2, 
we revisit the one-dimensional first-order advection equation with non-constant advection term $b(x)$
(case $\sigma=0$ in \eqref{eq:main}). We give a new unconditional stability result, and convergence proof,
extending similar results of \cite{cro-meh-vec-2010}, \cite{ros-sea-2011} (or \cite{Qiu_Shu_2011})
that was obtained for the case of a constant advection term.
The unconditional stability property 
can be interesting when compared to a standard DG approach where a restrictive
CFL condition 
must in general be considered~\cite{coc-2003}.

Based on the operator construction for first-order advection,
we then introduce, in Section 3, new schemes for linear second-order PDEs of type~\eqref{eq:main}, in the form of
explicit high-order SLDG schemes.
These schemes are based, for the temporal discretization, on the use of 
"weak Taylor approximations",  see in particular the review book by Kloeden and Platen \cite{klo-pla-95}
(see also Kushner \cite{kus-77} and the review book by Kushner and Dupuis~\cite{kus-dup-92},
Platen~\cite{Platen_1984}, Milstein~\cite{Milstein_1986}, 
Talay~\cite{Talay_1984}, Pardoux and Talay~\cite{Pardoux_Talay_1985},
Menaldi~\cite{Menaldi_1989}, Camilli and Falcone~\cite{Camilli_Falcone_95}, 
Milstein and Tretyakov~\cite{Milstein_Tretyakov_2001}, \cite{Debrabant_2010}).
Such approximations where used by R.~Ferretti in~\cite{fer-2010} 
as well as in Debrabant and Jakobsen~\cite{Debrabant_Jakobsen_2012}
in the context of semi-Lagrangian schemes, using interpolation methods for the space variable.
\MODIF{
The problem of coupling such approximations with a spatial grid approximation, in particular using a high-order interpolation method,
can be the stability and the convergence proof of the method. The $P_1$ interpolation is known to be $L^\infty$ stable,
but it is only second-order accurate in space (for regular data).
Some higher-order SL approximations have been proved to be stable (and convergent) for specific equations
and under large CFL numbers (see \cite{fer-02,car-fer-rus-05}), or for some advection equations when 
the SL scheme can be reinterpreted in a weak form (we refer in particular 
to Ferretti's work \cite{fer-2010,fer-2013}). 
}

The schemes of the present paper can be seen as projections of these approximation on a discontinous Galerkin basis.
We will in particular propose a second-order approximation (in time)
corresponding to a Platen's scheme~\cite[Chapter 14]{klo-pla-95}, but higher-order approximations (in time) could be obtained in the same way.
The scheme will be proved to be also high-order in space, stable and convergent under a weak CFL condition 
(of the form $\dx^4 \leq \ml \dt$ for some constant $\ml$, where $\dt$ and $\dx$ denote the time and mesh steps).

For the more simple case of second-order PDEs with constant coefficients,
we also propose explicit and unconditionally stable schemes,
high-order in space and up to third-order in time
(higher-order can be obtained~\cite{bok-bon-2013}).

In section 4 we consider extensions to some linear two-dimensional PDEs.
For first-order PDEs, we show how to combine the scheme with higher-order splitting techniques,
like Strang's splitting, but also Ruth's third-order splitting~\cite{Ruth_1983},
Forest's fourth-order splitting~\cite{Forest_1987} and Yoshida's sixth-order splitting~\cite{Yoshida_1990}
(see also \cite{Forest_Ruth_1990} and \cite{Yoshida_1993}). 
A splitting strategy to treat general second-order PDEs with constant coefficients is explained.
\MODIF{
The case of second-order PDEs with variable diffusion coefficients is discussed but only treated in some specific cases
(see Remark~\ref{rem:special_splittings} as well as \Exnin\ and~\Exten\ of Section~\ref{sec:num}).
The general case will be treated in a forthcoming work (see however Remark~\ref{rem:general_splittings}). 
}

Finally in Section 5 we show the relevance of our approach on several academic numerical examples 
in one and two dimensions (using Cartesian meshes),
including also a Black and Scholes PDE in mathematical finance.

The advantage of the proposed schemes is that they combine the DG framework 
which allows high-order spatial accuracy and the potential of degree adaptivity,
together with unconditional stability properties in the $L^2$ norm 
from the weak formulation of the semi-Lagrangian scheme. 

Note that our general strategy is to use a Cartesian grid, a particular one-dimensional advection scheme, 
and splitting techniques
(for more standard Discontinuous Galerkin approaches, see for instance \cite{coc-shu-07} or \cite{dipietro-ern-2012}).


Ongoing works using the current approach concern the construction of 
higher-order schemes for general second-order PDEs~\cite{bok-bon-2013}, extensions to
nonlinear PDEs arising from deterministic control~\cite{bok-che-shu-2013} or from stochastic control.

{\bf Acknowledgments.}
\textit{This work was partially supported by the EU under the 7th Framework Programme Marie Curie Initial
Training Network ``FP7-PEOPLE-2010-ITN'', SADCO project, GA number 264735-SADCO.
The first author also wishes to thank K. Debrabant for pointing out Platen's works
as well as an anonymous referee for related references, which helped to simplify the presentation.
We also thank D. Seal for useful comments and references. We are grateful to C.-W. Shu for pointing out problems 
in the preliminary version of the present work.
}

\section{Advection equation}

We first consider the semi-Lagrangian Discontinuous Galerkin scheme 
(SLDG for short) for the following one-dimensional first-order PDE, as in~\cite{cro-meh-vec-2010}
\begin{equation}\label{eq:adv}
  \begin{cases} 
    v_t + b(x) v_x = 0, \qquad (t,x) \in (0,T) \times \Omega\\
    v(0,x) = v_0(x), \qquad x \in \Omega 
  \end{cases}
\end{equation}
where $\mO= (x_{min}, x_{max})$, together with periodic boundary conditions on $\mO$.  


In order to simplify the presentation and the proofs, 
we will assume that $\Omega=(0,1)$ and that $b$ is a $1$-periodic function.

Let $y=y_x$ denote the solution of the differential equation
\be\label{eq:y_ode}
  \left\{ 
   \begin{array}{l} 
    \dot y(t)= b(y(t)), \quad t\in \R\\
    y(0)= x.
   \end{array}
   \right.
\ee
We will also assume that $b(\cdot)$ is Lipschitz continuous.

Let $N\in\N$, $N\geq 1$, $\dt = \frac{T}{N}$ a time step and $t_n = n\dt$ a time discretization.
Let $$ v^n(x):= v(t_n,x). $$

By the method of characteristics, the solution of~\eqref{eq:adv} satisfies
\begin{eqnarray}\label{eq:dpp}
  v^{n+1}(x)= v^n(y_x(-\dt)). 
\end{eqnarray}
Then we aim to obtain a fully discrete scheme.

Let us consider a space discretization that is considered uniform for the sake of simplicity of presentation.
Let $\dx = \frac{x_{max}-x_{min}}{M}$ for some integer $M\geq 1$,
$x_{i-\frac{1}{2}} := x_{min} + i\dx$, $\forall i = 0,.., M$, and $I_i := (\ximud, \xipud)$. 
Let $k \in \N$.
We define $V_k$ as the space of discontinuous-Galerkin elements on $\Omega$ with polynomials of degree~$k$, that is:
\begin{eqnarray}\label{eq:DGspace}
V_k&=&\{v \in L^2(\mO, \R) : v|_{I_i} \in P_k, \forall i=0, .., M-1\}
\end{eqnarray}
where $P_k$ denotes the set of polynomials of degree at most $k$.

\begin{rem}
In the classical semi-Lagrangian approach, looking for $u^n(x)$, an approximation of $v(t_n,x)$,
a first "direct" iterative scheme for \eqref{eq:dpp} would be
\begin{equation}
  u^{n+1}(x_i)=[u^n](y_{x_i}(-\dt))
\end{equation}
where $[u^n](x)$ denotes some interpolation of the function $u^n$ at point $x$. 
We could take for instance a set of $k+1$ values $(x^i_\ma)_{\ma=0,\dots,k}$ in each 
interval $I_i$, and define the new polynomial $u^{n+1}$ such that 
$u^{n+1}(x^i_\ma):=[u^n](x^i_\ma - b \dt)$ for all $\ma=0,\dots,k$.
However, given the discontinuities between the intervals ${I_i}$, 
this may lead to instabilities in the scheme (\cite{Qiu_Shu_2011}).
For instance, taking $x^i_\ma$ to be the Gauss quadrature points on each interval $I_i$ is in general unstable
(see Appendix~\ref{app:unstabilities}, see also \cite{morton1988}).
\end{rem}

Here we consider a Lagrange-Galerkin approach by taking the weak form of \eqref{eq:dpp}: 
for $n=0,\dots,N-1$, find $u^{n+1}\in V_k$ such that 
\be  \label{eq:LGex1}
  \intt_{\Omega}u^{n+1}(x) \varphi(x) dx 
   & = & \intt_{\Omega}u^n(y_x(-\scriptdt))\varphi(x) dx,\quad \forall \varphi \in V_k,
\ee
and for $n=0$, find $u^0 \in V_k$ such that:
\begin{eqnarray}\label{eq:LGex0}
  \intt_{\Omega}u^0(x) \varphi(x) dx = \intt_{\Omega}v_0(x) \varphi(x) dx,\quad \forall \varphi \in V_k. 
\end{eqnarray}
From now on, we rewrite~\eqref{eq:LGex1}
in the following abstract form : 
$$
  u^{n+1} = \cT_{b\dt} (u^n).
$$

In the case of a constant coefficient $b$, $y_x(-\dt)=x - b\dt$,
and $u^n(x-b\dt)$ is a piecewise constant polynomial.
The integral $\int_{I_i} u^n(x-b\dt) \varphi(x) dx$ will have in general two regular parts.  
Each part involves a polynomial of degree at most $2k$ and the Gaussian quadrature rule with $k+1$ points is applied 
and is exact. At this stage the method is the same as in~\cite{cro-meh-vec-2010}, or~\cite{ros-sea-2011}.
Hence the new function $u^{n+1}$ can be computed by solving exactly~\eqref{eq:LGex1}.


However, if $b(x)$ is not a constant, 
$x\converge u^n(y_x(-\scriptdt))$ is no more a piecewise polynomial.
Therefore the computing procedure for the right-hand-side (R.H.S.) of \eqref{eq:LGex1}
can no more be exact.

In order to obtain an implementable scheme, 
a precise ODE integration for the characteristics and a quadrature rule can be used.
We follow an approach 
very similar to \cite{Qiu_Shu_2011} for variable coefficients.
It consists in using Gaussian quadrature formula to approximate \eqref{eq:LGex1} 
in regions where the involved functions are smooth. 

\begin{rem}
Indeed, in \cite{Qiu_Shu_2011}, an other SLDG scheme is presented, but our form is equivalent to one form of SLDG 
as explained in \cite[Proposition 4.5]{Qiu_Shu_2011}. This may lead to different programming algorithms however.
\end{rem}

\subsection{Preliminaries}
Let $\{x_\alpha\}_{\alpha=0,..,k}$ be the set of Gauss points in the interval $(-1,1)$,
with its corresponding weights $\{w_\ma\}_{\ma=0,..,k}$  ($w_\ma>0$), such that:
\begin{align}
 \forall p \in P_{2k+1}, \quad 
 \intt_{-1}^1 p(x) dx   =   \sum_{\alpha=0}^k w_\alpha p(x_\alpha).
\end{align}
In particular, we get on the interval $I_i$, 
\be
  \forall p \in P_{2k+1}, \quad 
  \intt_{\ximud}^{\xipud} p(x) dx =  \sum_{\alpha=0}^k w_\alpha^i p(x_\ma^i),
\ee
where $x_\alpha^i := x_i + x_\alpha \dx \equiv \ximud + \frac{1}{2}(1+x_\alpha) \dx$  and $w_\alpha^i := \frac{\dx}{2}w_\alpha$.

To each set of Gauss points $\{x_\alpha^i\}_{\alpha=0,..,k}$ in $I_i$, 
we can associate the corresponding Lagrange polynomials (dual basis)
$\{\varphi_\alpha^i\}_{\alpha=0,.., k}$ defined by 
\begin{eqnarray}\label{eq:base}
  \varphi_{\alpha}^i(x) := 1_{I_i}(x)\ 
  \disp \prod_{\substack{0\le \beta \le k \\ \beta \neq \alpha}} \frac{x-x_\beta}{x_\alpha-x_\beta}.
\end{eqnarray}
For any $u^n \in V_k$, there exist coefficients
$(u_{\alpha,i}^n)_{\alpha=0,..,k}^{i=0,..,M-1} \in \mathbb{R}$ such that:
\begin{eqnarray}\label{eq:base2}
u^{n}(x) = \sum_{i=0}^{M-1} \sum_{\alpha=0}^k u_{\alpha,i}^{n} \varphi_\alpha^i (x).
\end{eqnarray}
In particular, the left-hand side of \eqref{eq:LGex1} for $\varphi=\varphi_\ma^i$ becomes
\begin{eqnarray}\label{eq:lLGex}
\intt_{\Omega} u^{n+1}(x) \varphi_\alpha^i (x) dx &=& \intt_{I_i} u^{n+1}(x) \varphi_\alpha^i (x) dx 
 =  u_{\alpha,i}^{n+1} w_\alpha^i \nonumber.
\end{eqnarray}


\medskip\noindent
\subsection{Definition of the scheme in the general case}
Due to the discontinuities of~$u^n$, we separate the right-hand side of \eqref{eq:LGex1} 
into several integral parts involving only regular functions:
the R.H.S.\ of~\eqref{eq:LGex1} is 
approximated by the Gaussian quadrature rule on each sub-interval where $u^n(y_x(-\dt))$ is a regular function.

For a given mesh cell $I_i$,
we first consider the points
$(x_{i,q})_{1\leq q\leq p_i}$ (in finite number) of the interval $(\ximud,\xipud)$,
such that for $1\leq q\leq p_i$,  $y_{x_{i,q}}(-\dt)=x_{\ell_{i,q}-\frac{1}{2}}$ for some 
$\ell_{i,q}\in \mathbb{Z}$, 
and $x_{i,0}:=\ximud$, $x_{i,p_i+1}:=\xipud$ (see Figure~\ref{fig:deform}).
Then we apply the Gaussian quadrature rule on each interval $J_{i,q}=(x_{i,q},x_{i,q+1})$
and obtain the following quadrature rule, for any polynomial $\varphi\in V_k$:
\be\label{eq:LGex_nonconstant_quadrature}
  \intt_{I_i}u^n(y_x(-\scriptdt))\varphi(x) dx 
  & =  & \sum_{q=0}^{p_i} \intt_{x_{i,q}}^{x_{i,q+1}} u^n(y_x(-\scriptdt))\varphi(x) dx\\
  & \simeq & \sum_{q=0}^{p_i} \sum_{\ma=0}^k  \tilde w^i_{q,\ma} u^n(y_{\tilde x^i_{q,\ma}}(-\scriptdt))\varphi(\tilde x^i_{q,\ma}),
\ee
with $\tilde w^i_{q,\ma}:=\frac{w_\ma}{2} (x_{i,q+1}-x_{i,q})$ and 
$\tilde x^i_{q,\ma} := 
x_{i,q}+\frac{1}{2}(1+x_\ma)(x_{i,q+1}-x_{i,q})  \equiv \frac{x_{i,q}+x_{i,q+1}}{2} +
x_\ma (\frac{x_{i,q+1}-x_{i,q}}{2})$.

\medskip\noindent{\underline{Definition of the scheme (operator $\wcT_{b,\dt}$):}
$u^{n+1}$ is the unique element of $V_k$ satisfying for all $\varphi\in V_k$, 
\be\label{eq:LGex_nonconstant_scheme}
  \intt_\mO u^{n+1}(x)\varphi(x) dx 
  & = &  
  \sum_{i=0}^{M-1} \sum_{q=0}^{p_i} \sum_{\ma=0}^k  
  \tilde w^i_{q,\ma} u^n\big(y_{\tilde x^i_{q,\ma}}(-\scriptdt)\big)\varphi(\tilde x^i_{q,\ma}).
\ee
The scheme is made explicit by using formula~\eqref{eq:LGex_nonconstant_scheme} on each $\varphi=\varphi^j_\mb$.
The scheme equivalently defines an operator $\wcT_{b,\dt}$ such that 
$$
  u^{n+1} =  {\wcT}_{b,\dt} \ u^n .
$$
In particular, if $b$ is constant, then $\wcT_{b,\dt} = \cT_{b,\dt}$, and this is no more true if $b$ is non-constant.
\begin{defi}
For further analysis, let us introduce the following scalar product on $V_k$ 
(where the index "G" stands for the use of the Gaussian quadrature rule):
\be \label{eq:defi_bilin_G}
  (\varphi,\psi)_{G} : = \sum_{i=0}^{M-1} \sum_{q=0}^{p_i} \sum_{\ma=0}^k  
  \tilde w^i_{q,\ma}\ \varphi(\tilde x^i_{q,\ma})\ \psi(\tilde x^i_{q,\ma}).
\ee
\end{defi}
Then the scheme \eqref{eq:LGex_nonconstant_scheme}
is equivalently defined by
$$
  (u^{n+1}, \varphi) = (u^n(y_\cdot(-\dt)), \varphi)_G,\ \  \forall \varphi \in V_{k}.
$$

\begin{figure}
\hspace*{-0.5cm}
\begin{picture}(0,0)%
\includegraphics{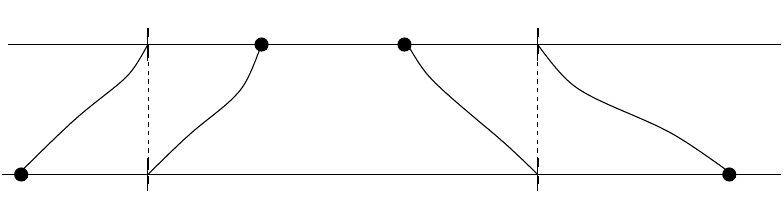}%
\end{picture}%
\setlength{\unitlength}{2279sp}%
\begingroup\makeatletter\ifx\SetFigFont\undefined%
\gdef\SetFigFont#1#2#3#4#5{%
  \reset@font\fontsize{#1}{#2pt}%
  \fontfamily{#3}\fontseries{#4}\fontshape{#5}%
  \selectfont}%
\fi\endgroup%
\begin{picture}(10844,3022)(-21,-2924)
\put(3421,-151){\makebox(0,0)[lb]{\smash{{\SetFigFont{7}{8.4}{\rmdefault}{\mddefault}{\updefault}{\color[rgb]{0,0,0}$x_{i,1}$}%
}}}}
\put(5311,-151){\makebox(0,0)[lb]{\smash{{\SetFigFont{7}{8.4}{\rmdefault}{\mddefault}{\updefault}{\color[rgb]{0,0,0}$x_{i,2}$}%
}}}}
\put(1711,-2851){\makebox(0,0)[lb]{\smash{{\SetFigFont{7}{8.4}{\rmdefault}{\mddefault}{\updefault}{\color[rgb]{0,0,0}$y_{x_{i,1}}(-\dt)$}%
}}}}
\put(6841,-2851){\makebox(0,0)[lb]{\smash{{\SetFigFont{7}{8.4}{\rmdefault}{\mddefault}{\updefault}{\color[rgb]{0,0,0}$y_{x_{i,2}}(-\dt)$}%
}}}}
\put(1531,-61){\makebox(0,0)[lb]{\smash{{\SetFigFont{7}{8.4}{\rmdefault}{\mddefault}{\updefault}{\color[rgb]{0,0,0}$x_{i,0}:=x_{i-1/2}$}%
}}}}
\put(6931,-61){\makebox(0,0)[lb]{\smash{{\SetFigFont{7}{8.4}{\rmdefault}{\mddefault}{\updefault}{\color[rgb]{0,0,0}$x_{i,3}:=x_{i+1/2}$}%
}}}}
\put(9271,-2851){\makebox(0,0)[lb]{\smash{{\SetFigFont{7}{8.4}{\rmdefault}{\mddefault}{\updefault}{\color[rgb]{0,0,0}$y_{x_{i+1/2}}(-\dt)$}%
}}}}
\put(  1,-2851){\makebox(0,0)[lb]{\smash{{\SetFigFont{7}{8.4}{\rmdefault}{\mddefault}{\updefault}{\color[rgb]{0,0,0}$y_{x_{i-1/2}}(-\dt)$}%
}}}}
\end{picture}%
%
\caption{\label{fig:deform} Determination of the point of discontinuity of the data.}
\end{figure}
\subsection{Stability and error estimate for constant drift coefficient}\label{sec:2.1b}

The weak form \eqref{eq:LGex1} gives the stability of the scheme in the $L^2$ norm, at least in the case when $b=const$.
Indeed, taking $\varphi=u^{n+1}$ in~\eqref{eq:LGex1} we get
$$ \| u^{n+1} \|_{L^2}^2 = (u^n(\cdot -b \dt),\ u^{n+1}) \leq \|u^n( \cdot  - b\dt) \|_{L^2}\, \|u^{n+1}\|_{L^2}, $$
where $\|\cdot\|_{L^2}$ denotes the $L^2$ norm on $\mO$ and $(\, . ,\,  .)$ is the associated scalar product. 
Then, by the periodic boundary condition, $\|u^n( \cdot  - b\dt) \|_{L^2}=\|u^{n}\|_{L^2}$ 
and therefore
\be
  \label{eq:firstbound}
  \| u^{n+1} \|_{L^2} \leq  \|u^{n}\|_{L^2}.
\ee
This proof works only for $b$ constant, however.

For any $w\in L^2$, we denote its projection on $V_k$ by $\Pi w$, corresponding to the unique element of $V_k$ such that 
\be
  \| w - \Pi w \|_{L^2} = \inf_{f \in V_k} \|w - f\|_{L^2}.
\ee

\begin{rem}
The function $u^{n+1}$ defined by~\eqref{eq:LGex1} corresponds to the projection of the function 
$x\converge u^n(y_x(-\dt))$ on the space $V_k$:
$$u^{n+1} = \Pi (u^n(y_{\cdot}(-\dt)),$$
and, in the same way, we have $u^0=\Pi v_0$.
\end{rem}

\FULL{
Hence, the previous stability property of the scheme is in fact a consequence of the $L^2$ norm stability of 
the projection, $\|\Pi w \|_{L^2} \leq \| w\|_{L^2}$,
and of the stability of the advection ($\| u^n(\cdot - b\dt) \|_{L^2} \leq \| u^n \|_{L^2}$).
}

We now recall a simple estimate for the $L^2$ projection on $V_k$. 

\begin{lemma}[Projection error]\label{lem:proj}
Let $k\geq 0$ and $\ell\leq k$. If $w\in \cC^{\ell+1}$, then 
 $$ \| w - \Pi w \|_{L^2} \leq |\mO|^{1/2} C_{\ell}(w)\ \dx^{\ell+1}$$
where $C_\ell(w) := \frac{1}{2^{\ell+1} (\ell+1)!}  \| w^{(\ell+1)}\|_{L^\infty}$. 
\end{lemma}
\begin{proof}
Let us write $w=P + R $ where $P$
is the element of $V_k$ corresponding, on each interval $I_i$, to the Taylor expansion of $w$ centered at $x_i$ and of degree $\ell$.
We have $\|w- \Pi w \|_{L^2} \leq \| w - P\|_{L^2} = \| R\|_{L^2}\leq |\Omega|^{1/2} \|R\|_{L^\infty}$. By the definition of $R$ and usual Taylor estimates, we have 
$\|R\|_{L^\infty} \leq C_\ell \dx^{\ell+1}$. 
\end{proof}

Let $v^n(x):=v(t_n,x)$ where $v$ denotes the exact solution of \eqref{eq:adv}.
Using the $L^2$-stability of the projection, it is straightforward to show that 
$ \|u^{n+1} - \Pi v^{n+1} \|_{L^2} = \| \Pi (u^n(\cdot - b\dt) - v^n (\cdot - b\dt)\|_{L^2} \leq \| u^n - v^n \|_{L^2}$, 
therefore we have 
$$
  \|u^{n+1} - v^{n+1} \|_{L^2} \leq \| u^n - v^n \|_{L^2} + \| v^{n+1} - \Pi v^{n+1}\|_{L^2}.
$$
By using Lemma \ref{lem:proj}, this leads to the following known convergence result \cite{Qiu_Shu_2011}.

\begin{theorem} \label{th:th0}
Let $k\geq 0$ and $b$ be a constant. 
Assume the initial condition  $v_0$ is $1$-periodic and in $\cC^{k+1}$.
Then, the following estimate holds:
\be \label{eq:th0}
  \| u^{n} - v^{n} \|_{L^2} \leq  \|u^0-v^0\|_{L^2}  +   C T \frac{\dx^{k+1}}{\dt}, 
  \quad \forall n\leq N,
\ee
where the constant $C$ depends only of $|\mO|$ and $k$.
\end{theorem}

\begin{rem}
By taking $\dt=\dx$ this leads to an error estimate in $O(\dx^k)$. However the examples (such as in Example 1) 
will show a numerical behavior in $O(\dx^{k+1})$ (as already remarked also in \cite{Qiu_Shu_2011}). 
We refer to the recent work in~\cite{ste-meh-bou-2013} for more insight about this gap.
\end{rem}

\if{
\proof[Proof of Theorem~\ref{th:th0}]
Let us remark that if $u^n$ is defined in $V_k$ by using the Gaussian quadrature formula (such that \eqref{eq:LGex1}
holds for any $\varphi \in V_k$), then it also holds for any $\varphi \in V_{k+1}$: 
\begin{eqnarray}\label{eq:LGexk1}
  \intt_{\Omega}u^{n+1}(x) \varphi(x) dx = \intt_{\Omega}u^n(x - b\dt) \varphi(x) dx,\quad \forall \varphi \in V_{k+1}.
\end{eqnarray}
Indeed, the Gaussian quadrature rule holds for polynomials  of degree $2k+1$. Hence 
the right hand side of~\eqref{eq:LGexk1} corresponds also for $\varphi \in V_{k+1}$ to the result of the Gaussian quadrature rule (and this is similar
for the left hand side of~\eqref{eq:LGexk1}). 

Now, let $\tilde v^{n+1} := \Pi_{V_{k+1}} v^{n+1}$, the projection on $V_{k+1}$. Then it holds
\beno
  \| u^{n+1} - \tilde v^{n+1} \|_{L^2}^2  
  & = & ( u^{n+1}, u^{n+1} - \tilde  v^{n+1} )  -  (\tilde v^{n+1}, u^{n+1} - \tilde v^{n+1}) \\
  & = & ( u^n(\cdot - b \dt), u^{n+1} - \tilde v^{n+1} ) -  (\tilde v^{n+1}, u^{n+1} -   \tilde v^{n+1})  \\
  & = & ( u^n(\cdot - b \dt) - \tilde v^{n+1},\  u^{n+1} - \tilde v^{n+1}) \\
  & \leq & \| u^n(\cdot - b \dt) - \tilde v^{n+1} \|_{L^2}  \| u^{n+1} - \tilde v^{n+1} \|_{L^2}.
\eeno
and therefore 
$$  \| u^{n+1} - \tilde v^{n+1} \|_{L^2} \leq \| u^n(\cdot - b \dt) - \tilde v^{n+1} \|_{L^2}.  $$
Since $\|v^{n+1} - \tilde v^{n+1} \|_{L^2} \leq C \dx^{k+2}$ and $v^{n+1}= v^n(\cdot - b\dt)$,
\beno
  \| u^{n+1} - v^{n+1} \|_{L^2} 
     & \leq &  \| u^n(\cdot - b\dt)  -  v^{n}(\cdot - b\dt) \|_{L^2}  +  2 C \dx^{k+2} \\
     & \leq &  \| u^n -  v^{n}\|_{L^2}  +  2 C \dx^{k+2}
\eeno
and the result follows by induction.
\endproof
}\fi


\if{
REMOVE: 

The solution of~\eqref{eq:adv} satisfies
$$ v(t_{n+1},x)= v(t_n, y_x(-\dt)). $$

Hence, a scheme definition should be:
$\forall n=0,\dots,N-1$, find $u^{n+1}\in V_k$ such that
\be  
  \intt_{\Omega}u^{n+1}(x) \varphi(x) dx 
   &  = &  \intt_{\Omega}u^n(y_x(-\scriptdt))\varphi(x) dx,\quad \forall \varphi \in V_k,
\ee
or, equivalently, $u^{n+1} = \cT_{b,\dt} u^n$ where
$$ \cT_{b,\dt} u := \Pi \big( u(y_{\cdot}(-\dt)) \big). $$
However, the computing procedure for the R.H.S.\ can no more be exact, because
$x\converge u^n(y_x(-\scriptdt))$ is no more a piecewise polynomial.

END REMOVE
}\fi

\if{
\medskip

In the case when $v$ is a regular data, we have the following estimate
\begin{lemma} 
Let $k\geq 0$, $b, v\in \cC^{\ell+1}$ for some $\ell\leq k+2$, with $b,v$ 1-periodic. 
Then
$$ \| \wcT_{b,\dt} v - \cT_{b,\dt} v \|_{L^2} \leq C \dx^{\ell+1},$$
for some constant $C\geq 0$ that depends only of $(\|b^{(i)}\|_{L^\infty},\,\|v^{(i)}\|_{L^\infty})_{1\leq i\leq \ell+1}$.
\end{lemma}
\proof
Let $a(x):= v(y_x(-\dt))$, which is $\cC^{\ell+1}$ regular.
Notice that 
\beno 
 (\cT_{b,\dt} v, \varphi)_{L^2} & = &  (\Pi a, \varphi)_{L^2} \ = \ (a,\varphi)_{L^2} \  =\  \sum_{i=0}^{M-1} \sum_{q=0}^{p_i} \int_{J_{i,q}} a(x) \varphi(x) dx.
\eeno
Furthermore, by the scheme definition, 
\beno 
  (\wcT_{b,\dt} v, \varphi)_{L^2} 
  & = &
  \sum_{i=0}^{M-1} \sum_{q=0}^{p_i} \sum_{\ma=0}^k  
  \tilde w^i_{q,\ma} v\big(y_{\tilde x^i_{q,\ma}}(-\dt)\big)\varphi(\tilde x^i_{q,\ma})\\
  & = &
  \sum_{i=0}^{M-1} \sum_{q=0}^{p_i} \sum_{\ma=0}^k  
  \tilde w^i_{q,\ma} a(\tilde x^i_{q,\ma}) \varphi(\tilde x^i_{q,\ma})
\eeno
We can write $a= P + R$ where $P\in P_\ell$ is the Taylor expansion of $w$ around the middle point of $J_{i,q}$, and the rest $R$ satisfies
$\|R\|_{L^\infty(J_{i,q})}\leq C \dx_{i,q}^{\ell+1}$. 
Then, for any $\varphi\in V_k$, the degree $ d^o(P\varphi)$ is lower or equal to $\ell+k \leq 2k+1$,
and the Gaussian quadrature rule is exact for $P \varphi$ on $J_{i,q}$. The reminding error
is of order 
$
  \int_{J_{i,q}} \|R\|_{L^\infty} |\varphi(x)| dx \leq C \dx_{i,q}^{\ell+1} \|\varphi\|_{L^1(J_{i,q})} 
  \leq C \dx^{\ell+1} \|\varphi\|_{L^1(J_{i,q})}.
$
Summing up these bounds, we find a global error or order 
\beno
  \bigg| (\wcT_{b,\dt} v, \varphi)_{L^2}  - (\cT_{b,\dt} v, \varphi)_{L^2}   \bigg|
  & \leq & 
  \sum_{i=0}^{M-1} \sum_{q=0}^{p_i}  C \dx^{\ell+1} \|\varphi\|_{L^1(J_{i,q})}\\
  & \leq  &  C\dx^{\ell+1} \| \varphi \|_{L^1} \ \leq \  C\dx^{\ell+1} \| \varphi \|_{L^2}.
\eeno
Since both $\wcT_{b,\dt} v$ and $\cT_{b,\dt}$ belongs to $V_k$, 
this proves the desired result (by taking $\varphi:=\wcT_{b,\dt} v - \cT_{b,\dt}$).
\endproof

In particular taking $\ell=k+1$ in the previous Lemma gives an error of order $O(\dx^{k+2})$.

We have also the following estimate:
\begin{lemma}
If $u \in L^\infty$, then 
$$ 
  \| \wcT_{b,\dt} u \|_{L^2} \leq \| u \|_{L^\infty}.
$$
\end{lemma}
\proof
For any $\varphi \in V_k$, $\wcT_{b,\dt} u$ is an element of $V_k$ such that 
\be
  & & \hspace{-2cm}  \intt_\mO (\wcT_{b,\dt} u) (x) \varphi(x) dx
   = \sum_{i=0}^{M-1} \sum_{q=0}^{p_i} \sum_{\ma=0}^k  
  \tilde w^i_{q,\ma} u\big(y_{\tilde x^i_{q,\ma}}(-\scriptdt)\big)\varphi(\tilde x^i_{q,\ma}) \\
  & \leq  & 
    \|u\|_{L^\infty} \sum_{i=0}^{M-1} \sum_{q=0}^{p_i} \sum_{\ma=0}^k  
    \tilde w^i_{q,\ma} |\varphi(\tilde x^i_{q,\ma})|\\
  & \leq  & 
    \|u\|_{L^\infty} 
      \bigg(\sum_{i=0}^{M-1} \sum_{q=0}^{p_i} \sum_{\ma=0}^k  \tilde w^i_{q,\ma} \varphi(\tilde x^i_{q,\ma})^2 \bigg)^\ud
       = \|u\|_{L^\infty}  \|\varphi\|_{L^2}
\ee
where we have made use of the Cauchy-Schwarz inequality, the fact that $\sum_{i,q,\ma} \tilde w^i_{q,\ma}=1$ and that $\varphi^2\in P_{2k}$ so that the 
Gaussian quadrature rule is exact for $\varphi^2$.
Taking $\varphi= \wcT_{b,\dt}u $ gives the result.
\endproof

}\fi

\subsection{Non-constant $b$: preliminary results}

For $u\in V_k$, the following approximation result is central. It controls the error between the desired 
formula~\eqref{eq:LGex1} and
the implementable scheme \eqref{eq:LGex_nonconstant_scheme}.

\begin{proposition}[{\bf Gauss quadrature errors}]\label{prop:2.1}
Let $k\geq 0$ and let $b$ be of class $\cC^{2k+2}$ and $1$-periodic. Then: \\
$(i)$ For all $u\in V_k$,
\be 
  \label{eq:bound_prop21.b}
  &  & \bigg| (u(y_\cdot(-\dt)),\varphi)_G -  (u(y_\cdot(-\dt)),\varphi)  \bigg| 
    \leq  
    C \dt \dx^2 \| u \|_{L^2} \| \varphi \|_{L^2},
  \quad \forall \varphi \in V_{k}.  \nonumber 
\ee
where $C\geq 0$ is a constant.
In particular, we have, in the $L^2$-norm:
\be  
  \wcT_{b,\dt} u^n \equiv  u^{n+1} = \cT_{b,\dt} u^n +  O(\dt \dx^2 \| u^n \|_{L^2}), \quad \forall n\geq 0.
\ee
$(ii)$ 
For all $u\in V_k$, for any $\psi$ in $\cC^{k+1}$, 1-periodic,
\be 
  &  & \hspace{-1cm}
    \bigg| (u(y_\cdot(-\dt)) - \psi(y_\cdot(-\dt)),\varphi)_G -  (u(y_\cdot(-\dt)) - \psi(y_\cdot(-\dt)),\varphi)  \bigg|  \nonumber \\
  &  & 
    \hspace{0cm} \leq  
    C  \dt \dx^2 \| u - \psi \|_{L^2} \| \varphi \|_{L^2}
    + C M_{k+1}(\psi) \dx^{k+1}  \| \varphi \|_{L^2},
  \quad \forall \varphi \in V_{k},
   \label{eq:bound_prop21.2b} 
\ee
where $C\geq 0$ is a constant which depends only of $k$, and 
\be 
  M_{p}(\psi):=\max_{0\leq r\leq p} \| \psi^{(r)}\|_{L^\infty}.
\ee
$(iii)$ 
For any regular $\psi \in \cC^{k+1}$, for any $\varphi\in V_{k}$, 
\be 
  \label{eq:reg_psi}
  (\psi, \varphi)_G  =  (\psi,\varphi) + O(M_{k+1}(\psi) \dx^{k+1} \|\varphi \|_{L^2}).
\ee
$(iv)$
Furthermore, $\exists C\geq 0$, for any $\psi \in \cC^{k+1}$, $1$-periodic,
\be
   \| \wcT_{b,\dt} \psi - \cT_{b,\dt}\psi \|_{L^2} \leq C M_{k+1}(\psi) \dx^{k+1}.
   \label{eq:bound_prop21.4b} 
\ee

\end{proposition}


\begin{rem}
Some assumptions can be weakened, for instance $(i)$ and $(ii)$ are still valid using that $b^{(2k+1)}$ is in $L^\infty$, then in the error bounds 
\eqref{eq:bound_prop21.b} and \eqref{eq:bound_prop21.2b} 
the $\dt \dx^2$ term should be replaced by $\dt\dx$.  However these bounds will be used in Section 3 and the form 
\eqref{eq:bound_prop21.b} and \eqref{eq:bound_prop21.2b} is preferred.
Also, it is possible to prove that the error term $O(M_{k+1}(\psi)\dx^{k+1})$ in $(ii)$, $(iii)$ and $(iv)$ can be improved to
$O(M_{2k+1}(\psi) \dx^{k+2})$ provided that $\psi \in \cC^{2k+1}$.
\end{rem}

\medskip\noindent
{\em Proof of Proposition~\ref{prop:2.1}}\label{app:b_nonconstant_bound}


Notice that 
the estimates of $(i)$ and $(iii)$ are a consequence of $(ii)$ (either by choosing $\psi\equiv 0$ to obtain $(i)$, or by choosing $\dt\equiv 0$
and $u\equiv0 $ to obtain $(iii)$). Then $(iv)$ is deduced from $(iii)$ when applied to the regular function $\psi_1(x):= \psi(y_x(-\dt))$.

The plan is first to prove $(i)$, and then to generalize to $(ii)$. 
Precise estimates for the $(2k+2)$nd derivative of $x\converge u(y_x(-\dt))$ will be needed in order to estimate the error when using a Gaussian quadrature formula.
In the following, we first bound the derivatives of $x\converge y_x(-t)$.

\begin{lemma}\label{lem:23}
Assume that $b\in \cC^{k}$, for some $k\geq 1$, and $1$-periodic. Let $L:=\|b'\|_{L^\infty}$ and let $t\in \R$.
Then $x\converge y \equiv y_x(-t)$ is of class $\cC^{k}$, $1$-periodic, and
\be\label{eq:y_relations}
   \left\{\barr{l} 
   \disp 
   \|y \|_{L^\infty(0,1)}\leq 1+\|b\|_{L^\infty} |t|,\\[0.2cm]
   \|\frac{\partial}{\partial x} y \|_{L^\infty(0,1)} \leq e^{L|t|},\\
   \mbox{and, if $k\geq 2$,}\quad 
     \disp \|\frac{\partial^q  }{\partial x^q} y \|_{L^\infty(0,1)} \leq C\,|t|^{q-1}\,e^{L|t|}, 
        \quad \mbox{$\forall q \in\{2,\dots,k\}$},
   \earr\right. 
\ee
for some constant $C\geq 0$. In particular, all the previous derivatives are
bounded on a fixed time interval $t\in[0,T]$. 
\end{lemma}
\proof
\newcommand{\dxy}{\frac{\partial }{\partial x} y}
\newcommand{\dqxy}[1]{\frac{\partial^{#1} }{\partial x^{#1}} y}
We consider $y$ as a function of the time $t$ and of $x$.
We can assume that $x\in[0,1]$ since we have $y_{k+x}(t) = k + y_x(t)$ for all $k\in\mathbb{Z}$ and $t,x\in\R$.
We denote by $y^{(k)}\equiv \dqxy{k}$ the $k$-th derivative of $y$ with respect to $x$. 

Firstly, $y(t,x)=x+\int_0^t b(y(s,x))ds$ and therefore, for $x\in(0,1)$, $|y(t,x)|\leq 1 + \|b\|_{L^\infty}|t|$.

For $k=1$ and $b\in \cC^1$, we have $\frac{\partial}{\partial t} \dxy = b'(y) \dxy$ and $\dxy(0)=1$, therefore
$|\dxy(t)|= \exp\big(\int_0^t b'(y(s)) ds\big) \leq e^{L|t|}$.

For $k\geq 2$,
we have 
\beno 
  \frac{\partial}{\partial t} y^{(k)} & = &  (b'(y) y^{(1)})^{(k-1)} \\
          & = & b'(y) y^{(k)} +  \sum_{\ell=1}^{k-1} C^\ell_{k-1} (b'(y))^{(\ell)} y^{(k-\ell)}.
\eeno
Then we use a recursion argument for $\ell=1,\dots, k$.
Let us assume that the spatial derivatives $y^{(\ell)}$ are bounded for $1\leq\ell\leq k-1$, with
$\|y^{(\ell)}\|_{L^\infty(0,1)}\leq C_\ell |t|^{\ell-1} e^{L |t|}$.
Then 
for $k\geq 2$, the function
$f : = \sum_{\ell=1}^{k-1} C^\ell_{k-1} (b'(y))^{(\ell)} y^{(k-\ell)}$ is bounded, with a bound of the form
$\|f(.,t)\|_{L^\infty(0,1)}\leq C |t|^{k-2} e^{L |t|}$, for some constant $C$.
By using the formula, for a given and fixed $x$,
$$
  y^{(k)}(t)= e^{\int_0^t b'(y(s)) ds} y^{(k)}(0) +  \int_0^t e^{\int_s^t b'(y(\mt)) d\mt} f(s) ds,
$$
the fact that $y^{(k)}(0)=0$ for $k\geq 2$ and for $s\in[0,t]$ (or $s\in[t,0]$ if $t\leq 0$):
\beno
  |e^{\int_s^t b'(y(\mt)) d\mt} f(s)| 
    & \leq &  C e^{L|t-s|}\, |s|^{k-2} e^{L |s|} \\
    & \leq & C |t|^{k-2}\, e^{L |t|}
\eeno
we conclude that
$|y^{(k)}(t)|\leq C |t|^{k-1}\, e^{L|t|}$.
\endproof

\begin{lemma}\label{lem:estim_2}
Assume $q\geq k+1$, and $u\in V_k$. On any interval $J$ where $u$ is regular, 
$$ 
  \| \frac{d^{q}}{dx^{q}}(u(y)) \|_{L^\infty(J)} \leq C  \dt \sum_{p=1}^k \| u^{(p)} \|_{L^\infty(y(J))}.
$$
\end{lemma}
\proof 
We first recall 
an expression for the $q$-th derivative of the composite function $u(y)$, also
known as "Fa{\`a} di Bruno's formula"~\cite{FaaDiBruno}:
\be\label{eq:FaadiBruno}
 \frac{1}{q!}\frac{d^q}{dx^q} (u(y(x)))= 
  \sum_{p=1}^{k} 
  u^{(p)}(y(x)) \bigg(
  \sum_{(\ma_j), \ \sum_j\ma_j=p,\ \sum_{j}j\ma_j=q} 
  \frac{ (y^{(1)}/1!)^{\ma_1} \cdots (y^{(q)}/q!)^{\ma_q}}{\ma_1! \cdots \ma_q!} 
  \bigg).
\ee
Here the sum is limited to $p\leq k$ (instead of $p\leq q$) since $u\in V_k$.
\FULL{
A brief proof is given in appendix~\ref{app:FaadiBruno}.
}

Therefore, together with Lemma~\ref{lem:23}, we obtain the bound
$$ 
  \|\frac{d^q}{dx^q} (u(y)) \|_{L^\infty(J)} \leq
  C \sum_{p=1}^{k} 
    \| u^{(p)} \|_{L^\infty(y(J))} 
    \bigg( \sum_{(\ma_j), \ \sum_{j=1}^q\ma_j=p,\ \sum_{j=1}^q j\ma_j=q} \dt^{\ma_2+\dots+\ma_q} \bigg).
$$
The case when $\ma_2=\dots=\ma_q=0$ happens only if $\ma_1=p=q$. Since $q\geq k+1$, and $p\leq k$, this case never occurs.
Therefore, the power of $\dt$ is at least $1$, which concludes the proof.
\endproof


\proof[Proof of Proposition~\ref{prop:2.1}$(i)$:]
Let $\eps$ be the error term, defined by 
\beno
  \eps:= \int_0^1 u(y_x(-\dt)) \varphi(x) dx  -
  \sum_{i=0}^{M-1} \sum_{q=0}^{p_i}\sum_{\ma=0}^{k} \tilde w^i_{q,\ma} 
   u(y_{\tilde x^i_{q,\ma}}(-\dt)) \varphi(\tilde x^i_{q,\ma}).
\eeno
We have $\eps=\sum_i \sum_{q=0}^{p_i} \eps_{i,q}$ where
\be\label{eq:gauss_iq}
  \eps_{i,q}:= 
  \int_{J_{i,q}}  u\big(y_{x}(-\scriptdt)\big)\varphi(x)\, dx
   - \sum_{\ma=0}^k \tilde w^i_{q,\ma} u\big(y_{\tilde x^i_{q,\ma}}(-\scriptdt)\big)\varphi(\tilde x^i_{q,\ma})
\ee
and with $J_{i,q}:=(x_{i,q},x_{i,q+1})$.

Let $u(y)$ be the function $x\converge u(y_x(-\dt))$.
Since $u(y)$ is $\cC^{2k+2}$ regular on $J_{i,q}$ for each fixed $i$, $q\in\{0,\dots,p_i\}$, 
and that the R.H.S.\ of \eqref{eq:gauss_iq} corresponds to the
Gaussian quadrature rule on $J_{i,q}$, then
we have in particular
$$ |\eps_{i,q}| \leq C  \dx_{i,q}^{2k+3}\, \| [u(y) \varphi]^{(2k+2)}\|_{L^{\infty}(J_{i,q})}, 
$$
where $\dx_{i,q}:=x_{i,q+1}-x_{i,q}$.

On the other hand, since $\varphi\in V_{k}$, 
$$
  \|[u(y)\varphi]^{(2k+2)}\|_{L^\infty(J_{i,q})}
  \leq C \sum_{r=0}^{k} \| \varphi^{(r)} \|_{L^\infty(J_{i,q})} \| [u(y)]^{(2k+2-r)} \|_{L^\infty(J_{i,q})}.
$$
For all $r\in\{0,\dots,k\}$ we have $2k+2-r\geq k+2 \geq k+1$, hence we can use Lemma~\ref{lem:estim_2}
and obtain the bound
$$
  \|[u(y)\varphi]^{(2k+2)}\|_{L^\infty(J_{i,q})}
  \leq C \big(\sum_{r=0}^{k} \| \varphi^{(r)} \|_{L^\infty(J_{i,q})}\big) \dt \big(\sum_{p=1}^k\| u^{(p)} \|_{L^\infty(y(J_{i,q}))}\big).
$$
In particular, 
$$
  \sum_{i,q} |\eps_{i,q}|
    \leq   C \sum_{r=0}^{k} \sum_{p=1}^k
    \sum_{i}\sum_{q=0}^{p_i} \dt \dx_{i,q}^{2k+3} \| \varphi^{(r)} \|_{L^\infty(J_{i,q})}  \| u^{(p)} \|_{L^\infty(y(J_{i,q}))}
$$

By a scaling argument \cite{ciarletbook-78,lesaint-raviart-74},
and using that $\varphi \in V_{k}$ for fixed $k$, we have,  $\forall 0\leq r\leq k$,
\be \label{eq:varphi_boundtok1}
    \| \varphi^{(r)} \|_{L^\infty(J_{i,q})} \leq \frac{C}{\dx_{i,q}^{r+1/2}}  \|\varphi \|_{L^2(J_{i,q})}
    \leq \frac{C}{\dx_{i,q}^{k+1/2}}  \|\varphi \|_{L^2(J_{i,q})},
\ee
for some constant $C$, assuming also $\dx_{i,q}\leq 1$
(the idea is to use the fact that for polynomials of degree $k$, by using norm equivalences, 
$\|\varphi^{(r)}\|_{L^\infty(0,1)}\leq C \|\varphi\|_{L^2(0,1)}$ for some constant $C$ independent of $\varphi$,
and then to use a scaling argument from 
$(0,1)$ to $J_{i,q}$ to obtain the desired inequality). 

Denoting by $|J|$ the length of any interval $J$, we have also 
$$|J_{i,q}| e^{-L\dt} \leq |y(J_{i,q})| \leq |J_{i,q}| e^{L\dt}, \quad L:=\|b'\|_{L^\infty},$$
where $|J_{i,q}|=\dx_{i,q}$.  
Hence, for $r\leq k$ and $p\leq k$,
\beno
  \dx_{i,q}^{2k+3} 
  \sum_{i,q} \| \varphi^{(r)} \|_{L^\infty(J_{i,q})}  \| u^{(p)} \|_{L^\infty(y(J_{i,q}))}
  & \leq  &
   C \dx_{i,q}^{2k+3} 
     \sum_{i,q}
     \frac{\| \varphi \|_{L^2(J_{i,q})}}{\dx_{i,q}^{r+1/2}} 
     \frac{\| u \|_{L^2(y(J_{i,q}))}}{|y(J_{i,q})|^{p+1/2}}\\
  & \leq  &
   C \dx_{i,q}^2 \sum_{i,q} \| \varphi \|_{L^2(J_{i,q})} \| u \|_{L^2(y(J_{i,q}))}.
\eeno
Finally, by the Cauchy-Schwarz inequality, 
\beno
  \sum_{i,q} \| \varphi \|_{L^2(J_{i,q})} \| u \|_{L^2(y(J_{i,q}))}  
    & \leq &  \bigg(\sum_{i,q} \| \varphi \|^2_{L^2(J_{i,q})} \bigg)^{1/2} 
              \bigg(\sum_{i,q} \| u \|^2_{L^2(y(J_{i,q}))} \bigg)^{1/2} \\
    & \leq &  \| \varphi \|_{L^2}  \|u\|_{L^2}.
\eeno
since $\bigcup_{i,q} J_{i,q}$ is a covering of $[0,1]$.
Hence we obtain 
$$ \sum_{i,q} |\eps_{i,q}| \leq  C \dt \dx^2 \| \varphi \|_{L^2}  \|u\|_{L^2},
$$
which concludes the proof of $(i)$. 


\proof[Proof of Proposition~\ref{prop:2.1}$(ii)$:]
Let us write $\psi= P + R$ where $P\in V_k$ is defined as the Taylor expansion of $\psi$ on each $J_{i,q}=(x_{i,q},x_{i,q+1})$, around $x_{i,q}$.
We consider the decomposition
\be \label{eq:eqa}
  u(y_{\cdot}(\scriptmdt)) - \psi(y_{\cdot}(\scriptmdt)) \equiv (u-P)(y_{\cdot}(\scriptmdt)) - R (y_{\cdot}(\scriptmdt)) 
\ee
Then by Proposition \ref{prop:2.1}$(i)$, for any $\varphi\in V_{k}$,
$$ | ((u-P)(y_\cdot(\scriptmdt)),\varphi)_G - ((u-P)(y_\cdot(\scriptmdt)),\varphi) | \leq C \dt \dx^2 \| u - P \|_{L^2} \|\varphi\|_{L^2}.$$
Using the fact that $\|R\|_{L^2} \leq C \|R\|_{L^\infty} \leq C M_{k+1}(\psi) \dx^{k+1}$, we obtain the bound
\be
  &  & \hspace{-2cm} | ((u-P)(y_\cdot(\scriptmdt)),\varphi)_G - ((u-P)(y_\cdot(\scriptmdt)),\varphi) |  \nonumber \\
  &  & \leq \   C \dt \dx^2 \| u - \psi \|_{L^2} \|\varphi\|_{L^2} + C M_{k+1}(\psi) \dt \dx^{k+3} \|\varphi\|_{L^2}. \label{eq:eqb}
\ee
There remains to bound the error 
$$ (R(y_{\cdot}(\scriptmdt)),\varphi)_G -  (R(y_{\cdot}(\scriptmdt)),\varphi). $$
This is easily bounded by $C \|R\|_\infty \|\varphi\|_{L^2} = O(\dx^{k+1}\|\varphi\|_{L^2})$.
Combined with \eqref{eq:eqa} and \eqref{eq:eqb}, we obtain the desired bound.
\endproof

\if{
Let us first establish that if $Q\in \cC^{k+1}(J_{i,q})$ (for all $i,q$), then, for all $\varphi\in V_{k}$, 
\be
  \label{eq:result}
   (Q,\varphi)_G - (Q,\varphi) = O( \dx\ M_{k+1}(Q)\ \|\varphi\|_{L^2} ).
\ee
In order to prove  \eqref{eq:result}, let $\eps_{i,q}(Q\varphi)$ be the error of the Gaussian quadrature formula on $J_{i,q}$ for the function $Q\varphi$. 
We have, using a Taylor expansion of order $k$:
\beno | \eps_{i,q}(Q) | 
  &  \leq &  C \dx_{i,q}^{k+2} \| [Q\varphi]^{(k+1)} \|_{L^\infty(J_{i,q})}  \\
  &  \leq &  C \dx_{i,q}^{k+2} \sum_{p\leq k} \| Q^{(k+1-p)} \|_{L^\infty(J_{i,q})}   \| \varphi^{(p)} \|_{L^\infty(J_{i,q})}  \\
  &  \leq &  C \dx_{i,q}^{k+2} M_{k+1}(Q)  \frac{ \| \varphi \|_{L^2(J_{i,q})}}{\dx_{i,q}^{k+1/2}}
   \leq   C \dx_{i,q}^{3/2} M_{k+1}(Q)  \| \varphi \|_{L^2(J_{i,q})}
\eeno
for some constant $C\geq 0$. We conclude using $\sum_{i,q} \dx_{i,q}^{1/2} \|\varphi\|_{L^2(J_{i,q})} \leq \|\varphi\|_{L^2}$.

Then we choose $Q= R(y_{\cdot}(\scriptmdt))$ in \eqref{eq:result}. By using the $\cC^{k+1}$ regularity of $x\converge y_x(\scriptmdt)$, 
we have $M_{k+1}(Q)\leq  M_{k+1}(R)$.
Denoting $\bar x = x_{i,q}$ and for $x\in J_{i,q}$, recall that 
$R(x) = \int_0^1 \frac{(1-t)^{k}}{k!} \psi^{(k+1)}(\bar x + t (x-\bar x))\,dx\ (x-\bar x)^{k+1}$.
In particular, 
$$R^{(p)}(x) = \int_0^1 \frac{(1-t)^{k}}{k!} t^p \psi^{(k+1+p)}(\bar x + t (x-\bar x))\,dx\ (x-\bar x)^{k+1}.$$
Therefore it holds $\|R^{(p)}\|_{L^\infty(J_{i,q})} \leq C \dx^{k+1} \|\psi^{(k+1+p)}\|_{L^\infty(J_{i,q})}$ for all $p\leq k+1$, and finally
$M_{k+1}(Q)\leq C M_{2k+2}(\psi) \dx^{k+1}$. Together with \eqref{eq:result}, we conclude to
\be \label{eq:eqc}
  \big| (R(y_{\cdot}(\scriptmdt)),\varphi)_G -  (R(y_{\cdot}(\scriptmdt)),\varphi) \big| \leq C \dx^{k+2} M_{2k+2}(\psi) \|\varphi\|_{L^2}, 
   \quad \forall \varphi\in V_{k}
\ee
Combining \eqref{eq:eqa}, \eqref{eq:eqb} and \eqref{eq:eqc}, we obtain the desired bound.
}\fi


\subsection{Non-constant $b$: stability and error analysis}

We now turn on the stability and convergence analysis. 
The following result shows the \textit{unconditional} stability of the scheme, for any $k\geq 1$.

\begin{proposition}[Stability] \label{prop:stab_b_nonconstant}
Let $k\geq 0$ and let $b$ be Lipschitz continuous and $1$-periodic. Then:\\ 
$(i)$ for any $u\in L^2$, and $\tilde u(x):= u(y_x(-t))$, it holds:
\be \label{eq:cv}
   \|\tilde u\|_{L^2} \leq e^{\ud L|t|}\  \| u \|_{L^2},
   \quad \mbox{where $L:= \|b'\|_{L^\infty}$.}
\ee
$(ii)$ If furthermore $b$ is of class $\cC^{2k+2}$,
there exists a constant $C_1\geq 0$ such that, $\forall u\in V_k$, 
$$
   \|\wcT_{b,\dt} u \|_{L^2} \leq e^{C_1 \dt } \| u \|_{L^2} \quad \forall u\in V_k.
$$
$(iii)$
In particular for the scheme $u^{n+1}= \wcT_{b,\dt} u^n$, 
$$
   \|u^{n}\|_{L^2} \leq e^{C_1 t_n} \| u^0 \|_{L^2}, \quad \forall n \geq 0,
$$
where $t_n = n \dt$.
\end{proposition}
\begin{proof}
$(i)$
We make use of the change of variable $x\converge z:=y_{x}(-t)$, with periodic boundary conditions 
for the integrands.
Therefore we have $x=y_z(t)$ and
$$
   \frac{\partial x}{\partial z} (t) = \exp\bigg( \int_0^t b'(y_z(s)) ds \bigg) \leq e^{L|t|}.
$$
We then obtain
\FULL{
\beno
   \| u(y_{\cdot}(-t)) \|_{L^2}^2 
   & = & \int_\mO |u (y_x(-t))|^2 dx \nonumber \\
   & = & \int_\mO |u(z)|^2 \left|\frac{\partial x}{\partial z}(t)\right| dz  \nonumber\\
   & \leq &  e^{L|t|} \int_\mO |u(z)|^2 dz. 
\eeno
}
\SHORT{
\beno
 \int_\mO |u (y_x(-t))|^2 dx  = \int_\mO |u(z)|^2 \left|\frac{\partial x}{\partial z}(t)\right| dz 
    \leq   e^{L|t|} \int_\mO |u(z)|^2 dz. 
\eeno
}
$(ii)$
By using \eqref{eq:bound_prop21.b}, we have
\be\label{eq:fbound}
    \|\wcT_{b,\dt} u \|_{L^2} &\leq& \| u(y_{\cdot}(-\scriptdt)) \|_{L^2}  + C \dt \dx^2 \|u\|_{L^2}. 
\ee
Together with \eqref{eq:cv} we get a stability constant
$$e^{\frac L 2 \dt} + C \dt \dx^2 \leq e^{\frac L 2 \dt}(1 + C \dt \dx^2)
   \leq e^{\frac L 2 \dt} e^{C \dt \dx^2},$$
hence the desired result for any $C_1\geq 0$ such that $C_1\geq \ud L + C \dx^2$.
\end{proof}

We now state a first convergence result.
It generalizes the error estimate of Theorem~\ref{th:th0}
established in the case when $b$ is constant, to the non-constant case.
\begin{theorem}[Convergence]\label{th:advect_conv}
Let $k\geq 0$. Assume the initial condition  $v_0$ is $1$-periodic and of class $\cC^{k+1}$.
Let $b$ be $1$-periodic and of class $\cC^{2k+2}$. 
There exist constants $C_1\geq 0$, $C\geq 0$ such that
\be \label{eq:th1bound}
  \| u^{n} - v^{n} \|_{L^2} \leq e^{C_1 T} \bigg( \|v^0-u^0\|_{L^2}  +   C T \frac{\dx^{k+1}}{\dt} \bigg),
  \quad \forall n\leq N.
\ee
\end{theorem}

\proof[Proof of Theorem~\ref{th:advect_conv}]
By using the regularity of $v^{n+1}$ and Proposition~\ref{prop:2.1}$(iv)$ we have
\be
  \Pi v^{n+1} = \cT_{b,\dt} v^n = \widetilde \cT_{b,\dt} v^{n} + O (\dx^{k+1}). 
\ee
Because of the projection error $ \| v^{n+1} - \Pi v^{n+1} \| = O(\dx^{k+1})$, then we obtain
the following consistency estimate:
\be
   v^{n+1} = \widetilde \cT_{b,\dt} v^{n} + O (\dx^{k+1}). 
\ee
Therefore
\be
  u^{n+1} - v^{n+1} & = & \widetilde \cT_{b,\dt} (u^n - v^n)  + O(\dx^{k+1}).
\ee
By the stability bound of Proposition~\ref{prop:stab_b_nonconstant}$(ii)$,
\beno
  \| u^{n+1}- v^{n+1} \|_{L^2}
     \leq e^{C_1\dt} \|u^n - v^n \|_{L^2} + C \dx^{k+1}. 
\eeno
We conclude by induction.
\endproof

\if{
Let $\tilde v^{n+1} := \Pi_{V_{k+1}} v^{n+1}$, 
which therefore satisfies 
\be \label{eq:vk1} 
  \| v^{n+1} - \tilde v^{n+1} \| = O(\dx^{k+2}),
\ee
and let 
$$ {\tilde e}^{n+1} : = u^{n+1} - \tilde v^{n+1}. $$
Then it holds, since $\tilde e^{n+1} \in V_{k+1}$ and $v^{n+1} \equiv v^n(y_\cdot(-\scriptdt))$:
\be
  \| \tilde e^{n+1} \|_{L^2}^2  
  & = & ( u^{n+1},\ \tilde e^{n+1}) -  (\tilde v^{n+1},\ \tilde e^{n+1})  \nonumber \\
  & = & ( u^n(y_\cdot(-\scriptdt)), \tilde e^{n+1})_G  - (\tilde v^{n+1},\ \tilde e^{n+1})  \nonumber \\
  & = & ( u^n(y_\cdot(-\scriptdt))- v^{n}(y_\cdot(-\scriptdt)), \tilde e^{n+1})_G  \nonumber  \\
  &   &    \hspace{2cm} + (v^{n+1}, \tilde e^{n+1})_G - (v^{n+1},  \tilde e^{n+1})  \nonumber \\
  &   &    \hspace{2cm} + (v^{n+1} - \tilde v^{n+1},\ \tilde e^{n+1}).  
       \label{eq:pro1a}  
\ee
By using the regularity of $v^{n+1}$ and Proposition~\ref{prop:2.1}$(iii)$ we have
\be
  |(v^{n+1}, \tilde e^{n+1})_G - (v^{n+1},  \tilde e^{n+1})|  \leq C  \dx^{k+2} \| \tilde e^{n+1} \|_{L^2}.
\ee
Then, we have also
\be
   (v^{n+1} - \tilde v^{n+1},\ \tilde e^{n+1})  \leq 
    \|v^{n+1} - \tilde v^{n+1}\|_{L^2} \| \tilde e^{n+1} \|_{L^2} \leq C \dx^{k+2} \| \tilde e^{n+1} \|_{L^2}.
\ee
Finally, by Proposition~\ref{prop:2.1}$(ii)$ we have
\beno
 & & \hspace{-1cm} 
  ( u^n(y_\cdot(-\scriptdt))- v^{n}(y_\cdot(-\scriptdt)), \tilde e^{n+1})_G    \\
 & & \ = \
   ( u^n(y_\cdot(-\scriptdt))- v^{n}(y_\cdot(-\scriptdt)), \tilde e^{n+1})  \\
 & & \ \ \
   + O(\dt \dx \| u^n - v^n \|_{L^2} \|\tilde e^{n+1} \|_{L^2}) 
   + O( \dx^{k+2} \| \tilde e^{n+1} \|_{L^2} ).
\eeno
Therefore, using the Cauchy-Schwarz inequality and previous bounds,
\beno
  \| u^{n+1}-\tilde v^{n+1} \|_{L^2}
     &  \leq &  \| u^{n}(y_\cdot (\scriptmdt)) -v^{n}(y_\cdot(\scriptmdt)) \|_{L^2}  + C_0\dt \dx \|u^n - v^n \|_{L^2} + C \dx^{k+2} 
\eeno
By the stability bound of Proposition~\ref{prop:stab_b_nonconstant}$(i)$,
\beno
  \| u^{n+1}-\tilde v^{n+1} \|_{L^2}
     &  \leq & ( e^{C_1\dt} + C_0 \dt\dx ) \|u^n - v^n \|_{L^2} + C \dx^{k+2}. 
\eeno
Using again~\eqref{eq:vk1}, and that $e^{C_1 \dt} + C_0\dt\dx \leq e^{(C_1  +  C_0 \dx)\dt} \leq e^{L_1\dt}$, we obtain
\beno
  \| u^{n+1}- v^{n+1} \|_{L^2}
     &  \leq &  e^{L_1\dt}  \|u^n - v^n \|_{L^2} + C \dx^{k+2}. 
\eeno
We conclude by induction.
}\fi

\subsection{Stability to perturbations}
We conclude by a stability result with respect to the error of the position of the characteristics.
\MODIF{
\begin{proposition}\label{prop:perturb}
Let $w_1(x):=y_x(-\dt)$ and $w_2(x):=\bar y_x(-\dt)$ be some approximation of $y_x(-\dt)$ such that 
$\max\limits_{i=1,2}|w_i(x)-x|\leq c_0\dt$ for some constant $c_0>0$.
Assume that $\frac{\dt}{\dx}\leq K$ for some constant $K>0$.
Then for all $u,\varphi \in V_k$, it holds
\be\label{eq:perturb}
  \bigg|  \int_0^1 u(w_2(x)) \varphi(x) dx -   \int_0^1 u(w_1(x)) \varphi(x) dx \bigg| \leq  
    C \frac{\|w_2-w_1\|_{L^\infty}}{\dx}\ \|u\|_{L^2}\|\varphi \|_{L^2}
\ee
for some constant $C\geq 0$ independent of $\dt,\dx$.
\end{proposition}
\proof
We first notice that $|y_x(-\dt)-x|\leq c_0\dt\leq c_0\frac{\dt}{\dx} \dx \leq q \dx$ for some integer $q\geq 1$, 
as well as $|\bar y_x(-\dt)-x|\leq q\dx$.
For a given interval $I$, let $I_q:=I+[-q,q]\dx$. It holds:
\beno 
  \|u(w_2)-u(w_1)\|_{L^2(I)} & \leq  & \|u'\|_{L^\infty(I_q)} \|w_2-w_1\|_{L^\infty} \dx^{1/2} \\
    & \leq &  c_1\frac{\|u\|_{L^2(I_q)}}{\dx^{3/2}}\|w_2-w_1\|_{L^\infty} \dx^{1/2} \leq 
     c_1\|u\|_{L^2(I_q)} \frac{\|w_2-w_1\|_{L^\infty}}{\dx}
\eeno
for some constant $c_1>0$ (we have used a scaling argument as before). 
We remark that $\|u\|^2_{L^2(I_q)}= \sum_{j=-q,\dots,q} \|u\|^2_{L^2(I+q\dx)}$ where $J=I+q\dx$ is also another interval of same length as $I$.
Hence  $\sum_I \|u\|^2_{L^2(I_q)}= (2q+1) \|u\|^2_{L^2}$, and 
\beno
  \|u(w_2)-u(w_1)\|_{L^2} & \leq  & c_1 \sqrt{2q+1} \|u\|_{L^2} \frac{\|w_2-w_1\|_{L^\infty}}{\dx}.
\eeno
The result \eqref{eq:perturb} follows by using a Cauchy-Schwarz inequality.
\endproof

\begin{corol}\label{cor:perturb}
We consider that an error is made in the computation of the characteristic $y_x(-\dt)$, such that 
\be\label{eq:epserror}
  |\bar y_x(-\dt)-y_x(-\dt) |\leq \eps
\ee
for some constant $C\geq 0$ and $\eps>0$.
Then the error estimate of order $CT \frac{\dx^{k+1}}{\dt}$ in Theorem~\ref{th:advect_conv}
must be replaced by 
\beno
  CT\frac{\dx^{k+1}}{\dt} + CT \frac{\eps}{\dx\dt} 
\eeno
\end{corol}
\begin{proof}[Sketch of proof.]
At each time step an error of order $\eps=\|w_2-w_1\|_{L^\infty}$
is made in the computation of the characteristics. By the previous Lemma this results
in a supplementary error term of order $\frac{\eps}{\dx}$.
Hence after $N=\frac{T}{\dt}$ time steps the error coming from  the computations of the integrals will be bounded by $T O(\frac{\eps}{\dx\dt})$.
\end{proof}

We remark that in practice, this approximation error is not seen in the numerical tests 
because the characteristics are computed using an analytical formula
or a machine precision fixed point method when needed. 
A high-order approximation method would also lead to $\eps:= C \dt^{q+1}$ in \eqref{eq:epserror} 
which can be made arbitrarily small
in particular because we deal only with one-dimensional approximations of characteristics in the proposed method.
}

\section{Second-order PDEs}
\newcommand{\Sohd}{D^0_{h^{1/2}}}
\newcommand{\Soht}{D^0_{h^{1/3}}}
\newcommand{\Sotd}{D^0_{\dt^{1/2}}}
\newcommand{\Sott}{D^0_{\dt^{1/3}}}
\newcommand{\Dhd}{D_{h^{1/2}}}
\newcommand{\Dht}{D_{h^{1/3}}}
\newcommand{\Dtd}{D_{\dt^{1/2}}}
\newcommand{\Dtt}{D_{\dt^{1/3}}}

This section deals with SLDG schemes for second-order PDEs.
We will first deal with a simple diffusion problem with constant coefficients, for which specific schemes can be obtained,
and then we consider the more general case of advection - diffusion problems with variable coefficients.

\subsection{Case of a diffusion equation with constant coefficient}
We first consider a diffusion equation with a constant coefficient $\sigma \in\R$:
\be
  & & v_t  - \frac{\sigma^2}{2} v_{xx} = 0, \quad x\in \mO,\ t\in (0,T), \label{eq:diffusion}\\
  & & v(0,x) = v_0(x), \quad x \in \Omega,
\ee
and aim to construct simple schemes in this particular setting. 
Following Kushner and Dupuis \cite{kus-dup-92}, a first scheme, in semi-discrete form, is
\be \label{eq:semi_d}
  u^{n+1}(x) =  \frac{1}{2} \bigg( u^n(x - \sigma \sqrt{\dt}) + u^n(x + \sigma \sqrt{\dt}) \bigg) \equiv S^0_\dt u^n\,(x).
\ee
It is easy to see that, taking $v^n(x):=v(t_n,x)$ where $v$ is the solution of \eqref{eq:diffusion} and is assumed sufficiently regular,
the following consistency error estimate holds:
$$
    \| \frac{v^{n+1} - S^0_\dt v^n}{\dt} \|_{L^2} = O(\dt).
$$
The basic SLDG scheme (also called hereafter SLDG-1) is based on the weak formulation of \eqref{eq:semi_d}.

\noindent
\mbox{\underline{SLDG-1 scheme:}} 
Define recursively $u^{n+1}$ in $V_k$ such that
\beno 
 & &  \int u^{n+1}(x)\varphi(x) dx = 
   \int  \frac{1}{2} \bigg( u^n(x-\sigma\sdt) + u^n(x+\sigma\sdt) \bigg) \varphi(x)\, dx,\quad \forall \varphi\in V_k.
\eeno
(The initialization of $u^0$ is done as before).
The scheme will be also written in abstract form as follows:
$$
   u^{n+1} = \cS_\dt(u^n),
$$
where
$$ 
  \cS_\dt :=  \Pi \cS^0_\dt \equiv \frac{1}{2} \bigg(\cT_{-\sigma\sdt}  + \cT_{\sigma\sdt} \bigg).
$$


Before doing the numerical analysis, our aim is first to improve 
the accuracy with respect to the time discretization.
The technique proposed here is to use a convex combination of $u$, $S_\dt$, $S_\dt S_\dt$, \dots
It will work only for the constant coefficient case ($\ms$ constant).

Using Taylor expansions, for $u$ sufficiently regular, we have, for $\dt$ small,
\be
  S^0_\dt u  
     & = & u + \frac{\ms^2}{2} u_{xx}\dt + \frac{\ms^4}{24} u^{(4)}_{x}\dt^2 + O(\dt^3), \label{eq:dev_S0u}\\
  S^0_\dt S^0_\dt u 
     & = & u + \ms^2 u_{xx}  \dt         + \frac{\ms^4}{3} u^{(4)}_{x}\dt^2  + O(\dt^3), \label{eq:dev_S0S0u}
\ee
where $u_x^{(q)}$ denotes the $q$-th derivative of $u$ w.r.t. $x$.

\medskip
On the other hand, if $v^n= v(t_n,x)$ where $v$ is the exact solution of $v_t = \frac{\ms^2}{2} v_{xx}$,
we have
\be
  v^{n+1} & = & v^n +   v_t \dt +  \frac{1}{2} v_{tt}\dt^2 + O(\dt^3) \\
          & = & v^n +  \frac{\ms^2}{2} v^n_{xx}\dt +  \frac{\ms^4}{8} v^{n,(4)}\dt^2 + O(\dt^3).
  \label{eq:vnplusun}
\ee

Now, looking for coefficients $a,b,c$  such that
$ a v^n + b S^0_\dt v^n + c S^0_\dt S^0_\dt v^n$ is equal to $v^{n+1}$ up to $O(\dt^3)$, 
using \eqref{eq:dev_S0u} and \eqref{eq:dev_S0S0u},
we obtain the system
\begin{eqnarray}
\left\{\begin{array}{rrrrrrr}  
    a & + & b & + &  c & = & 1 \\
      &   & \frac{b}{2} & + & c & = & \frac{1}{2} \\
      &   & \frac{b}{24} & + & \frac{c}{3}  & = & \frac{1}{8}
\end{array}
\right.
\end{eqnarray}
and we find that $a=b=c=\frac{1}{3}$.
Therefore, a second-order scheme (for constant coefficient) is now given by\\
\mbox{\underline{SLDG-2 scheme:}}
\be 
  & & \hspace{-3cm} u^{n+1} = S_\dt^{2} u^n :=
   \frac{1}{3} (u^n + S_{\dt} u^n  + S_{\dt} S_{\dt} u^n).
    \label{eq:SLDG-2const}
\ee

\begin{rem}\label{rem:3.1}
A variant of this scheme can be
\be \label{eq:SLDG-RK2variant}
  u^{n+1} = \Pi \frac{1}{3} \bigg(u^n + S^0_{\dt} u^n  +  S^0_{\dt} S^0_{\dt} u^n\bigg).
\ee
This is in general slightly different from \eqref{eq:SLDG-2const} because $S_\dt S_\dt = \Pi S^0_\dt \Pi S^0_\dt$
may differ from $\Pi S^0_\dt S^0_\dt$. Nevertheless, the difference between the two will be of the order of the projection 
error $O(\dx^{k+1})$ when applied to a regular data.
\end{rem}

In order to obtain a third-order scheme, we can proceed in a similar way.
First, we obtain the following expansions:
\beno
  S^0_\dt u       & = & u + \frac{\ms^2}{2} u_{xx}\dt + \frac{\ms^4}{24} u^{(4)}_{x}\dt^2 + 
   \frac{\ms^6}{6!} u^{(6)}_{x} \dt^3 + O(\dt^4),\\
  S^0_\dt S^0_\dt u & = & u +  \ms^2 u_{xx}\dt   + \frac{\ms^4}{3} u^{(4)}_{x} \dt^2 +   
    \frac{2}{45} \sigma^6 u^{(6)}_{x} \dt^3 + O(\dt^4),\\
  S^0_\dt S^0_\dt S^0_\dt u &  = &  
    u +  \frac{3}{2} \ms^2 u_{xx}\dt + \frac{7}{8} \ms^4 u^{(4)}_{x} \dt^2
      + \frac{61}{240}  \ms^6 u^{(6)}_{x}\dt^3  + O(\dt^4).
\eeno
Looking for coefficients $a,b,c,d$  such that 
$ a v^n + S^0_\dt v^n + S^0_\dt S^0_\dt v^n + S^0_\dt S^0_\dt S^0_\dt v^n $ 
is equal to $v^{n+1}$ up to $O(\dt^4)$, we find the system
\begin{eqnarray}
\left\{\begin{array}{rrrrrrrrr}  
    a & + & b & + &  c & + & d  = & 1 \\
      &   & \frac{b}{2} & + & c & + &  \frac{3}{2} d = & \frac{1}{2} \\
      &   & \frac{b}{24} & + & \frac{c}{3}  & + & \frac{7}{8} d  = & \frac{1}{8} \\
      &   & \frac{b}{6!} & + & \frac{2}{45} c  & + & \frac{61}{240}d  = & \frac{1}{48}
\end{array}
\right.
\end{eqnarray}
and its solution 
$$ 
   (a,b,c,d):= \frac{1}{45} (13, 21, 9 ,2).
$$
Thus, the following scheme is of $3$rd-order in time:\\
\mbox{\underline{SLDG-3 scheme:}} 
\be  
  & & \hspace*{-1cm} u^{n+1} = 
  S_\dt^{3} u^n :=
  \frac{13}{45} u^n + \frac{7}{15} S_\dt u^n +  \frac{1}{5} S_\dt S_\dt u^n + \frac{2}{45} S_\dt S_\dt S_\dt u^n. 
  \nonumber 
\ee

As in Remark~\ref{rem:3.1}, a variant of the scheme can be 
\be
  u^{n+1} = \Pi \bigg(\frac{13}{45} u^n + \frac{7}{15} S^0_\dt u^n +  \frac{1}{5} S^0_\dt S^0_\dt u^n + \frac{2}{45} S^0_\dt S^0_\dt S^0_\dt u^n\bigg). 
  \label{eq:SLDG-RK3variant} 
\ee

Since we are using a convex combination of stable schemes 
($S_\dt$, $S_\dt S_\dt$ or $S_\dt S_\dt S_\dt$), the schemes SLDG-1, SLDG-2 and
SLDG-3 are all stable in the $L^2$ norm.

\begin{rem}
Up to 5th-order schemes - in time - can also be obtained (see~\cite{bok-bon-2013}), using convex combinations of the form 
$u^{n+1} =\sum_{i=0}^p a_i (S^0_\dt)^i u^n$.
\end{rem}

We now state a convergence result for \eqref{eq:diffusion}.

\begin{theorem}\label{th:order2_conv1}
Let $k\geq 0$ and let $\sigma$ be a constant, and
assume that the exact solution $v$ of \eqref{eq:diffusion} 
has bounded derivative $\frac{\partial^{q} v}{\partial x^{q}}$ for $q=\max(k+2,2p+2)$.
We consider the SLDG-p schemes with $p=1,2$ or $3$.
Then
\be
  \| v^{n} - u^{n} \|_{L^2} \leq
    \|v^0-u^0\|_{L^2} +   C T (\frac{\dx^{k+1}}{\dt} + \dt^p),
  \quad \forall n\leq N.
\ee
Furthermore the same results hold for the variants \eqref{eq:SLDG-RK2variant},\eqref{eq:SLDG-RK3variant} for $p=2,3$.
%
\end{theorem}

\if{
We first consider the SLDG-1 scheme ($p=1$).
By definition, we have 
\be
  (u^{n+1}, \varphi) = (S^0_\dt u^n, \varphi), \quad \forall \varphi \in V_k.
\ee
As previously mentioned, since the scheme is implemented by using Gaussian quadrature formula which is exact for polynomials of $P_{2k+1}$
on each interval of regularity, we observe that the previous identity is valid for one more order:
\be (u^{n+1}, \varphi) = (S^0_\dt u^n, \varphi), \quad \forall \varphi \in V_{k+1}.
\ee
Then let $\tilde v^{n+1} : = \Pi_{V_{k+1}}(v^{n+1})$. By using the scheme definition,
\beno
  \| u^{n+1} - \tilde v^{n+1} \|_{L^2}^2  
  & = & ( u^{n+1}, u^{n+1} - \tilde  v^{n+1} )  -  (\tilde v^{n+1}, u^{n+1} - \tilde v^{n+1}) \\
  & = & ( S^0_\dt u^n, u^{n+1} - \tilde v^{n+1} ) -  (\tilde v^{n+1}, u^{n+1} -   \tilde v^{n+1})  \\
  & = & ( S^0_\dt u^n- \tilde v^{n+1} , u^{n+1} - \tilde v^{n+1} ). 
\eeno
In particular, by using the Cauchy-Schwarz inequality,
$$  \| u^{n+1} - \tilde v^{n+1} \|_{L^2} \leq \| S^0_\dt u^n - \tilde v^{n+1} \|_{L^2}.  $$
Since $\|v^{n+1} - \tilde v^{n+1} \|_{L^2} \leq C \dx^{k+2}$, and by using the consistency estimate
\be \label{eq:SLDGRK1consist}
   v^{n+1} =  S^0_\dt v^n + O(\dt^2 ( \|v^n_{tt}\|_{L^\infty} +  \| v^{n,(4)}_{x}\|_{L^\infty} )), 
\ee
we obtain
\beno
  \| u^{n+1} - v^{n+1} \|_{L^2} 
     & \leq &  \| S^0_\dt u^n  -  v^{n+1} \|_{L^2}  +  2 C \dx^{k+2} \\
     & \leq &  \| S^0_\dt u^n  -  S^0_\dt v^{n} \|_{L^2}  +  C \dt^2 +  C \dx^{k+2} \\
     & \leq &  \| u^n -  v^{n}\|_{L^2}  +  C \dt^2 +  C \dx^{k+2}. \\
\eeno
where we have also used the fact that $\| S^0_\dt w \|_{L^2} \leq \| w \|_{L^2}$.
The result follows by induction.
}\fi
\proof
We  will consider the proof in the case of the SLDG-2 scheme, with $p=2$, the other cases being similar.
By using the regularity of the exact solution ($\frac{\partial^3 v}{\partial t^3}$ and $v^{n,(6)}_x$ bounded), we have the following consistency 
estimate:
\be
   v^{n+1} = a_0 v^n + a_1 S^0_\dt v^n + a_2 S^0_\dt S^0_\dt v^n  + O(\dt^3),
   \label{eq:SLDGRK2consist}
\ee
where $a_0=a_1=a_2=\frac{1}{3}$, and the bound $O(\dt^3)$ is in the norm $\|.\|_{L^2}$. 
Since $\Pi S^0_\dt  \psi= \Pi S^0_\dt \Pi \psi + O (\dx^{k+1}) $ for regular data $\psi$, we have also
$S^2_\dt v^n = \Pi (S^0_\dt)^2 v^n  + O(\dx^{k+1})$, and thus
\be
   v^{n+1} = a_0 v^n + a_1 S_\dt v^n + a_2 S_\dt S_\dt v^n  + O(\dt^3) + O(\dx^{k+1}).
   \label{eq:SLDGRK2consist.b}
\ee
By the definition of the scheme we have
\be
  u^{n+1} = \sum_{i=0}^2 a_i (S_{\dt})^i u^n.
\ee
We deduce, using the consistency estimate \eqref{eq:SLDGRK2consist},
\beno
  \| u^{n+1} - v^{n+1} \|_{L^2} 
     & \leq &  \| \sum_{i\leq 2} a_i (S_\dt)^i (u^n - v^{n}) \|_{L^2}  +  C\dt^3 +  C \dx^{k+1}\\
     & \leq &  \sum_{i\leq 2} a_i \| (S_\dt)^i (u^n - v^n) \|_{L^2}  +  C\dt^3 +  C \dx^{k+1}\\
     & \leq &  \| u^n -  v^{n}\|_{L^2}  +  C \dt^3 +  C \dx^{k+1}, \\
\eeno
(since $a_i\geq 0$ and $\sum_i a_i =1$).
The result follows by induction.
\endproof

\if{
Also, by the $\cC^{k+1}$ regularity (in the variable $x$) of $v^{n+1}$ and Lemma~\ref{lem:proj} we have 
\be\label{eq:estim01}
  v^{n+1} = \Pi (a_0 v^n + a_1 S^0_\dt v^n + a_2 S^0_\dt S^0_\dt v^n)  + O(\dt^3)+ O(\dx^{k+1}). 
\ee
Furthermore, since $S^0_\dt v^n$ is $\cC^{k+2}$ regular,  we have $\Pi_{V_{k+1}} S^0_\dt v^n - S^0_\dt v^n =  O(\dx^{k+1})$ and therefore
$$ \Pi S^0_\dt \Pi S^0_\dt v^n - \Pi S^0_\dt S^0_\dt v^n =  O(\dx^{k+1})$$ 
(using the fact that $\| \Pi S^0_\dt w \|_{L^2} \leq \|w\|_{L^2}$).
From \eqref{eq:estim01}, and since $S_\dt=\Pi S^0_\dt$, 
\be
  v^{n+1} 
   & = & a_0 v^n + a_1 S_\dt v^n + a_2 S_\dt S_\dt v^n   \nonumber \\
   &   & \hspace{2cm} + O(\dt^3)+ O(\dx^{k+1}).  \label{eq:estim02}
\ee
Considering the difference with the scheme, 
$$  
  v^{n+1} - u^{n+1}  = a_0 (v^n - u^n) + a_1 S_\dt (v^n - u^n)  + a_2 S_\dt S_\dt (v^n - u^n) + O (\dt^3 + \dx^{k+1}).
$$
By $\|S_\dt w\|_{L^2} \leq \| w \|_{L^2} $ and the fact that $a_i\geq0$, $\sum_i a_i=1$,
we obtain 
$$ \| v^{n+1} - u^{n+1} \|_{L^2} \leq  \| v^n - u^n \|_{L^2}  + C (\dt^3 + \dx^{k+1}).$$
We conclude the proof by induction.
}\fi

\if{
\begin{rem}
The chosen denomination "RK1", "RK2" and "RK3" is because of the relationship with Runge-Kutta schemes for first-order equations.
It is indeed related to Strong Stability Preserving (SSP) schemes~\cite{Gottlieb_Shu_Tadmor_2001, Osher_Shu_88, Shu_REVIEW_07}. 
For instance, if we consider the Euler scheme for the ODE $\dot u = L(u)$ during time step $\dt$  written as $u^{n+1} = S(u^n)$, 
then it is known that the Heun scheme can also be written $u^{n+1}= \frac{1}{2} (u^n + S S u^n)$ 
and is of order $O(\dt^2)$, it corresponds to a so-called "SSP-RK2" approximation. This can be compared to the present SLDG-RK2 scheme.
A known "SSP-RK3" approximation for $\dot u= L(u)$ is given by $u^{n+1} = \frac{1}{6} (2 u^n + 3 S u^n + S S S u^n)$,
and can be compared to the present SLDG-RK3 scheme.
\end{rem}
}\fi

\FULL{
\begin{rem}
We have chosen the DG framework for the spatial approximation, but the same idea for improving the order in time can 
be combined with other spatial approximations~\cite{bok-bon-2013}.
\end{rem}
}

\if{
\subsection{$1d$-diffusion equations with non-constant coefficient $\sigma(x)$ {\bf (OLD)}}
We now consider the case of 
\be
   v_t + \frac{1}{2} \ms^2(x) v_{xx} = 0, \quad t>0, \quad x\in \mO.
  \label{eq:heateq}
\ee
On the first hand, if $v^n= v(t_n,x)$ where $v$ is the exact solution of \eqref{eq:heateq}, and for $h\equiv\dt$,
we have
\be
  v^{n+1} & = & v^n +  h v_t +  \frac{h^2}{2} v_{tt} + O(h^3)  \nonumber\\
          & = & v^n +  h \frac{\ms^2}{2} v_{xx} + h^2 \frac{\ms^2}{8} \big(\ms^2 v^n_{xx}\big)_{xx} + O(h^3) \nonumber \\
          & = & v^n +  h \frac{\ms^2}{2} v_{xx} + h^2 \frac{\ms^4}{8} (v^n)^{(4)}_x 
     + h^2 \bigg\{ \frac{\ms^2 (\ms^2)'}{4} (v^n)^{(3)}_{x} + \frac{\ms^2 (\ms^2)''}{8} v^n_{xx}\bigg\} \nonumber \\
          &   &  \hspace{8cm}+ O(h^3). 
  \label{eq:vnplusun-b}
\ee
Now, we observe that 
\be
  \frac{1}{3} (u + S^0_h u  + S^0_h S^0_h u) & =  & 
   \label{eq:mixedS0}
  \\
  & & \hspace{-3cm} 
   u +  h \frac{\ms^2}{2} u_{xx} + h^2 \frac{\ms^4}{8} u^{(4)}_{x} +  h^2
    \bigg\{   \frac{\ms^2 (\ms^2)'}{6} u^{(3)}_x + \frac{\ms^2 (\ms^2)''}{12} u_{xx} \bigg\} +  O(h^3).\nonumber
\ee
We have new terms depending of $h^2 u^{(3)}_x$ and $h^2 u^{(4)}_x$ and that will differ between~\eqref{eq:vnplusun-b}
and~\eqref{eq:mixedS0}.

In order to take these new terms into account,  we introduce the following approximations.
For a given $h>0$, 
let $D^0_h$ and $D_h$ be defined by
$$
   D^0_h u(x) := \frac{1}{2} (u(x+h) - u(x-h)) \quad \mbox{and} \quad D_h : = \Pi D^0_h.
$$
We have for a regular $u$: 
$$
  D^{0}_h u = h u_x + \frac{1}{6} h^3 u^{(3)}_x + O(h^5  \|u^{(5)}_x\|_\infty),
$$
from which it can be obtained that 
$$ 
   D^0_h D^0_h D^0_h u  = h^3 u^{(3)}_x + \frac{1}{2} h^{5} u_x^{(5)} +  O(h^7  \|u^{(7)}_x\|_\infty),
$$
and
$$ 
   D^0_h D^0_h D^0_h D^0_h D^0_h  u  = h^5 u^{(5)}_x + O(h^7  \|u^{(7)}_x\|_\infty).
$$
Therefore, 
\be\label{eq:approx_uxxx0}
  & & h (\Soht \Soht \Soht u - \frac{1}{2} \Soht \Soht \Soht \Soht \Soht u) \\
  & & \hspace{6cm} = h^2 u^{(3)}_x + O(h^{10/3}\|u_x^{(7)}\|), \nonumber 
\ee
and, in the same way, 
\be\label{eq:approx_uxx}
  h \Sohd \Sohd u  = h^2 u_{xx} + O(h^{3} \|u_x^{(4)}\|_\infty)
\ee
(\eqref{eq:approx_uxx} corresponds also to the usual second-order finite difference approximation).
Also, if $u$ is regular and that we replace the operator $D^0_h$ by $D_h$, they the above approximations
\eqref{eq:approx_uxxx0} or \eqref{eq:approx_uxx} are still valid up to a complementary error term in $O(\dx^{k+1})$
coming from the projection on $V_k$.
Hence we propose the following.\\ 
\begin{indent}
\mbox{\underline{SLDG-RK2 "modified" scheme}:} 
\end{indent}
\be
  & & u^{n+1} = \frac{1}{3} (u^n + S_\dt u^n  + S_{\dt} S_\dt u^n)  \nonumber \\
  & & \hspace{1cm}  
   + \dt \frac{\ms^2 (\ms^2)'}{12}  \big(  \Dtt \Dtt \Dtt u^n - \frac{1}{2} \Dtt\Dtt\Dtt\Dtt\Dtt u^n \big)  \nonumber\\
  & & \hspace{1cm}  
   + \dt \frac{\ms^2 (\ms^2)''}{24} \Dtd \Dtd u^n \label{eq:modified_SLDG_RK2}.
\ee
Indeed, the missing part for the term $\frac{1}{6} \ms^2 (\ms^2)' u_x^{(3)}$, with factor $\frac{1}{6}$, 
from the desired factor $\frac{1}{4}$ in~\eqref{eq:vnplusun-b}, is $\frac{1}{4}-\frac{1}{6}=\frac{1}{12}$.
Also, the missing part for the term $\frac{1}{12}\ms^2  (\ms^2)'' u_{xx}$, with factor $\frac{1}{12}$,
from the desired factor $\frac{1}{8}$ in~\eqref{eq:vnplusun-b}
is $\frac{1}{24}$.
The expected error is of order $O(\dt^3)+O(\dx^{k+1})$ at each time step, hence the expected global error is of order
$O(\dt^2) + O(\frac{\dx^{k+1}}{\dt})$.
}\fi

\subsection{Advection-diffusion with variable coefficients}
We recall that for the following PDE:
\be
  & & -v_t  - \frac{\sigma(t,x)^2}{2} v_{xx} - b(t,x) v_x + r(t,x) v = f(t,x), \quad x\in \mO,\ t\in (0,T), \nonumber \\
  & & \hspace{10cm} \label{eq:generalPDE-backward}
\ee
with $\mO=\R$ and terminal condition $v(T,x):=w(T,x)$,
introducing a probability space
 $\left({\mathcal Q}, \mathbb F, \mathbb P\right)$ with a  filtration $\{\mathbb F_t\}_{t\geq 0}$, 
and a one-dimensional Brownian motion $(W_t)_{t\geq 0}$,
and the solution $X_s=X^{t,x}_s$ of the stochastic differential equation
\beno
  & & dX_s = b(s,X_s) ds + \ms(s,X_s) dW_s,\quad s\geq t,\\
  & &  X_t=x,
\eeno
and if $v$ is a regular solution of the PDE \eqref{eq:generalPDE-backward}
on $(t,T)$ (assuming that the partial derivatives $\partial_t v$ and 
$\partial_{xx} v$ exist and are continuous) then 
the following equivalent expectation, or "Feynman-Kac" formula, holds:
\be
 & &  v(t,x) = \E \bigg[ e^{-\int_t^{T} r(\mt,X_\mt)d\mt} w(T,X^{t,x}_{T})
  +  \int_t^{T} e^{-\int_t^s r(\mt,X_\mt)d\mt} f(s,X^{t,x}_{s})\,ds
  \ \big|\ \cF_t \bigg]. \nonumber
 \\
 & & \label{eq:feynman-kac}
\ee

To simplify, we shall focus here on the case when $b$ and $\ms$ do not depend of time, and $r$ is constant.
We consider the forward PDE:
\be \label{eq:heateq}
  & & u_t\ -\frac{\sigma(x)^2}{2} u_{xx} - b(x) u_x + r u = f(t,x), \quad x\in \mO,\ t\in (0,T). \label{eq:generalPDE}
\ee
In that case the Feynman-Kac formula gives, with $h=\dt$, $T=t+h$ and $u^n(x):=u(t_n,x)$:
\be\label{eq:expect-formula} 
  & & u^{n+1}(x) = \E \bigg[ e^{-rh} u^n(X^{0,x}_{h})\big|\ \cF_t \bigg]  + w(h,x)
\ee
with 
\be \label{eq:whx}
  w(h,x):=
  \E \bigg[ \int_0^{h} e^{-rs} f(t_n+h-s,X^{0,x}_{s})\,ds
   \ \big|\ \cF_t \bigg].
\ee
Let $\cA w:= \frac{\sigma(x)^2}{2} w_{xx} + b(x) w_x - r w$.
The term $w(h,x)$ is also the solution at time $s=h$ of the linear problem
$ w_t(s,x) = (\cA w)(s,x) + \bar f(s,x)$ with initial condition $w(0,x)=0$, and with $\bar f(s,x):=f(t_n+s,x)$.
Assuming that the source term $f$ is regular and that we can use its derivatives, 
we can approximate it with an error $O(h^{q+1})$ by using a  Taylor expansion:
$w(h,x)\simeq \sum_{j=1}^q \frac{h^j}{j!} w_{jt}(0,x)$ (where $w_{jt}$ denotes the $j$-th derivative
with respect to time).
In particular, $w_t(0,x)=\bar f(0,x) = f(t_n,x)$, and 
$w_{tt}= (\cA w + \bar f)_t = \cA w_t + \bar f_t = \cA (\cA w + \bar f) + \bar f_t$, so $w_{tt}(0,x)=(\cA f)(t_n,x) + f_t(t_n,x)$.
Hence in order to devise a second-order scheme we approximate \eqref{eq:whx} by
\be \label{eq:whx-approx}
  w(h,x) 
  & = & h f(t_n,x) +  \frac{h^2}{2}(\cA f + f_t)(t_n,x) + O(h^3).
\ee
The modification of the scheme is obtained, therefore, by adding at each time step the following correction term atGGauss quadrature points
\be \label{eq:whx-correction}
  h f(t_n,x) +  \frac{h^2}{2}(\cA f + f_t)(t_n,x).
\ee


For the approximation of the expectation in \eqref{eq:expect-formula}, 
we aim to use a higher-order semi-discrete approximation
also called "weak Taylor approximations" in the stochastic setting, 
see in particular Kloeden and Platen~\cite[Chapter 15]{klo-pla-95}.
General semi-discrete (and fully-discrete) approximations can be found in~\cite{kus-dup-92}. 

We will focus on first- and second-order weak Taylor approximations.
Some of these approximation may use the derivatives of $b$ and $\ms$
(Milstein~\cite{Milstein_1986}, Talay~\cite{Talay_1984}, Pardoux and Talay~\cite{Pardoux_Talay_1985}).
In our case we shall use a derivative-free formula of Platen~\cite{Platen_1984}
(explicit second- and third-order derivative-free formula can be found in Kloeden and Platen~\cite{klo-pla-95},
as well as multidimensional extensions).

\if{
\footnote{OLD: 
Let $X_h$ be the Markov chain such that if $X_0=x$, then
$$ 
  X_h = x +  \ms(x) W\, \sqrt{h}  
          + \frac{1}{2}\ms(x)\ms'(x) (W^2-1)\, h 
	  + \frac{1}{4}\ms^2(x)\ms''(x) h^{3/2} W,
$$
where $W$ is a $N(0,1)$ random variable.
We consider the approximation of $W$ by a real-valued random variable $Q$ such that 
$\bP(Q=\{-\sqrt{3},0,\sqrt{3}\})$ is $\{\frac{1}{6}, \frac{2}{3}, \frac{1}{6}\}$,
respectively, and the corresponding trajectories 
$$ 
  y^q_h(x) = x +  \ms(x) q\, \sqrt{3h}  
          + \frac{1}{2}\ms(x)\ms'(x) (3q^2-1)\, h 
	  + \frac{1}{4}\ms^2(x)\ms''(x) \sqrt{3} q\, h^{3/2}
$$
that will be considered for $q \in\{-1,0,1\}$ only.
}
}\fi

Let us denote $b=b(x)$, $\ms=\ms(x)$ as well as $\mg^q_\dt=\mg^q_\dt(x)$:
\be\label{eq:y_euler}
  \mg^q_\dt(x)  & := & x + b(x)\dt +   q  \ms(x) \sqrt{\dt}.
\ee

Our SLDG-1 scheme, corresponding to a first-order (weak Euler scheme), is defined by 
\be \label{eq:SLDG_RK1}
  u^{n+1}\equiv S^{(1)}_\dt u^n := \Pi \bigg( \sum_{q=\pm 1} \ma_q u^n(y^q_\dt(\cdot))\bigg)
\ee
with weights $\ma_{-1}=\ma_1=\frac{1}{2}$ and characteristics $y^q_h=\mg^q_h$.

Our SLDG-2 scheme, corresponding to the second-order Platen's scheme, is defined by
\be \label{eq:SLDG-2}
  u^{n+1}\equiv S^{(2)}_\dt u^n := \Pi \bigg( \sum_{-1\leq q\leq1 } \ma_q u^n(y^q_\dt(\cdot))\bigg)
\ee
with weights $\ma_{-1}=\ma_1=\frac{1}{6}$ and $\ma_0=\frac{2}{3}$ and characteristics $y^q_h=y^q_h(x)$ defined by:
\be
  y^q_h(x) 
  &  = &  x + \frac{1}{2}(b(\mg^{\sqrt{3} q}_h)+b)\,h   
    \label{eq:y_platen}
  \\
  & & 
  + \frac{1}{4}\bigg[\big( \ms(\mg^1_h) + \ms(\mg^{-1}_h)+ 2\ms\big) \sqrt{3}\ q 
  + \big(\ms(\mg^1_h) - \ms(\mg^{-1}_h)\big) (3q^2-1) \bigg]
    \sqrt{h}.
    \nonumber
\ee

\begin{rem} In the constant coefficient case $\ms(x)\equiv \ms$, the scheme becomes
\be \label{eq:SLDG_Platen_msconstant}
  u^{n+1}\equiv S^{(2)}_\dt u^n := \Pi \bigg( \frac{1}{6} u^n(x-\ms\sqrt{3\dt}) + \frac{2}{3} u^n(x) +  \frac{1}{6} u^n(x+\ms\sqrt{3\dt})
  \bigg)
\ee
\end{rem}

\begin{rem} 
Higher-order weak Taylor schemes can be found in \cite{klo-pla-95} and could be used with DG to devise fully discrete schemes in the same way.
\end{rem}

The above SLDG-1/2 schemes
are no more exactly implementable because $b(x)$ and $\ms(x)$ are not constant.
So, as in the advection case,
we consider the use of a Gaussian quadrature rule on each interval of regularity 
of the data.


\begin{rem} Notice that if $h$ is small enough such that
\be \label{eq:inverse-1-cond} 
  \| h b' + \sqrt{h} \ms'\|_{L^\infty}< 1,
\ee
then for each $q=\pm 1$ the function $x\converge \mg^q_h(x)$ is a one-to-one and onto function.
Furthermore, its inverse can be easily and rapidly computed by using a fixed point method or Newton's algorithm.
Details are left to the reader.

In the same way, for $h$ small enough such that, for instance,
\be \label{eq:inverse-2-cond} 
  h \|b'\|_{L^\infty} + 3\sqrt{h} \|\ms'\|_{L^\infty}< 1,
\ee
then $x\converge y^q_h(x)$ as defined in \eqref{eq:y_platen}
is one-to-one and onto function.
\end{rem}

\indent\mbox{\underline{SLDG-1 scheme (fully discrete)}:} 
For each given $\eta=\pm 1$,
we consider a partition of $I_i$ into intervals $J^\eta_{i,q}$ such that all $y^{\eta}(J^\eta_{i,q})$ are subintervals
of some $I_j$. We then define Gauss points $\tilde x^{i,\eta}_{q,\ma}$ and the bilinear product $(a,b)_{G^\eta}$ in a similar way as 
in~\eqref{eq:defi_bilin_G}, that is, 
using the Gaussian quadrature rule on each $J^\eta_{i,q}$.
Hence we define $\widetilde S^{(1)}_{\dt} u^{n}$ in $V_k$ such that\\
\be
  & &  \hspace{-2.5cm} 
    (\widetilde S^{(1)}_{\dt} u^{n}, \varphi) =  \ud \sum_{\eta=\pm} (u^n(y^\eta),\varphi)_{G^\eta}, \quad \forall \varphi \in V_k.
  \label{eq:tildeSdtdef_nonconstant}
\ee
Formula~\eqref{eq:tildeSdtdef_nonconstant} involves two different quadrature rules, because the discontinuity
points of $u^n(y^+(x))$ and $u^n(y^-(x))$ are not the same.
It differs from the definition of $S^{(1)}_{\dt} u$, which satisfies
\be    (S^{(1)}_{\dt} u, \varphi) =  \ud \sum_{\eta=\pm} (u(y^\eta),\varphi), \quad \forall \varphi \in V_k.
  \label{eq:Sdtdef_nonconstant}
\ee

\indent\mbox{\underline{SLDG-2 scheme (fully discrete)}:}  In a similar way, we define
$\widetilde S^{(2)}_{\dt} u^n$ in $V_k$ by: 
\be
  & &  \hspace{-2.5cm} 
    (\widetilde S^{(2)}_{\dt} u^{n}, \varphi) =   \sum_{-1\leq \eta \leq 1} \ma_\eta
    (u^n(y^\eta_\dt),\varphi)_{G^\eta}, \quad \forall \varphi \in V_k.
  \label{eq:tildeS2}
\ee


\if{
{\color{blue}
In order to implement the SLDG-2 modified scheme, 
the operator $\widetilde S_\dt$ is used instead of $S_\dt$ in \eqref{eq:modified_SLDG_RK2}.
New terms appear also in \eqref{eq:modified_SLDG_RK2} and 
that are computed naturally in the following way.
We first compute the DG polynomials for $\Dtt \Dtt \Dtt u^n$ (resp. $\Dtd \Dtd u^n$, etc.). 
The operators $\Dtt$ or $\Dtd$ involve only exact quadrature rules since they use constant advection parameters.
Then we pointwise multiply the polynomials at each Gauss quadrature points $x^i_\ma$ of the intervals $I_i$,
by the corresponding value of $\dt\, \ms^2(\ms^2)'$ (resp. $\dt\, \ms^2 (\ms^2)''$). 

These implementation steps create errors with respect to the theoretical schemes, \eqref{eq:modified_SLDG_RK2} or \eqref{eq:SLDG_RK1}, 
and that need to be controlled.
}
}\fi

%


\subsection{Stability and convergence}



We first state some useful estimates for the operators 
$\widetilde S_\dt\in\{ \widetilde S^{(1)}_\dt,\, \widetilde S^{(2)}_\dt\}$.
The proof is similar to the one of Proposition~\ref{prop:2.1}.

\begin{proposition}\label{prop:3.1}
Let $k\geq 0$ and let $\ms$ be of class $\cC^{2k+2}$ and $1$-periodic. Then:\\
$(i)$ there exists a constant $C\geq 0$ such that, for any $y^q_\dt$, for all $u\in V_k$, 
\be 
  \label{eq:bound_prop31.b}
  & & \bigg| (u(y^q_\dt),\varphi)_{G^\eta} - (u(y^q_\dt),\varphi) \bigg|
  \leq
  C \sdt \dx^2 \| u \|_{L^2} \| \varphi \|_{L^2}
  \quad \forall \varphi \in V_{k}.  \nonumber 
\ee
In particular, for any $u\in V_k$,
\be
  \label{eq:bound_prop31_2}
  \widetilde S_{\dt} u = S_{\dt} u + O\big( \sdt\dx^2  \| u\|_{L^2}\big).
\ee
$(ii)$ 
For all $u\in V_k$, for any $\psi$ in $\cC^{k+1}$, 1-periodic,
\be 
    \widetilde S_\dt (u - \psi) =  S_\dt (u-\psi) + 
    O(\sdt \dx^2 \| u - \psi \|_{L^2}) + O(M_{k+1}(\psi) \dx^{k+1}),
   \label{eq:bound_prop31_ii.b}
\ee
where $C\geq 0$ is a constant.\\
$(iii)$ 
For any regular $\psi \in \cC^{k+1}$, $1$-periodic, we have in the $L^2$ norm 
\be
  \widetilde S_\dt  \psi  = S_\dt \psi + O (M_{k+1}(\psi) \dx^{k+1}).
  \label{eq:bound_prop31_iiib}
\ee
\end{proposition}


\if{
It is known that for the error of the Gaussian quadrature rule we have, $\forall \varphi \in \cC^{2k+2}$, $\exists \xi_x\in(a,b)$ 
such that
\be\label{eq:gausserr}
  \int_a^b \varphi(x) dx - \sum_{\ma} w_\ma \varphi(x_\ma) =  \frac{\varphi^{(2k+2)}(\xi_x)}{(2k+2)!} \int_a^b p_n^2(x) dx
\ee
where $p_n(x)=\prod_{\ma=0,\dots,k} (x-x_\ma)$ (see~\cite{Stoer_Bulirsch}), and 
$ \int_a^b p_n^2(x) dx \leq C (b-a)^{2k+3}$.

Now, we define again an error term $\eps^{\pm}_{i,q}$, in a similar way as in~\eqref{eq:gauss_iq},
\be \label{eq:gauss_iq_new}
  \eps^{\pm}_{i,q}:= 
  \sum_{\ma=0}^k \tilde w^i_{q,\ma} u\big(y^{\pm}(\tilde x^i_{q,\ma})\big)\varphi(\tilde x^i_{q,\ma})
  -
  \int_{J_{i,q}}  u\big(y^\pm_{x}\big)\varphi(x)\, dx,
\ee
and the global error to control corresponds to $\eps:=\sum_{i,q} \sum_{\eta=\pm} \eps^\eta_{i,q}$.
We look at the sum of the two error terms for given indices $(i,q)$, using~\eqref{eq:gausserr}:
\beno
  \sum_{\eta=\pm} \eps^{\pm}_{i,q}:= 
  \sum_{\eta=\pm} 
   \frac{ \big(u(y^\eta)\varphi\big)^{(2k+2)}(\xi^\eta_{i,q})}{(2k+2)!} \int_{J_{i,q}} p_n^2(x) dx,
\eeno
with $J_{i,q}:=(x_{i,q},x_{i,q+1})$ and $\xi^\pm_{i,q}\in J_{i,q}$. 
Then we expand the $(2k+2)$ derivatives, using that $\varphi\in V_{k+1}$:
\beno 
  \sum_{\eta=\pm} (u(y^\eta)\varphi)^{(2k+2)}(\xi^\eta_{i,q})
    & = & 
   \sum_{r=0}^{k}  C_{2k+2}^r \sum_{\eta=\pm} \ \varphi^{(r)}(\xi^\eta_{i,q})\  (u(y^\eta))^{(2k+2-r)}(\xi^{\eta}_{i,q}) 
\eeno
Let us denote, for given indices $(i,q)$, 
$$
  \Delta_r := \frac{1}{2} \sum_{\eta=\pm} \varphi^{(r)}(\xi^\eta_{i,q})  (u(y^\eta))^{(2k+2-r)}(\xi^\eta_{i,q}).
$$
Since $\int_{J_{i,q}} p_n^2(x) dx \leq C \dx^{2k+3}$, for a given $(i,q)$ we have the bound
\beno
  |\sum_{\eta=\pm} \eps^{\pm}_{i,q}| & \leq  &  \sum_{r=0}^{k} \dx^{2k+3} |\Delta_r|.
\eeno

By making use of Fa{\`a} di Bruno's formula and as in the proof of Proposition~\ref{prop:2.1}, 
denoting $h=\eta\sdt$ for $\eta =\pm$, 
we obtain that for each $0\leq r\leq k+1$, 
\beno
  \dx_{i,q}^{2k+3}|\Delta_r| 
    & \leq & C \dx_{i,q}^{2k+3} \|\varphi^{(r)} \|_{L^\infty(J_{i,q})} \sum_{\eta=\pm}\| (u(y^\eta))^{(2k+2-r)} \|_{L^\infty(J_{i,q})}  \\
    & \leq & C \dx_{i,q}^{2k+3} \frac{\|\varphi\|_{L^2(J_{i,q})}}{\dx_{i,q}^{r+\ud}}  \sum_{\eta=\pm} |h|
      \frac{\| u\|_{L^2(y^\eta(J_{i,q}))}}{\dx_{i,q}^{k+\ud}}  \\
    & \leq & C |h| \dx_{i,q}^{k-r+2} \sum_{\eta=\pm} \|\varphi\|_{L^2(J_{i,q})} \| u\|_{L^2(y^\eta(J_{i,q}))}.
\eeno
Using this bound for all $r\leq k$, and using $\dx_{i,q}\leq \dx\leq 1$, we obtain
\be \label{eq:3.1firstpart}
  \dx_{i,q}^{2k+3} \sum_{r=0}^{k} |\Delta_r|  \leq C |h| \dx^{2} \sum_{\eta=\pm}  \|\varphi\|_{L^2(J_{i,q})} \| u\|_{L^2(y^\eta(J_{i,q}))}.
\ee
There remains to bound the last term for $r=k+1$. 
We remark that $\varphi^{(k+1)}$ is a constant on the interval $J_{i,q}^\eta$ and for $\varphi \in V_{k+1}$, hence 
\be
  \hspace{-0.5cm}
  \Delta_{k+1}
         & = & \varphi^{(k+1)} \ud \sum_{\eta=\pm} (u(y^\eta))^{(k+1)}(\xi^\eta_{i,q}) \nonumber \\
         & = & \varphi^{(k+1)} \bigg(\ud \sum_{\eta=\pm}  (u(y^\eta))^{(k+1)}(\bar x_{i,q})\bigg) \nonumber \\
	 &   & \hspace{1cm}   +  O(\dx) \sum_{\eta=\pm} 
	        \|\varphi^{(k+1)}\|_{L^\infty(J_{i,q}^\eta)} \| u(y^\eta)^{(k+1)} \|_{L^{\infty}(J_{i,q}^\eta)},
	        \label{eq:Delta_k}
\ee
where we have denoted $\bar x_{i,q}:=\frac{1}{2} (x_{i,q} + x_{i,q+1})$, the center of the interval $J_{i,q}$. 
We use again Fa{\`a} di Bruno's formula \eqref{eq:FaadiBruno}. 
Let $h:=\sdt$. Using the fact that $y^\eta(x)=1 + \eta h \ms(x)$ and  $(y^\eta)^{(p)}=\eta h\ms^{(p)}(x)$ for $p\geq 1$ and $\eta=\pm$,
we obtain a Taylor expansion of the first term of $(u(y^\eta))^{(k+1)}$ in \eqref{eq:Delta_k} in terms of powers of $h$,
the first power being greater than one (as in Lemma~\ref{lem:estim_2}).
Hence, the contribution in $h$ in the first term of \eqref{eq:Delta_k} vanishes
(because of the sum over $\eta=\pm$), and there will remain only terms of the order of $h^2$, or of higher-order:
$$
    |\sum_{\eta=\pm}  (u(y^\eta))^{(k+1)}(\bar x_{i,q})| 
       \leq C h^2 \sum_{\eta=\pm} \sum_{p=1}^k \|u^{(p)}\|_{L^\infty(y^\eta(J_{i,q}^\eta))} 
       \leq C h^2 \sum_{\eta=\pm}  \frac{\|u\|_{L^2(y^\eta(J_{i,q}^\eta)) } }{\dx_{i,q}^{k+\ud}},
$$
from which follows 
\be
  \dx_{i,q}^{2k+3} |\Delta_{k+1}| 
    & \leq & C \dx_{i,q}^{2k+3} \frac{\| \varphi\|_{L^\infty(J_{i,q}^\eta)}}{\dx_{i,q}^{k+1+\ud}} (h^2 +h\dx)\frac{\|u\|_{L^2} }{\dx_{i,q}^{k+\ud}} \nonumber \\
    & \leq & C (h^2 \dx + h \dx^2)  \sum_{\eta=\pm}  \|\varphi\|_{L^2(J_{i,q}^\eta)} \| u\|_{L^2(y^\eta(J_{i,q}^\eta))}.
 e \label{eq:3.1secondpart}
\ee
Finally, summing the estimates of \eqref{eq:3.1firstpart} and \eqref{eq:3.1secondpart}, we obtain
\beno
  |\eps| & \leq &  C (\dt\dx + \sdt\dx^2)\  \sum_{\eta=\pm} \sum_{i,q}  \|\varphi\|_{L^2(J_{i,q}^\eta)} \| u\|_{L^2(y^\eta(J_{i,q}^\eta))} \\
         & \leq &  C (\dt\dx + \sdt\dx^2) \ 2 \| \varphi\|_{L^2}  \| u\|_{L^2}
\eeno
(as in Proposition~\ref{prop:2.1}$(i)$), which concludes the proof.
}\fi

\if{
Now we turn on the proof of $(iii)$. Indeed we proceed as in $(ii)$, using now the regularity of $\psi$. 
In the case when $\varphi \in V_{k+1}$, we find that 
\be 
    & & \hspace{-2cm}  (\dx^\eta_{i,q})^{2k+3} \|[\psi(y^\eta) \varphi]^{(2k+2)} \|_{L^\infty(J^\eta_{i,q})}  \nonumber \\
    & \leq &
   (\dx^\eta_{i,q})^{2k+3} \sum_{r\leq k+1} 
   \| \varphi^{(r)} \|_{L^\infty(J^\eta_{i,q})} \| (\psi(y^\eta))^{(2k+2-r)} \|_{L^\infty(J^\eta_{i,q})}  \nonumber \\ 
     & \leq & 
   (\dx^\eta_{i,q})^{k+1+1/2} \| \varphi \|_{L^2(J^\eta_{i,q})} M_{2k+2}(\psi) \nonumber 
\ee
Summing up this bounds and using 
$ \sum_{i,q} (\dx^\eta_{i,q})^{1/2} \| \varphi \|_{L^2(J^\eta_{i,q})} \leq \| \varphi \|_{L^2}$, 
we deduce, for each $\eta=\pm$, for all $\varphi \in V_{k+1}$,
\beno 
  \bigg| (\psi(y^\eta),\varphi)_{G^\eta} - (u(y^\eta),\varphi) \bigg| 
  \leq
  C \dx^{k+1}  M_{2k+2}(\psi)  \| \varphi \|_{L^2}. 
\eeno
Hence, by \eqref{eq:tildeSdtdef_nonconstant} and \eqref{eq:Sdtdef_nonconstant} (still valid for any $\varphi\in V_{k+1}$), 
we obtain the desired result.
In the case $\varphi \in V_k$, we obtain estimate
\eqref{eq:bound_prop31_iiib} in the same way.
}\fi

We now establish stability  properties.

\begin{proposition} \label{prop:3.2}
Let $k\geq 0$, and assume that $h$ is small enough in order that~\eqref{eq:inverse-1-cond} (resp. \eqref{eq:inverse-2-cond}) holds.
\\
$(i)$ (Stability with exact integration as in~\eqref{eq:SLDG_RK1}.)
For any $u \in V_k$, 
\beno
 \| S_\dt u \|_{L^2} \leq  (1+C\dt) \|u \|_{L^2},
\eeno
where $C\geq 0$ is a constant.\\
$(ii)$ (Stability with Gaussian quadrature rule as in~\eqref{eq:tildeSdtdef_nonconstant}.)
For any $u\in V_k$,
\beno
 \| \widetilde S_\dt u \|_{L^2} \leq  ( 1+ C \dt + C \sdt\dx^2) \| u\|_{L^2}.    
\eeno
$(iii)$ In particular the fully discrete schemes SLDG-1 and -2 are $L^2$ stable under the "weak" CFL condition 
\be
  \label{eq:weakCFL} \dx^4 \leq \ml \dt, \quad \mbox{for some $\ml>0$}.
\ee 
\end{proposition}
\proof
$(i)$
By making use of the convexity of $x\converge x^2$, the change of variable formula $x\converge y^q_\dt(x)$
(and denoting also $z\converge x^q_\dt(z)$ the inverse function of $y^q_\dt$),
we have
\beno
 \| S_\dt u \|_{L^2}^2 
  &  =   &  \int \bigg|\sum_{q} \ma_q  u(x+ b(x)\dt + q \ms(x) \sqrt{\dt}) \bigg|^2 dx \\
  & \leq &  \int \sum_{q} \ma_q \bigg| u(x+ b(x)\dt + q \ms(x) \sqrt{\dt}) \bigg|^2 dx \\
  &   =  &  \int \sum_q   \frac{\ma_q}{1+ b'(x^q(z))\dt + q \ms'(x^q(z))\sqrt{\dt}} |u(z)|^2 dz.
\eeno
Then we remark that $x^q_\dt(z)=x + O(\sqrt{\dt})$, 
so $1+ b'(x^q(z))\dt + q \ms'(x^q(z))\sqrt{\dt}  = 1 + q \ms'(x) \sqrt{\dt} + O(\dt)$, 
and for $\dt$ small enough  $0\leq (1+ b'(x^q(z))\dt + q \ms'(x^q(z))\sqrt{\dt})^{-1}  \leq  1 - q \ms'(x) \sqrt{\dt} + C\dt$
for some constant $C\geq 0$. 
Hence 
\beno
 \| S_\dt u \|_{L^2}^2 
  & \leq &  \int \sum_q   \ma_q (1 - q \ms'(x)\sqrt{\dt} + C \dt)  |u(z)|^2 dz \\
  & \leq &  (1+C\dt) \int |u(z)|^2 dz 
\eeno
where we have used that $\sum \ma_q=1$, and $\sum_q q \ma_q=0$.
The desired result follows.

$(ii)$ This is a consequence of $(i)$ and of the bound \eqref{eq:bound_prop31_2} of Proposition~\ref{prop:3.1}.
\endproof

The convergence result for the approximation of~\eqref{eq:heateq} is the following.

\begin{theorem}\label{th:order2_conv2}
Let $k\geq 0$ and let $\ms$ be a $1$-periodic function, of class $\cC^{2k+2}$.
We consider the schemes SLDG-p for $p=1,2$ (implementable version). 
\\
Assume the exact solution $v$ has a bounded derivative $\frac{\partial^{q} v}{\partial x^{q}}$ for $q=\max(2p+2,k+1)$,
and that the weak CFL condition \eqref{eq:weakCFL} is satisfied, then 
\be
  & & \hspace{-0.5cm} \| u^{n} - v^{n} \|_{L^2} \leq
   e^{L_1 T} \bigg( \|u^0-v^0\|_2  +   C T (\frac{\dx^{k+1}}{\dt} + \dt^p)   \bigg),
   \quad \forall n\leq N,
\ee
for some constant $L_1\geq 0$.
\end{theorem}

In particular for $\dt=\lambda \dx$ for any $\lambda >0$,
and $k=p \in \{1,2\}$, the SLDG-p schemes are fully discrete schemes and of order $O(\dx^p)$.

\newcommand{\point}{\bullet}

\proof[Proof of Theorem~\ref{th:order2_conv2}.]
\if{
\fbox{\bf OLD WAY !}

On the other hand, we have supplementary terms involving operators $D_h$, times a regular function, times a factor $\dt$. 
We first remark that  $\| D_h u \|_{L^2} \leq \|u\|_{L^2}$.
Then, let us define, for a given (continuous) function $a$,   the operator $a\point u$ in $V_k$ by
\be \label{eq:apoint}
   (a\point u)(x^i_\ma) :=  (au)(x^i_\ma), \quad \forall i,\ma.
\ee
One can establish using similar estimates as in the proof of Proposition~\ref{prop:2.1}, if $a$ is in $C^{2k+2}$ and is a
$1$-periodic function:
\be
    a\point u =  au  + O (M_{2k+2}(a)\, \dx^2 \| u \|_{L^2}), \quad \forall u \in V_k, 
    \label{eq:apointu.a}
\ee
and, if furthermore $\psi\in C^{k+1}$, $1$-periodic,
\be
    a\point \psi =  a\psi  + O (\dx^{k+1} M_{k+1}(a\psi)).
    \label{eq:apointu.b}
\ee
For the SLDG-RK2 scheme,  denoted $u^{n+1}=S^{RK2} u^n$, by using the estimate~\eqref{eq:apointu.a},
all supplementary corrections are controlled by (at most) $C \dt \|u^n \|_{L^2}$ for some constant $C\geq0$. 
In the end we obtain a bound of the form
$\| u^{n+1}  \|_{L^2} \leq (e^{L \dt} + C\dt) \|u^n\|_{L^2}$, which is enough to conclude to the stability.
}\fi
We first consider the SLDG-1 scheme $u^{n+1}=\widetilde S_\dt u^n$.
By making use  of the consistency error estimate, we have
\be 
  v^{n+1} = \Pi S^0_\dt v^n + O(\dt^2) + O(\dx^{k+1}) =  S_\dt v^n + O(\dt^2) + O(\dx^{k+1}).
\ee
Furthermore, by proposition \ref{prop:3.1}$(iii)$,
\be
  \| \widetilde S_\dt v^n - S_\dt v^n \|_{L^2} \leq C M_{k+1}(v^n) \dx^{k+1}.
\ee
Hence
\be
  v^{n+1} = \widetilde S_\dt v^n + O(\dt^2) +  O(\dx^{k+1}),
\ee
and by difference with the scheme $u^{n+1}=\widetilde S_\dt u^n$:
\be 
   \| u^{n+1} - v^{n+1}  \| & = &  \| \widetilde S_\dt u^n - \widetilde S_\dt v^n  \|_{L^2}  + C (\dt^2  + \dx^{k+1}) \\
    & \leq &  e^{C \dt} \| u^n - v^n  \|_{L^2}  + C (\dt^2  + \dx^{k+1}),
\ee
for some constant $C\geq 0$,
where we have made use of the stability estimate for $\widetilde S_\dt$.
Therefore  we obtain the desired error bound.

For the SLDG-2 scheme, the estimates are similar,
using the fact Platen's scheme is second-order to get the consistency estimate
$v^{n+1} = S^{(2)}_\dt v^n  + O(\dt^3) + O(\dx^{k+1})$. The conclusion follows.
\endproof

\if{
Let $v^n(x)=v(t_n,x)$, $\tilde v^{n+1} := \Pi_{V_{k+1}} v^{n+1}$ and $\tilde e^{n+1}:=u^{n+1} - \tilde v^{n+1}$.
We have
\be
  \| \tilde e^{n+1} \|_{L^2}^2  
  & = & ( \widetilde S_{\dt} u^n, \tilde e^{n+1} ) -  (\tilde v^{n+1}, \tilde e^{n+1})  \nonumber \\
  & = & ( \widetilde S_{\dt} u^n - \widetilde S_{\dt} v^n, \tilde e^{n+1} )  \nonumber \\
  &   &    \hspace{0cm} +\ ( \widetilde S_{\dt} v^n - S_\dt v^n, \tilde e^{n+1} )  \nonumber \\
  &   &    \hspace{0cm} +\ ( S_\dt v^n - v^{n+1}, \tilde e^{n+1} )  \nonumber \\
  &   &    \hspace{0cm} +\ (v^{n+1} - \tilde v^{n+1},\ \tilde e^{n+1}). 
       \label{eq:aaa}
\ee
By making use of the consistency error estimate, we have $v^{n+1} = S_\dt v^n + O(\dt^2)$ and therefore
\be
  |( S_\dt v^n - v^{n+1}, \tilde e^{n+1} )|\leq  C \dt^2 \|\tilde e^{n+1}\|_{L^2}. \label{eq:th22b1}
\ee
The other terms of the R.H.S.\ of \eqref{eq:aaa}
are bounded in a similar way as in the proof of Proposition~\ref{prop:2.1}. Using now Proposition~\ref{prop:3.1} we obtain here:
\be
 & & \hspace{-1cm}  ( \widetilde S_{\dt} u^n - \widetilde S_{\dt} v^n, \tilde e^{n+1} ) 
   \leq (S_{\dt} u^n -  S_{\dt} v^n, \tilde e^{n+1} )  \nonumber \\
 & &  + C \sdt \dx \| u^n - v^n \|_{L^2} \| \tilde e^{n+1}\|_{L^2}  + C \dx^{k+2} M_{2k+3}(\psi)\|\tilde e^{n+1}\|_{L^2} \label{eq:th22b2}\\
 & & \hspace{-1cm}  ( \widetilde S_{\dt} v^n - S_\dt v^n, \tilde e^{n+1} )  \leq   C M_{2k+2}(v^n) \dx^{k+1} \|\tilde e^{n+1}\|_{L^2}, \label{eq:th22b3}
\ee
and 
\be
 & & \hspace{-1cm}  (v^{n+1} - \tilde v^{n+1},\ \tilde e^{n+1}) \leq C \dx^{k+2} \|\tilde e^{n+1}\|_{L^2}.\label{eq:th22b4}
\ee
Combining the bounds \eqref{eq:th22b1},\eqref{eq:th22b2},\eqref{eq:th22b3}, \eqref{eq:th22b4},  and the stability bound for $\widetilde S_\dt$, 
we obtain
\be
   \| u^{n+1} - v^{n+1} \|_{L^2} \leq (e^{C\dt} + C \sdt \dx)  \| u^n - v^n \|_{L^2}  + C\dx^{k+1} + C \dt^2.
\ee
For $\dx\leq \sqrt{\ml} \sdt$ we have $C \sdt \dx \leq C\sqrt{\ml} \dt=C_1 \dt$ and therefore 
\be
   \| u^{n+1} - v^{n+1} \|_{L^2} \leq e^{C_1 \dt} \| u^n - v^n \|_{L^2}  + C\dx^{k+1} + C \dt^2.
\ee
We then conclude to the desired error bound.
}\fi


\if{
we first notice that we have, by construction
and using~\eqref{eq:approx_uxxx0} and~\eqref{eq:approx_uxx}:
\be
  & & v^{n+1} = \frac{1}{3} (v^n + S^0_{\dt} v^n  + S^0_{\dt} S^0_\dt v^n)  \nonumber \\
  & & \hspace{1cm}  
   + \dt \frac{\ms^2 (\ms^2)'}{12}  \big(  \Sott \Sott \Sott v^n - \frac{1}{2} \Sott\Sott\Sott\Sott\Sott v^n \big)  \nonumber\\
  & & \hspace{1cm}  
   + \dt \frac{\ms^2 (\ms^2)''}{24} \Sotd \Sotd v^n  + O (\dt^3) \label{eq:modified_SLDG_RK2_vn1}.
\ee
where $O(\dt^3)$ is a term bounded by some constant $C$ that depends of
$\| v_{ttt} \|_{L^\infty}$, $\| v_x^{(3)}\|_{L^\infty}$ and $\| v_x^{(7)} \|_{L^\infty}$.
As already remarked, using the regularity of $S^0_{\dt} v^n$  
we have $S_{\dt} v^n \equiv \Pi S^0_{\dt} v^n = S^0_{\dt} v^n + O(\dx^{k+1})$ as well
as $\Pi S^0_{\dt} \Pi S^0_{\dt} v^n = \Pi S^0_{\dt} S^0_{\dt} v^n  + O(\dx^{k+1})$. Therefore,  
\be
  & & v^{n+1} = \frac{1}{3} (v^n + S_{\dt} v^n  + S_{\dt} S_\dt v^n)  \nonumber \\
  & & \hspace{1cm}  
   + \dt \frac{\ms^2 (\ms^2)'}{12}  \big(  \Sott \Sott \Sott v^n - \frac{1}{2} \Sott\Sott\Sott\Sott\Sott v^n \big)  \nonumber\\
  & & \hspace{1cm}  
   + \dt \frac{\ms^2 (\ms^2)''}{24} \Sotd \Sotd v^n  + O (\dt^3) + O(\dx^{k+1}) \label{eq:modified_SLDG_RK2_vn2}.
\ee
Next, we establish the following estimates:
\be\label{eq:estim1}
  S_{\dt} v^n  = \widetilde S_{\dt} v_n + O(\dx^2 (\dt+ \dx^2) \|v^n\|)
\ee
and
\be\label{eq:estim2}
  S_{\dt} S_{\dt} v^n  = \widetilde S_{\dt} \widetilde S_{\dt} v_n + O(\dx^2 (\dt+ \dx^2) \|v^n\|).
\ee
\centerline{\fbox{\bf {FINIR}}}
}\fi

\if{
\begin{rem}
In the case of 
\be
  & & v_t  - \frac{\sigma^2}{2} v_{xx}  - b v_x + r v = 0, \quad x\in \mO,\ t\in (0,T),\\
  & & v(0,x) = v_0(x), \quad x \in \Omega 
\ee
a direct scheme is
$$
  u^{n+1}  =  e^{-r\dt}\ \Pi \frac{1}{2} \bigg( u^n( \cdot + b \dt  + \sigma \sqrt{\dt}) + u^n( \cdot + b \dt  - \sigma \sqrt{\dt}) \bigg)  
$$
However this scheme is only first-order in time (when $b\neq0$, $\sigma\neq 0$) and we do not know how to directly improve this order. 
\end{rem}
}\fi

\section{Extension to two-dimensional PDEs and splitting strategies} 

\subsection{First-order PDEs - two-dimensional case}\label{sec:4.1}
We aim to extend the previous scheme to treat two-dimensional PDEs, by using splitting strategies and 
one-dimensional solvers of the previous section for advection in the direction of the coordinate axes.

Let $\mO$ be a  square box domain $\mO=[\xumin,\xumax]\times[\xdmin,\xdmax]$ with periodic boundary conditions.
Let us consider a spatial discretization of $\mO$ into cells $I_{i,j}:=I_i\times J_j$ where $I_i$ (resp. $J_j$)
is a cell  discretization of $[\xumin,\xumax]$ (resp.  $[\xdmin,\xdmax]$)
as in the one-dimensional case using $M_1$ (resp. $M_2$) points.
We define the corresponding space of 2d discontinuous Galerkin elements by using the $Q_k$ basis
($v\in Q_k$ if $v(x)=\sum_{i,j\leq k} v_{ij} x_1^i x_2^j$):
\be \label{eq:DGspace_2d}
  & & V_k^{(2)} := \bigg\{v \in L^2(\mO, \R),\, v|_{I_{i,j}} \in Q_k,\ 
  \forall (i,j)
  \bigg\}.
\ee

We consider the case of 
\be\label{eq:b1-b2}
  u_t + b_1(x_1,x_2) u_{x_1} + b_2(x_1,x_2) u_{x_2} = 0, \quad (x_1,x_2)\in \mO.
\ee
The idea, already proposed in \cite{ros-sea-2011} or \cite{cro-meh-vec-2010} is to split the equation into 
\be\label{eq:adv-1}
  u_t + b_1(x_1,x_2) u_{x_1} = 0, \quad (x_1,x_2)\in \mO
\ee
and 
\be\label{eq:adv-2}
  u_t + b_2(x_1,x_2) u_{x_2} = 0, \quad (x_1,x_2)\in \mO.
\ee
Let the corresponding characteristics $X^q_{(x_1,x_2)}(t)$ be defined by :
\begin{itemize}
\item for $q=1$:
  $X^1_{(x_1,x_2)}(t)=(y_1(t),x_2)$ where
 $$\mbox{ $y_1(t)$ is the solution of $\dot y_1(t) = b_1(y_1(t),x_2)$ with $y_1(0)=x_1$,}$$
\item 
  for $q=2$: $X^2_{(x_1,x_2)}(t)=(x_1,y_2(t))$ where
 $$\mbox{$y_2(t)$ is the solution of $\dot y_2(t) = b_2(x_1,y_2(t))$ with $y_2(0)=x_2$.} $$
\end{itemize}
Let $\cE^q_t$ be the corresponding exact evolution operator in the direction of $x_q$.
The exact solution of \eqref{eq:adv-1}, with $q=1$ (resp. \eqref{eq:adv-2}, with $q=2$) satisfies
$$  
  v^{n+1}(x_1,x_2) = v^n(X^q_{(x_1,x_2)}(-\dt)) = \cE^q_\dt(v^n)(x_1,x_2).
$$

We define the discrete evolution operator for \eqref{eq:adv-1}, denoted $\tilde\cT^1_{b_1, \dt}$, so that 
for each fixed Gauss points $x_2=x^i_\ma$ the one-dimensional scheme is used for the evolution in the direction $x_1$.
We define in the same way the operator $\tilde\cT^2_{b_2,\dt}$ for the approximation of \eqref{eq:adv-2}.


\begin{rem} In the case of \eqref{eq:adv-1}
we do not try to compute precisely the $2d$ integrals
\be\label{eq:integrals-1}
  \int_{I_i\times J_j} u^n(X^1_{(x_1,x_2)}(-\dt))\, \varphi_1(x_1) \varphi_2(x_2) dx_1 dx_2,
\ee
where $\varphi_1$ and $\varphi_2$ are polynomial basis functions. The discontinuities of the
integrand are no longer well localized and it would not be possible to obtain easily an accurate approximation 
for~\eqref{eq:integrals-1}.
Rather, the discrete scheme computes a high-order approximation of the following integrals on a full band $[0,1]\times J_j$
\be\label{eq:integrals-2}
  \int_{[0,1]\times J_j} u^n(X^1_{(x_1,x_2)}(-\dt))\, \varphi_1(x_1) \varphi_2(x_2) dx_1 dx_2,
\ee
and this is all what is needed.
\end{rem}


Now, the results of Section 2, in particular Propositions~\ref{prop:2.1} and \ref{prop:stab_b_nonconstant},
can be extended to the operators $\tilde\cT^q_{b_q,\dt}$, $q=1,2$. The difference is now that the consistency estimates 
are typically as follows, for $q=1,2$:
$$
  \| \cE^q_{\dt} \varphi - \tilde\cT^q_{b_q,\dt} \varphi\|_{L^2} \leq C \dt^2 \dx_q^{k+1} \|\varphi\|_{L^2}, \quad \forall \varphi \in V_k^{(2)},
$$
and
$$
  \| \cE^q_{\dt} \psi - \tilde\cT^q_{b_q,\dt} \psi\|_{L^2} \leq C(\psi) \dx_q^{k+1}, \quad \forall \psi \in C^{k+1}.
$$






Let furthermore  $\cE_t$ be the evolution operator for the initial advection problem~\eqref{eq:b1-b2}.
In the case when $b=(b_1,b_2)$ is constant we have  
$$ 
  \cE_{\dt} = \cE^{2}_{\dt} \cE^{1}_{\dt}
$$
and we can therefore approximate the exact evolution 
$\cE_{\dt} v^n$ by $\cT^{2}_{b_2,\dt} \cT^{1}_{b_1,\dt} u^n$ with no error coming from the splitting.

In the following, when there is no ambiguity, we furthermore denote
$$ \cT^{q}_{\dt} = \cT^{q}_{b_q,\dt}  \quad q=1,2.
$$ 

In the case when $b=(b_1,b_2)$ is non-constant,
we recall the following approximations of the exponential $e^{(A+B)\dt}$ 
for $A$ and $B$ matrices and for small $\dt$:
\FULL{
\be \label{eq:exp-trotter}
   e^{(A+B)\dt}= e^{B\dt} e^{A\dt} + O(\dt^2) \quad \mbox{(Trotter spitting)},
\ee
and
\be \label{eq:exp_strang}
   e^{(A+B)h}= e^{B\frac{h}{2}} e^{Ah} e^{B\frac{h}{2}} + O(h^3) \quad \mbox{(Strang's spitting)},
\ee
as well as 
\be 
   e^{(A+B)h} 
  & = & \frac{2}{3} (e^{A\frac{h}{2}} e^{Bh} e^{A\frac{h}{2}} + e^{B\frac{h}{2}} e^{Ah} e^{B\frac{h}{2}})
          - \frac{1}{6}(e^{Bh} e^{Ah}  + e^{Ah} e^{Bh})  \nonumber\\
  &   & \hspace{5cm} +  O(h^4)  
  \label{eq:exp-o3}
\ee
and
\be \label{eq:exp-o4}
   e^{(A+B)h}= \frac{4}{3} e^{A\frac{h}{4}} e^{B\frac{h}{2}} e^{A\frac{h}{2}} e^{B\frac{h}{2}} e^{A\frac{h}{4}} 
    - \frac{1}{3} e^{A\frac{h}{2}} e^{Bh} e^{A\frac{h}{2}} +  O(h^5) 
\ee
(see for instance \cite{bid-2010}).
}
\SHORT{
\be 
   & &
     \label{eq:exp-trotter}
     e^{(A+B)\dt}= e^{B\dt} e^{A\dt} + O(\dt^2) \quad \mbox{(Trotter spitting)},\\
   & & 
     \label{eq:exp_strang}
     e^{(A+B)\dt}= e^{B\frac{\dt}{2}} e^{A\dt} e^{B\frac{\dt}{2}} + O(\dt^3) \quad \mbox{(Strang's spitting)}.
\ee
}
leading us to consider the following splitting approximations  
\newcommand{\dtun}{\dt}
\newcommand{\dtsd}{\frac{\dt}{2}}
\newcommand{\dtsq}{\frac{\dt}{4}}
\be \label{eq:split-o1}
  & &  \cT_{b\dt} \simeq \cT^{2}_{\dt} \cT^{1}_{\dt}  \quad \mbox{(Trotter)} \\
  & & \label{eq:split-o2}
   \cT_{b\dt} \simeq \cT^1_{\dtsd} \cT^2_{\dt} \cT^{1}_{\dtsd}  \quad \mbox{(Strang)}
\ee
of expected consistency error $O(\dt)$ and $O(\dt^2)$ respectively.%
\footnote{Denoting $\tau=T/N$ for $N\geq 1$, and $q\geq0$, if linear operators $A_\tau$ and $B_\tau$ on a normed vector space satisfy 
$A_\tau = B_\tau + O(\tau^{q+1})$, with $\|A_\tau^n\|,\|B_\tau^n\|\leq C$ for all $0\leq n\leq N$, then 
$A_\tau^N = B_\tau^N + O(\tau^{q})$.}
These last two splitting schemes are similar to the ones used in~\cite{Qiu_Shu_2011}.


Following~\cite{ros-sea-2011}, we shall also consider a $3$rd-order splitting scheme 
of Ruth~\cite{Ruth_1983}, a $4$th-order splitting scheme of Forest~\cite{Forest_1987}
(see also Forest and Ruth~\cite{Forest_Ruth_1990}), as well as a $6$th-order splitting of Yoshida~\cite{Yoshida_1993}).

\underline{Ruth's $3$rd-order splitting:} 
\be \label{eq:split-o3}
    \cT_{b \dt} \simeq 
    \cT^1_{c_1 \dtun} \cT^2_{d_1 \dtun} 
    \cT^1_{c_2 \dtun} \cT^2_{d_2 \dtun} 
    \cT^1_{c_3 \dtun} \cT^2_{d_3 \dtun},
\ee
with 
$$
  c_1=7/24,\ c_2= 3/4,\ c_3=-1/24 \quad \mbox{and} \quad 
  d_1=2/3, \ d_2=-2/3,\ d_3=1.
$$

\underline{Forest's 4th-order splitting:}
\be \label{eq:split-o4}
  \cT_{b \dt}  \simeq 
    \cT^1_{\mg_1 \dtsd} 
    \cT^2_{\mg_2 \dtun} 
    \cT^1_{(\mg_1+\mg_2)\dtsd} 
    \cT^2_{\mg_2\dtun}
    \cT^1_{(\mg_1+\mg_2)\dtsd} 
    \cT^2_{\mg_2 \dtun}
    \cT^1_{\mg_1 \dtsd},
\ee
with
$$
  \mg_1:=\frac{1}{2-2^{1/3}}  \quad \mbox{and} \quad \mg_2= - \frac{2^{1/3}}{2-2^{1/3}}.
$$

\underline{Yoshida's 6th-order splitting:} 
\be
   \label{eq:split-o6}
   \cT_{b \dt}  \simeq  \cT^{4th}_{y_1 \dt} \cT^{4th}_{y_2 \dt} \cT^{4th}_{y_1 \dt},
\ee
where $\cT^{4th}_{\dt}$ denotes the previous Forest's $4$th-order approximation method,
$$ 
  y_1:=\frac{1}{2-2^{1/5}} \quad \mbox{and}\quad  y_2:=-\frac{2^{1/5}}{2-2^{1/5}}.
$$

\begin{rem}
Stability in the $L^2$-norm is then easily obtained.
Indeed, we have the $L^2$-stability of the one-directional advection operators $\cT^k_{\dt}$,
that is, for variable coefficients
\be\label{eq:tau_stab_exp}
  \|\cT^k_{\dt} u\|_{L^2}\leq e^{c\dt} \| u\|_{L^2}
\ee
for some constant $c$.
Then, for instance for the Trotter splitting, we have
$\|\cT^1_{\dt} \cT^2_{\dt} u\|_{L^2}\leq e^{2c\dt} \| u\|_{L^2}$, which gives the $L^2$ stability result
\be 
  \|(\cT^1_{\dt} \cT^2_{\dt})^n u\|_{L^2}\leq e^{2c t_n} \| u\|_{L^2}.
\ee
In the same way any finite product of operators of the form of $\cT^k_{\ma_k \dt}$ (or any convex 
combination of such products) would lead to stable schemes.
\end{rem}

Hence the results of Section 2 can be extended:
for $\ma=1,2,3,4$ and $6$ corresponding to the splittings
\eqref{eq:split-o1}, \eqref{eq:split-o2}, \eqref{eq:split-o3}, \eqref{eq:split-o4} and \eqref{eq:split-o6}  respectively,
for regular solutions, the one time step error will be of order 
\be\label{eq:one-time-step-error}
  O(\dt^{\ma+1}) + O(\dx^{k+1}),
\ee
and the convergence error bound after $N$ time steps will be of order 
\be\label{eq:N-time-step-error}
  O(\dt^\ma) + O(\frac{\dx^{k+1}}{\dt}).
\ee

\subsection{Second-order PDEs - two-dimensional case}\label{sec:4.2}
We consider the case of
\be 
   & & u_t - \ud Tr(\sigma(x) \sigma(x)^T D^2 u) + b(x)\cdot \nabla u =  f(t,x), \quad x\in \mO,\ t\in(0,T) \nonumber \\
   & & 
    \label{eq:o2} 
\ee
(with initial condition $u(0,x)=u_0(x)$),
where $\sigma(x)\in \R^{2\times 2}$ and $Tr(A)$ denotes the trace of the matrix $A$.

We introduce the following decomposition into the direction of diffusions represented by the column vectors of the matrix $\ms$
(similar decompositions have been used by Kushner and Dupuis~\cite{kus-dup-92}, Menaldi~\cite{Menaldi_1989}, 
Camilli and Falcone~\cite{Camilli_Falcone_95}, 
Debrabant and Jakobsen~\cite{Debrabant_Jakobsen_2012}, etc.):
$$
  \sigma\sigma^T= \sum_{q=1}^2 \ms_q \ms_q^T, \quad 
   \mbox{where}\ \ms_q:=\VECT{\sigma_{1,q} \\ \sigma_{2,q}}.
$$

Setting $B_1=\VECT{b_1\\0}$ and $B_2=\VECT{0\\b_2}$, we write~\eqref{eq:o2} as follows:
\be \label{eq:o2_all}
   u_t + \sum_{q=1,2} \bigg( - \frac{1}{2}  Tr(\ms_q \ms_q^T D^2 u) + B_q \cdot \nabla u \bigg)  = f(t,x). 
\ee

Let us first consider the one-directional problem (one direction of diffusion):
\be \label{eq:o2_q}
   u_t - \frac{1}{2} Tr (\ms_q \ms_q^T D^2 u) + B_q \cdot\nabla u = 0.
\ee 
For this subproblem we consider weak Taylor schemes exactly as for the one-dimensional 
SLDG-1 and SLDG-2 schemes \eqref{eq:y_euler}-\eqref{eq:SLDG_RK1} 
and \eqref{eq:SLDG-2}-\eqref{eq:y_platen}.
Indeed these approximations are known to be also of order $1$ and $2$ in time for \eqref{eq:o2_q} in any dimension \cite{klo-pla-95}.

It remains to give the definition of a scheme, of sufficient order,
for the approximation in two dimensions for terms of the form
\be\label{eq:proj-2d}
  \Pi(u^n(y^q_\dt(\cdot))) 
\ee
where $\Pi$ is the projection on $V^{(2)}_k$ and $y^q_\dt(x)$ is now a vector of $\R^2$.

\begin{rem}\label{rem:general_splittings}
\MODIF{
In view of the definition of the characteristics \eqref{eq:y_euler} or \eqref{eq:y_platen},
}
a typical problem is to compute accurately the projection on $V^{(2)}_k$ of a function of the form
\be\label{eq:proj-2d-typical}
   (x_1,x_2) \converge u^n(f_1(h,x_1,x_2),f_2(h,x_1,x_2)),
\ee 
with $h=\sqrt{\dt}$,
where $f_1$ and $f_2$ are regular functions with known expressions, and such that
\be\label{eq:fassump}
  \mbox{$f_1(0,x_1,x_2)= x_1$ and $f_2(0,x_1,x_2)=x_2$.}
\ee
A high-order approximation of the term \eqref{eq:proj-2d}, or \eqref{eq:proj-2d-typical} in the general case
can be obtained by using the PDE satisfied by $v(s,x_1,x_2):=u^n(f_1(s,x_1,x_2),f_2(s,x_1,x_2))$.

\MODIF{
More precisely, 
assuming that $u^n$ is a regular function, we observe that 
$\partial_s v = \<\partial_s f,\ \nabla u^n(f_1,f_2)\>$ and $\nabla v= Df^T\, \nabla u^n(f_1,f_2)$ 
(where $Df:=(\frac{\partial f_i}{\partial x_j})$ and $\nabla u= (\frac{\partial u}{\partial x_i})$).
Therefore 
$\partial_s v = \<\partial_s f,\ (Df^T)^{-1} \nabla v\> = \<Df^{-1} \partial_s f,\,\nabla v\>$ and $v$ is solution of the PDE
\begin{subequations}
\be
 & &  \partial_s v - \<Df^{-1} \partial_s f,\,\nabla v\> = 0, \quad s>0, \\
 & &  v(0,x_1,x_2)=u^n(x_1,x_2) 
\ee
\end{subequations}
(the matrix inverse $Df(s,x_1,x_2)^{-1}$ is well defined for small $s\geq 0$
since by the assumptions \eqref{eq:fassump} we have $Df(0,x_1,x_2)=Id$).
Then we have a problem of the form \eqref{eq:b1-b2} and we can apply the splitting approaches of Section~\ref{sec:4.1} to obtain
a high-order approximation of \eqref{eq:proj-2d-typical} on a DG basis.
}
\end{rem}

\MODIF{
\begin{rem}\label{rem:special_splittings}
In the present work we will consider only numerical examples involving terms of the form $\Pi u^n(f_1(h,x_1,x_2),x_2)$ 
or  $\Pi u^n(x_1,f_2(h,x_1,x_2))$ (i.e. $f_2(h,x_1,x_2)\equiv x_2$, or $f_1(h,x_1,x_2)\equiv x_1$),
or of the form $\Pi u^n(f_1(h,x_1),f_2(h,x_2))$ with regular functions $f_1$ and $f_2$ and $h=\sqrt{\dt}$.
For such cases, the one-dimensional discretization can be extended to two dimensions by straightforward splitting.
\end{rem}
}

Finally, for the general case of \eqref{eq:o2}, we define the scheme by using Strang's splitting of the one time-step evolution operators 
for \eqref{eq:o2_q} and by adding the correction \eqref{eq:whx-correction} for the source term.


\if{
\be\label{eq:y_euler}
  \mg^q_h(x)  & := & x + b(x)h +   q  \ms(x) \sqrt{h}.
\ee

Our SLDG-1 scheme, corresponding to a first-order (weak Euler scheme), is defined by 
\be \label{eq:SLDG_RK1}
  u^{n+1}\equiv S^{(1)}_\dt u^n := \Pi \bigg( \sum_{q=\pm 1} \ma_q u^n(y^q_\dt(\cdot))\bigg)
\ee
with weights $\ma_{-1}=\ma_1=\frac{1}{2}$ and characteristics $y^q_h=\mg^q_h$.

Our SLDG-2 scheme for variable coefficients, corresponding to the second-order Platen's scheme, is defined by
\be \label{eq:SLDG-2-gen}
  u^{n+1}\equiv S^{(2)}_\dt u^n := \Pi \bigg( \sum_{-1\leq q\leq1 } \ma_q u^n(y^q_\dt(\cdot))\bigg)
\ee
with weights $\ma_{-1}=\ma_1=\frac{1}{6}$ and $\ma_0=\frac{2}{3}$ and characteristics $y^q_h=y^q_h(x)$ defined by:
\be
  y^q_h(x) 
  &  = &  x + \frac{1}{2}(b(\mg^{\sqrt{3} q}_h)+b)\,h   
    \label{eq:y_platen}
  \\
  & & 
  + \frac{1}{4}\bigg[\big( \ms(\mg^1_h) + \ms(\mg^{-1}_h)+ 2\ms\big) \sqrt{3}\ q 
  + \big(\ms(\mg^1_h) - \ms(\mg^{-1}_h)\big) (3q^2-1) \bigg]
    \sqrt{h}.
    \nonumber
\ee
}\fi

\if{
%
%
\paragraph{\bf Case of a constant diffusion matrix $\sigma$:} 
In that case, $\mS_k(x)\equiv \mS_k$ and equation \eqref{eq:o2_all} can be splitted exactly into equations \eqref{eq:o2_q}
for $k=1,\dots,d$.
The expected global error is of order $O(\dt^p)$,
depending on which SLDG-p scheme we use, plus the spatial error of order $O(\frac{\dx^{k+1}}{\dt})$.

On the other hand, for a given constant vector $\mS\in\R^d$,
the evaluation of $\Pi (u( \cdot + \mS\sdt))\equiv \cT^{\mS}_{\sdt} u$ can then be done exactly by using  
by an exact splitting as in Section \ref{sec:2.3}, and using exact quadrature rule for each dimension. An illustration is given 
in the Numerical section (see~\Exnin).

\medskip

\paragraph{\bf Case of a non-constant diffusion matrix $\sigma(x)$:}
In the more general case, there is an error $O(\dt)$ coming from the splitting. 
Assuming that $d=2$, instead of $u^{n+1} = S^{\mS_2}_\dt \ S^{\mS_1}_{\dt}  u^n $ 
we can consider Strang's splitting  
\be\label{eq:2D-strang}
  u^{n+1} =   S^{\mS_2}_{\frac{1}{2}\dt}\,  S^{\mS_1}_\dt \, S^{\mS_2}_{\frac{1}{2}\dt}  u^n 
\ee
(with consistency error of order $O(\dt^2)$),
where each $S^{\mS_k}_\dt$ is computed by using the \textit{modified} SLDG-2 scheme (consistency error of order $O(\dt^{2})$).
In that case, 
the global expected error is of order $O(\dt^{2}) +  O(\frac{\dx^{k+1}}{\dt})$.
Adaptation of Strang's splitting when $d\geq 2$ can also be considered.

\medskip

\paragraph{\bf General convection-diffusion equation.} 

Let us consider 
$$ u_t - \frac{1}{2} Tr(A D^2 u) + b \cdot \nabla u + r u = 0, $$
with $A=\ms \ms^T = \sum_k \mS_k \mS_k^T$ and $B=(b_1,\dots,b_d)^T$.
In the constant coefficient case, the spatial differential operators commute and therefore we can use Trotter's splitting:
$$
  u^{n+1} = e^{-r\dt}
    S^{\mS_d}_\dt \ \cdots\ S^{\mS_2}_\dt \ S^{\mS_1}_\dt  \cT^1_{b_1\dt} \cdots \cT^d_{b_d\dt} u^n
$$
where $S$ is one of the SLDG-p schemes ($p=1,2,\dots$).

For non-constant coefficients,
we can use a generalization of the previous splitting techniques.

...

$(i)$ 
We first consider the case of RK1 ($q=1$).
Let $S^0_\dt$ correspond to the direct semi-Lagrangian scheme 
\be
  (S^0_\dt v^n)(x) := e^{-r\dt} \frac{1}{2} \bigg( v^n( x + b \dt  + a \sqrt{\dt}) + v^n( x + b \dt  - a \sqrt{\dt}) \bigg).
\ee
and let us first show the following  consistency error bound:
\be \label{eq:eps1}
 \| \frac{1}{\dt}  (v^{n+1} - \cS^0_\dt v^n) \|_\infty \leq C_1 \dt.
\ee
where $C$ is a constant that depends only of $\max_{t\in[0,T]}\|\frac{\partial^2 v}{\partial t^2}(t,.)\|_\infty$.

By Taylor expansion of $v^n$, for $\theta>0$ small, we have
\be \label{eq:tayl}
v^n(x+ b \mt + \eps a \sth) 
  & = &  v^n(x)+(\eps a \sth + \mt b) \dfrac{d v^n}{dx}(x)+ \frac {1}{2!}(\epsilon^2 a^2 \theta + 2 a b \epsilon \theta^{\frac{3}{2}})\dfrac{\partial^2 v^n}{\partial x^2}(x) \\
  &   &  + \frac {1}{3!} \epsilon^3 a^3 \theta^{\frac{3}{2}} \dfrac{\partial^3 v^n}{\partial x^3}(x) +\mathcal{O}(\theta^2) \nonumber
\ee
Hence
\begin{eqnarray}
 e^{r\dt} (S^0_\mt v^n)(x) & = &   \frac{1}{2}\sum_{\epsilon=\pm1} v(t_n,  x+ b \mt + \eps a \sth)  \\
  &=& v(t_n,x)+ \theta \Big[ b \dfrac{\partial v}{\partial x}(t_n,x) + \frac{a^2}{2} \dfrac{\partial^2 v}{\partial x^2}(t_n,x) \Big] +\mathcal{O}(\theta^2) \nonumber 
\end{eqnarray}
In particular, taking $\mt=\dt$, we get
\be 
 \frac{e^{r\dt} (S^0_\dt v^n)(x) - v^n(x)}{\dt}  & = & 
    b \dfrac{\partial v}{\partial x}(t_n,x) + \frac{a^2}{2} \dfrac{\partial^2 v}{\partial x^2}(t_n,x)  +\mathcal{O}(\dt) \nonumber 
\ee

Note that we can also write:
\begin{eqnarray}
({\dfrac{\partial v}{\partial t} + rv})\bigg|_{t=t_n} = \frac{d}{dt}(e^{rt}v)e^{-rt}\bigg|_{t=t_n} 
  & = & \Big( \frac{e^{rt_{n+1}}v^{n+1}-e^{rt_n}v^n}{\dt} \Big) e^{-rt_n}  + \mathcal{O}(\dt) \nonumber \\
  & = & \Big( \frac{e^{r \dt}v^{n+1}-v^n}{\dt} \Big) + \mathcal{O}(\dt)
\end{eqnarray}
Therefore, using the fact that $v$ is a solution of the diffusion equation and combining the previous equalities we obtain
\beno
  \frac{e^{r \dt} (v^{n+1}(x) - S^0_\dt v^n (x))}{\dt}  = \cO(\dt)
\eeno
which proves the estimate~\eqref{eq:eps1}.

Next, we define the consistency error using the true scheme  $S_\dt$: 
\be\label{eq:epsn}
  \eps^{n}(x) = \frac{1}{\dt}  (v^{n+1}(x) - S_\dt v^n (x)). 
\ee
Since $v^n$ is regular, we observe that 
\be\label{eq:bound2}
  \|\cS_\dt v^n - \cS^0_\dt v^n\|_2 & = & \|\Pi(\cS^0_\dt v^n) - \cS^0_\dt v^n\|_2  \nonumber \\
   & \leq &  C_k(\cS^0_\dt v^n) |\mO|^{1/2}\ \dx^{k+1} \nonumber \\
   & \leq &  C_k( v^n) |\mO|^{1/2}\ \dx^{k+1}
\ee
where $C_k$ is defined in Lemma~\ref{lem:proj}.
We deduce from \eqref{eq:eps1} and \eqref{eq:bound2} the following bound
\be\label{eq:eps1_2}
 \| \eps^n \|_2 \leq C_1 (\frac{\dx^{k+1}}{\dt} + \dt).
\ee
where $C_1$ depends only of 
$\max_{t\in[0,T]}\|\frac{\partial^2}{\partial t^2}v(t,.)\|_\infty$,
$\max_{t\in[0,T]}\|\frac{\partial^{(k+1)}}{\partial x^{k+1}}v(t,.)\|_\infty$ and $|\mO|$.

The end of the proof follows the lines of the proof of Theorem~\ref{th:advect_conv}.
We have $u^{n+1}= \cS_\dt u^n$ and $v^{n+1} = \cS_\dt v^n + \dt \eps^n$.
Hence $ \| u^{n+1} - v^{n+1} \|_2 \leq  \| \cS_\dt (u^n - v^n) \|_2 + \dt \|\eps^n\|_2$, 
from which we deduce, using~\eqref{eq:eps1_2}, the $L^2$ norm stability of $\cS_\dt$, and Lemma~\ref{lem:proj}:
\beno  
  \| u^{n} - v^{n} \|_2
    &  \leq  & \|u^0 - v^0 \|_2 + T C_1 (\frac{\dx^{k+1}}{\dt} + \dt) \\
    &  \leq  &  C_1 (T+\dt)  (\frac{\dx^{k+1}}{\dt} + \dt) 
\eeno
Hence we obtain the desired bound where $C$ is a constant such that $C\geq C_1 (T+\dt)$.

$(ii)$ Now, we consider the case $p=2$.
The new part is to prove that the following consistency bound holds:
\be \label{eq:eps_RK2}
 \| \frac{1}{\dt}  (v^{n+1} - \cS_\dt^{0,RK_2} v^n) \|_\infty \leq C_2 \dt^2.
\ee
}\fi

\section{Numerical examples}\label{sec:num}
The first \MODIF{three} examples are devoted to advection problems, while the other examples concern second-order equations.

\MODIF{
We recall that $N$ is the number of time steps (and $\dt=T/N$), and $M$ is the number of spatial mesh points in 
the one-dimensional case (resp. $M_1,M_2$ for two-dimensional cases).

Unless otherwise specified, the characteristics are one-dimensional and are always computed exactly
(see added sentence in Section 5 before the first example).

Computations were performed on a DELL Latitude E6220, Intel Core i5, 2.50GHz, 4GO RAM, 
with Linux OS, 32-bit, using GNU {\sc C++}.
}

\medskip


\if{
\noindent
{\bf \Exone}
In this test we consider an advection equation in one dimension with smooth initial data:
\be\label{eq:ex1a}
  & &  v_t + v_x =0  \qquad x\in (0,1),\ t\in(0,T), \\
  & & v(0,x)=v_0(x) \qquad x\in(0,1),
\ee
using periodic boundary conditions on $\Omega = (0,1)$, terminal time $T=1$,
\FULL{
and together with one of the following smooth initial data:
\be \label{eq:ex1_u01}
 & & v_0(x):= \sin(2\pi x),
\ee
or 
\be\label{eq:ex1_u02}
 & & v_0(x) :=\sin(2\pi (x + \sin(2 \pi x))).
\ee
}
\SHORT{
and the following smooth initial data:
\be\label{eq:ex1_u02}
 & & v_0(x) :=\sin(2\pi (x + \sin(2 \pi x))).
\ee
}

%

\FULL{Results for the $L^2$ error and corresponding orders 
are given in Tables~\ref{tab:ex1.1} and~\ref{tab:ex1.2} for initial data \eqref{eq:ex1_u01} and \eqref{eq:ex1_u02} respectively.}
\SHORT{Results for the $L^2$ error and corresponding orders 
are given in Table~\ref{tab:ex1.2}.}

In this test we have fixed the CFL number to be $2.3$, thus
$\dt \equiv C \dx$ for some constant $C$, and the expected error is
$O(\frac{\dx^{k+1}}{\dt})=O(\dx^{k})$.

We see that the expected order of convergence is well recovered. We have observed that the
$L^\infty$ and $L^1$ errors give similar result here.
In this example the scheme and hence the results are similar to \cite[Example 1]{Qiu_Shu_2011} 
excepted for the initial data.

\begin{table}[!hbtp]
\begin{center}
\begin{tabular}{|c|cc|cc|cc|cc|}
\hline
   & \multicolumn{2}{|c|}{k=1} & \multicolumn{2}{|c|}{k=2} & \multicolumn{2}{|c|}{k=3} & \multicolumn{2}{|c|}{k=4} \\
\hline
 Mesh $M$ &  error   &order & error    & order & error    & order & error    & order \\
\hline
 23       & 8.01E-2 & -    & 3.52E-3 & -     & 4.09E-04 & -     & 3.57E-05 & -     \\
 46       & 1.44E-2 & 2.47 & 4.05E-4 & 3.11  & 2.16E-05 & 4.24  & 1.28E-06 & 4.80  \\
 92       & 2.42E-3 & 2.57 & 5.31E-5 & 2.93  & 1.25E-06 & 4.10  & 4.15E-08 & 4.94  \\
184       & 4.79E-4 & 2.33 & 6.75E-6 & 2.97  & 7.81E-08 & 4.00  & 1.31E-09 & 4.98  \\
368       & 1.09E-4 & 2.12 & 8.47E-7 & 2.99  & 4.85E-09 & 4.00  & 4.10E-11 & 4.99  \\
\hline
\end{tabular}
\end{center}
\caption{\label{tab:ex1.2} 
\textit{(\Exone) advection test,} $L^2$ errors with initial data $v_0(x)=\sin(2\pi (x + \sin(2 \pi x)))$ and CFL$=2.3$.
}
\end{table}

}\fi


\medskip

\noindent
{\bf \Extwo.}
We consider an advection equation with non-constant advection term
\be\label{eq:ex2}
 & & v_t  + b(x) v_x= 0, \qquad  x \in (0,1), \ t \in(0,T), \\
 & & v(0,x)= sin(2 \pi x),  \qquad x\in (0,1),
\ee
and 
\be
  b(x) := C_0 + C_1 \sin(2 \pi x), \qquad \mbox{with $C_0=1$ and $C_1:=0.8$}
\ee
together with periodic boundary conditions on $(0,1)$.
The exact solution is given by $v(t,x)= \sin(2 \pi y_x(-t))$, where
$$
  y_x(-t)=\frac{1}{\pi} \mbox{atan}\bigg(-r + \mbox{tan}\bigg( \mbox{atan}( \frac{\tan(\pi x) + r}{a})-C_0 \pi a t \bigg) \bigg)
$$
with $r:=\frac{C_1}{C_0}$ and  $a:= \sqrt{1-r^2}$.

The results are given in Table~\ref{tab:ex2} for  $\dt\sim\dx$ with fixed CFL$=1.8$ and terminal time $T=1.3$.
(Here the CFL corresponds to $\| b\|_\infty \frac{\dt}{\dx}$.)
The numerical error behaves approximatively one order better than
the expected one when $\dt=\ml \dx$, that is of the order of $O(\frac{\dx^{k+1}}{\dt})\equiv O(\dx^k)$. 
Super-convergence results can be explained in some cases for other DG methods \cite{yan-shu-12}.

\begin{table}[!hbtp]
\begin{center}
\begin{tabular}{|cc|cc|cc|cc|cc|}
\hline
\multicolumn{2}{|c|}{$L^2$ error}  
   & \multicolumn{2}{|c|}{$k=1$} & \multicolumn{2}{|c|}{$k=2$} & \multicolumn{2}{|c|}{$k=3$} & \multicolumn{2}{|c|}{$k=4$} \\
\hline
  $M$ & $N$  &  error   &order & error    & order & error    & order & error    & order\\
\hline
  10  &  10  & 1.95E-01 &    - & 3.45E-02 &    -  & 1.45E-02 &    -  & 7.83E-03 &    - \\
  20  &  20  & 2.67E-02 & 1.93 & 6.06E-03 & 2.50  & 1.38E-03 & 3.39  & 2.33E-04 & 5.07 \\
  40  &  40  & 7.80E-03 & 1.77 & 6.39E-04 & 3.24  & 3.22E-05 & 5.42  & 4.31E-06 & 5.75 \\
  80  &  80  & 1.47E-03 & 2.40 & 3.62E-05 & 4.13  & 1.52E-06 & 4.40  & 7.74E-08 & 5.80 \\
 160  & 160  & 2.27E-04 & 2.69 & 3.31E-06 & 3.45  & 7.13E-08 & 4.41  & 2.48E-09 & 4.96 \\
 320  & 320  & 3.92E-05 & 2.53 & 4.03E-07 & 3.04  & 3.92E-09 & 4.18  & 8.03E-11 & 4.95 \\
\hline
\end{tabular}
\end{center}
\caption{\label{tab:ex2} \textit{(\Extwo) non-constant advection,} $\dt\sim\dx$ and CFL$=1.8$, $T=1.3$.
}
\end{table}

\if{
\begin{table}[!hbtp]
\begin{center}
\begin{tabular}{c c | c c|c c|c c}
\multicolumn{8}{c}{\uline{$k=1, \text{and } \dx \sim \dt $}}\\
$M$ & $N$ & $L^1$-Error & order & $L^2$-Error & order & $L^\infty$-Error & order\\
\hline
\hline
     10    &     10    & 7.123E-01 &     -     & 1.019E-01 &     -     & 1.951E-01 &     -      \\
     20    &     20    & 3.353E-01 & 1.087     & 2.671E-02 & 1.932     & 6.892E-02 & 1.501      \\
     40    &     40    & 9.412E-02 & 1.833     & 7.802E-03 & 1.775     & 2.124E-02 & 1.698      \\
     80    &     80    & 3.768E-02 & 1.321     & 1.472E-03 & 2.406     & 4.709E-03 & 2.174      \\
    160    &    160    & 7.241E-03 & 2.380     & 2.277E-04 & 2.693     & 8.056E-04 & 2.547      \\
    320    &    320    & 1.374E-03 & 2.398     & 3.922E-05 & 2.537     & 1.410E-04 & 2.514      \\
\\
\multicolumn{8}{c}{\uline{$k=2, \text{and } \dx \sim \dt $}}\\
$M$ & $N$ & $L^1$-Error & order & $L^2$-Error & order & $L^\infty$-Error & order\\
\hline
\hline
     10    &     10    & 2.889E-01 &     -     & 3.452E-02 &     -     & 7.165E-02 &     -      \\
     20    &     20    & 7.150E-02 & 2.015     & 6.063E-03 & 2.509     & 1.656E-02 & 2.113      \\
     40    &     40    & 1.277E-02 & 2.485     & 6.390E-04 & 3.246     & 2.274E-03 & 2.864      \\
     80    &     80    & 9.637E-04 & 3.728     & 3.627E-05 & 4.139     & 1.413E-04 & 4.008      \\
    160    &    160    & 9.471E-05 & 3.347     & 3.318E-06 & 3.451     & 1.297E-05 & 3.446      \\
    320    &    320    & 1.127E-05 & 3.072     & 4.031E-07 & 3.041     & 1.568E-06 & 3.048      \\
\\
\multicolumn{8}{c}{\uline{$k=3, \text{and } \dx \sim \dt $}}\\
$M$ & $N$ & $L^1$-Error & order & $L^2$-Error & order & $L^\infty$-Error & order\\
\hline
\hline
     10    &     10    & 1.033E-01 &     -     & 1.456E-02 &     -     & 2.950E-02 &     -      \\
     20    &     20    & 2.644E-02 & 1.966     & 1.388E-03 & 3.391     & 4.725E-03 & 2.642      \\
     40    &     40    & 1.082E-03 & 4.610     & 3.223E-05 & 5.429     & 1.464E-04 & 5.012      \\
     80    &     80    & 5.286E-05 & 4.356     & 1.522E-06 & 4.405     & 6.345E-06 & 4.528      \\
    160    &    160    & 2.672E-06 & 4.306     & 7.137E-08 & 4.414     & 2.953E-07 & 4.425      \\
    320    &    320    & 1.512E-07 & 4.143     & 3.925E-09 & 4.184     & 1.625E-08 & 4.184      \\
\\
\multicolumn{8}{c}{\uline{$k=4, \text{and } \dx \sim \dt $}}\\
$M$ & $N$ & $L^1$-Error & order & $L^2$-Error & order & $L^\infty$-Error & order\\
\hline
\hline
     10    &     10    & 7.105E-02 &     -     & 7.839E-03 &     -     & 1.855E-02 &     -      \\
     20    &     20    & 4.251E-03 & 4.063     & 2.332E-04 & 5.071     & 7.709E-04 & 4.589      \\
     40    &     40    & 2.264E-04 & 4.231     & 4.318E-06 & 5.755     & 1.944E-05 & 5.309      \\
     80    &     80    & 4.019E-06 & 5.815     & 7.749E-08 & 5.800     & 3.730E-07 & 5.704      \\
    160    &    160    & 1.505E-07 & 4.740     & 2.486E-09 & 4.962     & 1.266E-08 & 4.881      \\
    320    &    320    & 5.412E-09 & 4.797     & 8.033E-11 & 4.952     & 4.023E-10 & 4.975      \\
\end{tabular}
\end{center}
\caption{\label{tab:ex2} (\Extwo) $u_t + b(x) u_x=0$ with non-constant $b(x)$, and $\dt\equiv \dx$, $CFL=1.8$.}
\end{table}
}\fi

}

\medskip

\noindent{\bf \Exfour\ (2D advection with non-constant coefficients).}   
We consider the following rotation example of a "bump":
\beno
   & &  u_t + 2\pi(-x_2,x_1)\cdot\nabla u = 0, \qquad x=(x_1,x_2)\in \mO,\ t\in(0,T),\\
   & &   u(0,x)= 1- e^{-20((x_1-1)^2 + x_2^2-r_0^2)},
\eeno
with $\mO:=(-2,2)^2$, $r_0=0.25$ and terminal time $T=0.9$.
Since $b(x_1,x_2)=2\pi(-x_2,x_1)$ is non-constant, Trotter's splitting is no longer exact.

In Table~\ref{tab:ex4}, 
we test and compare the splitting algorithms as described in subsection 2.3, from order $2$ to $6$
(Strang's splitting, Forest's 4th-order splitting and Yoshida's 6th-order splittings,
tested with $k=2,4$, and $k=6$ respectively), using $M_1=M_2=M$ spatial mesh points.
Trotter's splitting error, not represented in Table~\eqref{tab:ex4}, is of order $1$. 
We have avoided taking the particular case of $T=1$ (full turn)
because it gives better numerical results but prevents proper understanding of the order of the method.

In this example, 
the initial datum is sufficiently close to $1$ outside a ball of radius $1.5$,
so that the error coming from the boundary treatment is negligible.


\begin{small}
\begin{table}[!hbtp]
\begin{center}
\MODIF{
\begin{tabular}{|cc|ccc|ccc|ccc|}
\hline
\multicolumn{2}{|c|}{$L^2$ error}  
   & \multicolumn{3}{|c|}{Strang (with $k=2$)} 
   & \multicolumn{3}{|c|}{Forest (with $k=4$)} & \multicolumn{3}{|c|}{Yoshida (with $k=6$)} \\
\hline
  $N$ & $M$ &  error & order & cpu(s)  &   error & order & cpu(s) &   error & order  & cpu(s) \\
\hline
    10 &    10  & 2.91E-01 &  -     &   0.004 & 1.66E-01 &    -   &   0.01 &   1.81E-02 &    -   &   0.07 \\
    20 &    20  & 6.62E-02 &  2.13  &   0.012 & 1.01E-02 &  4.04  &   0.03 &   2.45E-04 &  6.21  &   0.26 \\
    40 &    40  & 1.60E-02 &  2.05  &   0.032 & 6.24E-04 &  4.01  &   0.22 &   3.64E-06 &  6.07  &   1.65 \\
    80 &    80  & 3.99E-03 &  2.01  &   0.272 & 3.89E-05 &  4.00  &   2.04 &   5.61E-08 &  6.02  &  15.06 \\
   160 &   160  & 9.96E-04 &  2.00  &   2.844 & 2.43E-06 &  4.00  &  18.25 &   1.03E-09 &  5.77  & 120.98 \\
\hline
\end{tabular}
}
\end{center}
\caption{\label{tab:ex4} \textit{(\Exfour), 2D rotation,} $L^2$ errors at time $T=0.9$, using $M\times M$ grid points
and splittings of order $2,4$ and $6$.}
\end{table}
\end{small}


\medskip

\noindent{\bf \Exfive\ (2D deformation with non-constant coefficients)}   
In this example, close to the one in for instance Qiu and Shu~\cite[Example 5]{Qiu_Shu_2011}, the advection term is non-constant  
\beno
   & &  u_t -\bigg(g(t)\cos(\frac{x^2}{2})\sin(y)\bigg) u_{x} + \bigg(g(t)\cos(\frac{y^2}{2})\sin(x)\bigg) u_{y} = 0, \nonumber \\
   & & \hspace{5cm} \qquad (x,y)\in \mO,\ t\in(0,T),
\eeno
with $\mO:=(-2,2)^2$, $T=1$ and same initial datum as in Example~4.
Here we furthermore  consider $g(t):=1$ for $t\in[0,\frac{T}{2}]$ and then $g(t):=-1$ for $t\in]\frac{T}{2}, T]$, 
so that the exact solution after time $T$ is $u(T,x,y)=u_0(x,y)$.

In Table~\ref{tab:ex5n}, 
we test and compare the splitting algorithms of orders $2$,$4$ and  $6$
(Strang's, Forest's and Yoshida's splittings),
using polynomials of degree $k=2$, $4$ and $6$ respectively.
The cpu times are also given in seconds.


\begin{small}
\begin{table}[!hbtp]
\begin{center}
\begin{tabular}{|cc|ccc|ccc|ccc|}
\hline
\multicolumn{2}{|c|}{$L^2$ error}  
   & \multicolumn{3}{|c|}{Strang (with $k=2$)} 
   & \multicolumn{3}{|c|}{Forest (with $k=4$)} & \multicolumn{3}{|c|}{Yoshida (with $k=6$)} \\
\hline
 $N$ & $M$  &  error & order  & cpu(s)  &   error & order  & cpu(s)  &   error & order & cpu(s) \\
\hline
10 &    10  &   1.28E-01 &   -    &  0.005 &    7.82E-03 &   -     & 0.08 &    7.70E-04  &  -    &   0.85\\
20 &    20  &   1.45E-02 &  3.14  &  0.034 &    2.78E-04 &  4.81   & 0.36 &    6.60E-06  & 6.87  &   3.65\\
40 &    40  &   1.44E-03 &  3.33  &  0.104 &    9.06E-06 &  4.94   & 1.58 &    3.32E-08  & 7.64  &  16.20\\
80 &    80  &   1.66E-04 &  3.12  &  0.620 &    3.30E-07 &  4.78   & 7.73 &    2.71E-10  & 6.94  & 140.41\\
\hline
\end{tabular}
\end{center}
\caption{\label{tab:ex5n} \textit{(\Exfive) 2D deformation}, 
$L^2$ errors at time $T=1$, using $M\times M$ grid points and splittings of order $2,4$ and $6$.}
\end{table}
\end{small}

\comment{\noindent{\bf Example 5  (2D - discontinuous data) : MAY BE ?}}


\noindent{\bf \Exsix\ (1D convection diffusion).}
Now, we consider the diffusion equation
\begin{eqnarray}\label{eq:ex3}
 & & v_t  - \frac{1}{2}\sigma^2 v_{xx} + b v_x = 0, \quad \forall x \in \Omega,\ t \in (0,T) \\
 & & v(0,x)=\cos(2\pi x) + \frac{1}{2} \cos(4\pi x),\quad x\in \mO
\end{eqnarray}
together with periodic boundary conditions on $\Omega=(0,1)$, with constants $\sigma=0.1$, $b=0.3$, and $T=0.2$.
The exact solution is given by 
$$v(t,x)=\sum_{k=1,2} c_k \exp(-2\sigma^2 k^2 \pi^2 t)  \cos (2 k\pi (x-bt)),$$
with $c_1=1$ and $c_2=\frac{1}{2}$. 

Since the operators $\frac{1}{2}\sigma^2 \partial^2_{x}$ and $b \partial_x$ commute, 
we use the simple scheme 
$$
   u^{n+1} = S^\sigma_\dt \cT_{b\dt} u^n.
$$

In Table~\ref{tab:ex5.2b} we study the orders of the SLDG-RKp schemes when $\dt \sim \dx$ and $p\in\{1,2,3\}$. 
The orders are as expected.

We also give in Table~\ref{tab:ex5.3b} the errors when taking larger time steps ($\dt \gg \dx$), 
still showing good behavior, while the ratio $\frac{\dt}{\dx}$ varies from $0.40$ to $6.40$.

We have numerically also tested the case when $b=0$ (pure diffusion);
the numerical results are very close to the present case.


\FULL{
\begin{figure}[!hbtp]\label{fig:Diff_RK3}
\centering \subfloat {\includegraphics[scale=0.5]{Diff_RK3}}
\caption{\Exsix\ (1D diffusion) results for $k=3$ and $\dt\equiv\dx$}
\end{figure}
}

\if{
\begin{table}[H]
\begin{center}
\begin{tabular}{c c | c c|c c|c c}
\multicolumn{8}{c}{\uline{$k=4, RK2, \text{and } M=1000$}}\\
$M$ & $N$ & $L^1$-Error & order & $L^2$-Error & order & $L^\infty$-Error & order\\
\hline
\hline
     1000 &      10   & 5.428E-05 &     -     & 3.278E-05 &     -     & 3.644E-05 &     -      \\
     1000 &      20   & 1.334E-05 & 2.025     & 8.051E-06 & 2.026     & 8.950E-06 & 2.025      \\
     1000 &      40   & 3.307E-06 & 2.012     & 1.995E-06 & 2.013     & 2.218E-06 & 2.013      \\
     1000 &      80   & 8.233E-07 & 2.006     & 4.966E-07 & 2.006     & 5.520E-07 & 2.006      \\
     1000 &     160   & 2.054E-07 & 2.003     & 1.239E-07 & 2.003     & 1.377E-07 & 2.003      \\
     1000 &     320   & 5.130E-08 & 2.002     & 3.094E-08 & 2.002     & 3.439E-08 & 2.002      \\
\\
\multicolumn{8}{c}{\uline{$k=4, RK3, \text{and } M=1000$}}\\
$M$ & $N$ & $L^1$-Error & order & $L^2$-Error & order & $L^\infty$-Error & order\\
\hline
\hline
     1000 &      10   & 2.183E-07 &     -     & 1.371E-07 &     -     & 1.522E-07 &     -      \\
     1000 &      20   & 2.696E-08 & 3.017     & 1.692E-08 & 3.018     & 1.880E-08 & 3.018      \\
     1000 &      40   & 3.349E-09 & 3.009     & 2.103E-09 & 3.009     & 2.336E-09 & 3.009      \\
     1000 &      80   & 4.175E-10 & 3.004     & 2.620E-10 & 3.004     & 2.911E-10 & 3.004      \\
     1000 &     160   & 5.241E-11 & 2.994     & 3.275E-11 & 3.000     & 3.638E-11 & 3.000      \\
     1000 &     320   & 7.163E-12 & 2.871     & 4.184E-12 & 2.968     & 4.657E-12 & 2.966      \\
\end{tabular}
\end{center}
\caption{\label{tab:5.1b}}
\end{table}
}\fi

\begin{table}[!hbtp]
\begin{center}
\begin{tabular}{|cc|cc|cc|cc|}
\hline
  \multicolumn{2}{|c|}{$L^2$ error} 
   & \multicolumn{2}{|c|}{SLDG-RK1 ($P_1$)} & \multicolumn{2}{|c|}{SLDG-RK2 ($P_2$)} & \multicolumn{2}{|c|}{SLDG-RK3 ($P_3$)} \\
\hline
  $M$ & $N$  &  error   & order  & error    & order  &  error   & order  \\
\hline
\hline
 10 &   10  & 9.94E-03 &   -    & 1.37E-03 &    -   & 8.66E-05 &     -  \\ 
 20 &   20  & 1.39E-03 &  2.84  & 1.08E-04 &  3.67  & 3.70E-06 &  4.55  \\ 
 40 &   40  & 2.93E-04 &  2.25  & 3.63E-06 &  4.90  & 1.03E-07 &  5.17  \\ 
 80 &   80  & 8.02E-05 &  1.87  & 6.28E-07 &  2.53  & 9.81E-09 &  3.39  \\ 
160 &  160  & 2.35E-05 &  1.77  & 9.72E-08 &  2.69  & 7.00E-10 &  3.81  \\ 
320 &  320  & 8.22E-06 &  1.52  & 2.60E-08 &  1.90  & 5.79E-11 &  3.60  \\ 
640 &  640  & 4.06E-06 &  1.02  & 6.17E-09 &  2.08  & 5.81E-12 &  3.32  \\ 
\hline
\end{tabular}
\caption{\label{tab:ex5.2b} \textit{\Exsix\ (1D diffusion),}
SLDG-RKp schemes with $\dt \sim\dx$.
}
\end{center}
\end{table}

\begin{table}[!hbtp]
\begin{center}
\begin{tabular}{|cc|c|c|c|}
\hline
  \multicolumn{2}{|c|}{$L^2$ error} 
   & \multicolumn{1}{|c|}{SLDG-RK1 ($P_1$)} & \multicolumn{1}{|c|}{SLDG-RK2 ($P_2$)} & \multicolumn{1}{|c|}{SLDG-RK3 ($P_3$)} \\
\hline
  $M$ & $N$  & error  &  error &  error \\
\hline \hline
  20 &   10  &  1.37E-03   & 4.34E-05   & 1.79E-06   \\ 
  40 &   15  &  5.13E-04   & 6.87E-06   & 1.41E-07   \\ 
  80 &   20  &  1.39E-04   & 1.40E-06   & 1.11E-08   \\ 
 160 &   25  &  1.05E-04   & 1.83E-07   & 5.20E-10   \\ 
 320 &   30  &  8.49E-05   & 6.14E-08   & 3.09E-11   \\ 
 640 &   35  &  7.26E-05   & 4.35E-08   & 1.15E-11   \\ 
1280 &   40  &  6.35E-05   & 3.31E-08   & 7.02E-12   \\ 
\hline
\end{tabular}
\caption{\label{tab:ex5.3b} \textit{\Exsix\ (1D diffusion),}
SLDG-RKp with large time steps $\dt \gg \dx$. 
}
\end{center}
\end{table}

\if{
\begin{table}[!hbtp]
\MODIF{
\begin{center}
\begin{tabular}{|cc|cc|cc|cc|}
\hline
  \multicolumn{2}{|c|}{$L^2$ error} 
   & \multicolumn{2}{|c|}{SLDG-RK1 ($P_1$)} & \multicolumn{2}{|c|}{SLDG-RK2 ($P_2$)} & \multicolumn{2}{|c|}{SLDG-RK3 ($P_3$)} \\
\hline
  $M$ & $N$  & error & order  & error & order  & error & order   \\
\hline \hline
    20 &    10  & 1.37E-03 &  -     & 4.34E-05 &  -     & 1.79E-06 &  -    \\
    40 &    15  & 5.13E-04 &  2.42  & 6.87E-06 &   4.55 & 1.41E-07 &  6.27 \\
    80 &    20  & 1.39E-04 &  4.54  & 1.40E-06 &   5.53 & 1.11E-08 &  8.84 \\
   160 &    25  & 1.05E-04 &  1.26  & 1.83E-07 &   9.12 & 5.20E-10 & 13.72 \\
   320 &    30  & 8.49E-05 &  1.17  & 6.14E-08 &   5.99 & 3.09E-11 & 15.48 \\
   640 &    35  & 7.26E-05 &  1.02  & 4.35E-08 &   2.24 & 1.15E-11 &  6.41 \\
  1280 &    40  & 6.35E-05 &  1.00  & 3.31E-08 &   2.05 & 7.02E-12 &  3.70 \\
\hline
\end{tabular}
\caption{\label{tab:ex5.3b} \textit{\Exsix\ (1D diffusion),}
SLDG-RKp with large time steps $\dt \gg \dx$.
}
\end{center}
}
\end{table}
}\fi

\FULL{
\begin{figure}[!hbtp]\label{fig:DiffConv_RK3}
\centering
\subfloat {\includegraphics[scale=0.5]{DiffConv_RK3}}
\caption{Results for $k=3, \dt \equiv \dx$}
\end{figure}
}

\medskip

\noindent{\bf \Exhei\ (1D Black and Scholes and boundary conditions)}
This example deals with 
the one-dimensional Black-Scholes (B\&S) PDE for the pricing of a European put option with one asset \cite{Wilmott}.
After a change of variable in logarithmic coordinates,\footnote{
The classical B\&S PDE for the put option reads
$$ v_t - \frac{1}{2} \ms^2 s^2 v_{ss}  - b s v_s + r v = 0, \qquad s \in (0,\infty),\
  t \in (0,T), 
$$
(where $b=r-\frac{1}{2}\ms^2$),
with initial condition $v(0,s)=\varphi(s)\equiv \max(K-s,0)$. Then using the change of variable $x=\log(s/K)$ and $u(t,x):=v(t,s)$,
we obtain the PDE \eqref{eq:BS} on $x\in \R$.
}
the equation for the European put option becomes on $\mO:=(\xmin,\xmax)$:
\begin{eqnarray}\label{eq:BS}
\begin{cases} 
  u_t - \frac{1}{2} \ms^2 u_{xx} + b u_x + r u = 0, \qquad x \in \mO,\ t \in (0,T), \\
  u(0,x)=u_0(x)=K \max(1-e^x,0) \qquad x\in \mO,\\
  u(t,x)=u_\ell(t)\equiv K e^{-rt} - K e^x \qquad t\in (0,T),\ x\leq \xmin,\\
  u(t,x)=u_r(t)\equiv 0 \qquad t\in (0,T), \ x\geq \xmax,
\end{cases}
\end{eqnarray}
with $b:=-(r-\frac{1}{2}\ms^2)$ and where $\xmin<0$ and $\xmax>0$,
and we have imposed boundary conditions outside of $\mO$.
Numerically, the initial datum exhibits singular behavior at $x=0$ (as it is only Lipschitz regular).

For this PDE the scheme reads 
$$
   u^{n+1} = e^{-r\dt} S^{\ms}_\dt \cT^b_{\dt} u^n.
$$

The following financial parameters are used: $K=100$ (strike price), $r=0.10$ (interest rate), $\sigma=0.2$ (volatility), and 
$T=0.25$ (maturity).
Since the interesting part of the solution lies in a neighborhood of $x=0$ (notice that $\varphi$ has a singularity at $x=0$),
for the computational domain we consider
\begin{equation*}  
   \mO = (\xmin, \xmax) = (-2,2).
\end{equation*}
In principle the PDE should be considered with $|\xmin|,|\xmax|>\!\!>1$, 
but here it can be numerically observed that the solution doesn't really change
for \mbox{$|\xmin|,|\xmax|\geq 2$.}

\MODIF{
Results are reported in Table~\ref{tab:8.1} for the $L^2$ errors, 
where $\dt$ is chosen of the same order as $\dx$, and the SLDG-RK1 SLDG-RK2 and SLDG-RK3 schemes are compared, together with a $P_4$ polynomial basis 
($k=4$). We used a $P_4$ basis so that the error from the spatial approximation is in principle negligible with respect to the time discretisation error.
We numerically observe the expected order 1 (resp. 2) for the SLDG-RK1 (resp. SLDG-RK2) scheme,
and approximatly order 3 for the SLDG-RK3 scheme (of expected theoretical order 3).
}

\begin{rem}[Boundary treatment]
For semi-Lagrangian schemes, the knowledge of $u(t,x)$ for $x\leq \xmin$ or $x\geq \xmax$ can be used if it is available.
Here, "out-of-bound" values are needed for computing $S^0 v^n$, $S^0 S^0 v^n$ and $S^0 S^0 S^0 v^n$ for $v^n=\cT^b_{\dt} u^n$. 
In particular, the values $u^n(x+k\ms\sqrt{\dt}-b\dt)$ for $|k|\leq 3$ are used when $y:=x+k\ms\sqrt{\dt}-b\dt$ lies outside of $(\xmin,\xmax)$. 
In that case, we simply directly use the "out-of-bounds" values $u_\ell(t_n,y)$ when $y\leq \xmin$  or $u_r(t_n,y)$ when $y\geq \xmax$.

It is clear that this will not work for a general PDE posed on a given domain with given boundary conditions.
(See however \cite{ach-fal-2012} for an example of a semi-Lagrangian scheme applied to a PDE with Neuman boundary conditions.)
\end{rem}

\if{
\begin{table}[!hbtp]
\begin{center}
\begin{tabular}{|cc|cc|cc|cc|}
\hline
$M$ & $N$ & $L^1$-Error & order & $L^2$-Error & order & $L^\infty$-Error & order\\
\hline
%
%
   10 &   10  & 2.99E-02 &   -    & 3.49E-02 &     -  & 6.22E-02 &     -  \\ 
   20 &   20  & 9.50E-04 &  4.98  & 1.50E-03 &  4.55  & 3.25E-03 &  4.26  \\ 
   40 &   40  & 5.04E-05 &  4.24  & 8.20E-05 &  4.19  & 2.82E-04 &  3.53  \\ 
   80 &   80  & 1.66E-06 &  4.92  & 2.75E-06 &  4.90  & 1.12E-05 &  4.65  \\ 
  160 &  160  & 4.79E-08 &  5.12  & 8.84E-08 &  4.96  & 4.49E-07 &  4.64  \\ 
\hline
\end{tabular}
\end{center}
\caption{\label{tab:8.1} \textit{\Exhei, 1D Black and Scholes PDE.}
Error table with $\dt\sim \dx$, using SLDG-RK3 and $P_3$ polynomials.
}
\MODIF{\fbox{TODO}: add RK1 and RK2  method, keep $L^2$ errors, keep CPUtimes}
\end{table}
}\fi

\begin{table}[!hbtp]
\MODIF{
\begin{center}
\begin{tabular}{|cc|ccc|ccc|ccc|}
\hline
\multicolumn{2}{|c|}{$L^2$ error}  
   & \multicolumn{3}{|c|}{SLDG-RK1} & \multicolumn{3}{|c|}{SLDG-RK2} & \multicolumn{3}{|c|}{SLDG-RK3} \\
\hline
$M$ & $N$ & error & order & cpu(s) & error & order & cpu(s) & error & order & cpu(s)\\
\hline
      10 &    10  & 6.30E-02 &  -     &   0.001  & 3.84E-02 &  -     &   0.001 & 4.17E-02 &  -    &   0.004 \\
      20 &    20  & 6.63E-03 &  3.25  &   0.008  & 2.27E-03 &   4.08 &   0.004 & 2.49E-03 &  4.07 &   0.004 \\
      40 &    40  & 2.54E-03 &  1.39  &   0.012  & 1.00E-04 &   4.50 &   0.016 & 1.24E-04 &  4.32 &   0.016 \\
      80 &    80  & 1.26E-03 &  1.01  &   0.028  & 4.11E-06 &   4.61 &   0.036 & 4.58E-06 &  4.76 &   0.040 \\
     160 &   160  & 6.28E-04 &  1.00  &   0.124  & 7.85E-07 &   2.39 &   0.124 & 1.13E-07 &  5.34 &   0.152 \\
     320 &   320  & 3.14E-04 &  1.00  &   0.424  & 1.94E-07 &   2.01 &   0.464 & 1.17E-08 &  3.27 &   0.528 \\
     640 &   640  & 1.57E-04 &  1.00  &   1.668  & 4.84E-08 &   2.00 &   1.805 & 1.23E-09 &  3.25 &   2.128 \\
\hline
\end{tabular}
\end{center}
\caption{
\MODIF{
\label{tab:8.1} \textit{\Exhei\ (1D Black and Scholes PDE).}
Error table with $\dt\sim \dx$, using SLDG-RK1, SLDG-RK2 and SLDG-RK3 methods with $P_4$ polynomials ($k=4)$.
}
}
}
\end{table}

\if{
We now turn on the corresponding American put option, which is 
the similar convection-diffusion equation with an obstacle term:
\begin{eqnarray}\label{eq:AMER}
\begin{cases} 
  \min\bigg(u_t - \frac{1}{2} \ms^2 u_{xx} - b u_x + r u,\ u-\varphi(x)\bigg) = 0 \qquad x \in \mO,\ t \in (0,T), \\
  u(0,x)=\varphi(x)=K \max(1-e^x,0) \qquad x\in \mO,\\
  u(t,\xmin)=\bar u_\ell(t)\equiv K - K e^x \qquad t\in (0,T),\\
  u(t,\xmax)=\bar u_r(t)\equiv 0 \qquad t\in (0,T).
\end{cases}
\end{eqnarray}
Here we have imposed also boundary conditions.

In that case we do not expect high-order behavior of the scheme everywhere. 
However in a region where the function is smooth, we may expect
to find again high-order.

We do not know the analytical solution for the American option
We have first computed $20$ reference values $\bar u_i$ corresponding to point $\bar x_i$ such that $s_i=K e^{\bar x_i}$ is uniformly distributed in 
$[-90,\dots,110]$ by using the SLDG-RK3 scheme with $P_3$ polynomials, 
and $M=N=10000$.
In Table~\ref{tab:8.2}, we then estimate a normalized $L^2$ error on these reference values and using coarser discretization:
$$ err_{L^2}\equiv \bigg(\sum_i |u^N(x_i)-\bar u_i|^2/(s_{\max}-s_{\min})\bigg)^{1/2}.$$
The corresponding $L^\infty$ error is about one digit lower precision.

\Olivier{Explain the projection step}

\begin{table}[!hbtp]
\begin{center}
\begin{tabular}{|cc|cc|cc|cc|}
\hline
\multicolumn{2}{|c|}{$L^2$ error}  
   & \multicolumn{2}{|c|}{SLDG-RK1} & \multicolumn{2}{|c|}{SLDG-RK2} & \multicolumn{2}{|c|}{SLDG-RK3} \\
\hline
  $M$ & $N$    &  error   &order & error    & order & error    & order\\ 
\hline
 10 &   10  & 6.41E-01 &  -     & 3.44E-01 &  -     & 1.06E-01 &   -  \\ 
 20 &   20  & 4.29E-01 &  0.58  & 8.78E-02 &  1.97  & 1.46E-02 &  2.87  \\ 
 40 &   40  & 1.61E-01 &  1.41  & 9.26E-03 &  3.24  & 7.89E-04 &  4.21  \\ 
 80 &   80  & 4.24E-02 &  1.93  & 7.32E-04 &  3.66  & 9.62E-05 &  3.04  \\ 
160 &   160 & 1.05E-02 &  2.02  & 9.34E-05 &  2.97  & 5.22E-06 &  4.20  \\ 
\hline
\end{tabular}
\end{center}
\caption{\label{tab:8.2} \Exhei, $1D$ nonlinear American option with $r=0$, $T=0.25$, $\ms=0.2$. 
Error computed on $20$ referenced values around $x=0$ ($x\in[\log(0.9),\log(1.1)]$),
and with $\dt\sim \dx$.
}
\end{table}

\begin{figure}[!hbtp]
\centering
\includegraphics[scale=0.5]{gobs.pdf}
\caption{\label{fig:amer} nonlinear American option (\Exhei), $M=N=80$, K3-RK3.}
\end{figure}
}\fi

\noindent{\bf \Exsev\ (1D diffusion with non-constant $\ms(x)$).}
Now, we consider the following diffusion equation 
\be\label{eq:ex6}
 & & v_t  - \frac{1}{2}\sigma^2(x) v_{xx} = f(t,x), \qquad  x \in (0,1),\  t \in (0,T) \\
 & & v(0,x)= 0 \qquad x\in(0,1),
\ee
with periodic boundary conditions, 
$$\sigma(x):=\sin(2\pi x),$$
and, for testing purposes,
$
  f(t,x):={\bar v}_t(t,x) - \frac{1}{2}\ms^2(x) {\bar v}_{xx}(t,x)
$
where ${\bar v}(t,x):=\sin(2\pi t)\, \cos(2\pi(x-t))$, which is therefore the exact solution ($v\equiv \bar v$).

In this case, in order to get higher than first-order accuracy in time, we use the SLDG-2 scheme corresponding to a 
Platen's weak Taylor scheme.
The correction for the source term $f(t,x)$ is treated by adding the term \eqref{eq:whx-correction} 
at Gauss quadrature points, at each time step.


In Table~\ref{tab:ex6.1} we first check the accuracy with respect to time discretization, with fixed spatial mesh size
so that only the time discretization error appears.

Then, in Table~\ref{tab:ex6.2} the errors are given for varying mesh sizes such that $\dt\equiv\dx$ and with $P_1$ or $P_2$ elements 
($k=1$ or $k=2$). We find the expected orders for the schemes SLDG-1/2.  

\begin{rem}\label{rem:vanish-sigma}
Notice that there is no need 
for an assumption that the diffusion coefficient is non-vanishing in the proposed method.
\end{rem}

\begin{table}[!hbtp]
\begin{center}
\begin{tabular}{|c|cc|cc|}
\hline
$L^2$ error
   & \multicolumn{2}{|c|}{SLDG-1} & \multicolumn{2}{|c|}{SLDG-2} \\
\hline
   $N$    &  error   &order & error    & order \\
\hline
  100   &  1.19E-03  &  -    & 1.89E-04 & 2.05 \\
  200   &  5.95E-04  &  1.01 & 4.57E-05 & 1.97 \\
  400   &  2.96E-04  &  1.01 & 1.16E-05 & 1.93 \\
  800   &  1.48E-04  &  1.00 & 3.07E-06 & 1.91 \\
 1600   &  7.40E-05  &  1.00 & 8.17E-07 & 1.92 \\
\hline
\end{tabular}
\end{center}
\caption{\label{tab:ex6.1}
\textit{\Exsev\ (1D diffusion with non-constant coefficient)},
with fixed spatial mesh ($M=100$ and $P_4$ polynomials) and varying time steps $N$.
}
\end{table}

\begin{table}[!hbtp]
\begin{center}
\begin{tabular}{|cc|cc|cc|}
\hline
\multicolumn{2}{|c|}{$L^2$ error}  
   & \multicolumn{2}{|c|}{SLDG-1 (with $P_1$)} & \multicolumn{2}{|c|}{SLDG-2 (with $P_2$)} \\
\hline
   $M$ & $N$    &  error   &order & error    & order \\
\hline
    10 &     10   &  8.60E-02 &  -   & 4.13E-02  &  -   \\
    20 &     20   &  3.52E-02 & 1.29 & 7.30E-03  & 2.50 \\
    40 &     40   &  1.59E-02 & 1.15 & 1.39E-03  & 2.39 \\
    80 &     80   &  7.54E-03 & 1.08 & 3.03E-04  & 2.20 \\
   160 &    160   &  3.67E-03 & 1.04 & 7.17E-05  & 2.08 \\
   320 &    320   &  1.81E-03 & 1.02 & 1.80E-05  & 1.99 \\
\hline
\end{tabular}
\end{center}
\caption{\label{tab:ex6.2} 
\textit{\Exsev\ (1D diffusion with non-constant coefficient)},
with $\dt\sim\dx$; 
}
\end{table}

\medskip

\noindent{\bf \Exnin\ (2D diffusion)} 
We consider the following two-dimensional diffusion equation:
\be
  & & u_t - \frac{1}{2} (5 u_{xx} -4 u_{xy} + u_{yy}) = 0, \quad  x\in \mO,\ t\in(0,T), \\
  & & u(0,x)=u_0(x), \quad x\in \mO 
\ee
set on $\mO=(0,1)^2$ with periodic boundary conditions, and $T=0.2$.
The initial datum is given by 
$ u_0(x)=u_{01}(x+2y) + u_{02}(-y) $ and $u_{0i}(\xi):=\sum_{q=1,2} c^i_q \cos(2\pi q \xi)$ with the constant $c^i_q=\frac{1}{i+q}$.
The exact solution is known.%
\footnote{
Making the change of variable $\xi=(\xi_1,\xi_2)$ such that $\xi_1=x+2y$ and \mbox{$\xi_2=-y$} we find that $v(t,\xi)=u(t,x)$
satisfies $v_t -\frac{1}{2} (v_{\xi_1\xi_1} + v_{\xi_2\xi_2})=0$ and $v(0,\xi)=u_{01}(\xi_1) + u_{02}(\xi_2)$ and therefore the
exact solution is given by 
  $u(t,x)=v(t,\xi)=u_1(t,\xi_1) + u_2(t,\xi_2)$
where 
  $u_i(t,\xi) = \sum_{q=1,2} c^i_q e^{-(2\pi q)^2 t/2} \cos(2\pi q \xi)$.
}

In order to define the numerical scheme, we use the fact that
$$ A:=\left[\barr{rr} 5 & -2 \\ -2 & 1\earr\right] = \sum_{k=1,2}\ms_k \ms_k^T, \qquad 
\mbox{ with }
   \ms_1:=\VECT{1\\0}, \  \ms_2:=\VECT{2\\-1}.
$$ 

\FULL{
The results are given in Table~\ref{tab:7.1} and Table~\ref{tab:7.2}.

In Table~\ref{tab:7.1}, we check that the time discretization 
schemes SLDG-RK1, SLDG-RK2 and SLDG-RK3 have the correct order ($1$, $2$  or $3$)  by using a fixed spatial mesh ($M=200$ intervals)
and fixed spatial order ($P_4$, or $k=4$). 

In Table~\ref{tab:7.2} we then consider variable time steps and mesh steps.
Setting $\dt\sim \dx$ we expect a global order of 
$O(\dt^{q}) + O(\frac{\dx^{q+1}}{\dt}) \equiv O(\dx^q)$, which is roughly verified for $q=1,2$ and $3$. 
}

\SHORT{
The results are  given in Table~\ref{tab:7.2}, where we consider variable time steps and mesh steps $\dt\sim \dx$,
$p=k$, and expect a global error of order $O(\dt^{p}) + O(\frac{\dx^{k+1}}{\dt}) \equiv O(\dx^k)$.
}

\MODIF{
In this example involving constant diffusion coefficients, we test up to third-order schemes.
}



\FULL{
\begin{table}[!hbtp]
\begin{center}
\begin{tabular}{|c|cc|cc|cc|}
\hline
$L^2$ error
   & \multicolumn{2}{|c|}{SLDG-RK1} & \multicolumn{2}{|c|}{SLDG-RK2} & \multicolumn{2}{|c|}{SLDG-RK3} \\
\hline
   $N$    &  error   &order & error    & order & error    & order\\ 
\hline
    10    & 6.67E-03 &   -  & 1.86E-04 &   -   & 2.20E-06 &   -   \\
    20    & 3.26E-03 & 1.03 & 4.44E-05 & 2.07  & 2.62E-07 & 3.07  \\
    40    & 1.61E-03 & 1.01 & 1.08E-05 & 2.03  & 3.20E-08 & 3.03  \\
    80    & 8.04E-04 & 1.00 & 2.68E-06 & 2.01  & 3.95E-09 & 3.01  \\
   160    & 4.01E-04 & 1.00 & 6.66E-07 & 2.00  & 4.90E-10 & 3.00  \\
\hline
\end{tabular}
\end{center}
\caption{\label{tab:7.1} \Exnin\ (2D diffusion equation), 
error table with variable $\dt$ and fixed mesh size ($M=200$ and order $k=4$).
}
\end{table}
}


\begin{table}[!hbtp]
\begin{center}
\begin{tabular}{|cc|cc|cc|cc|}
\hline
\multicolumn{2}{|c|}{$L^2$ error}  
   & \multicolumn{2}{|c|}{SLDG-RK1 ($Q_1$)} & \multicolumn{2}{|c|}{SLDG-RK2 $(Q_2)$} & \multicolumn{2}{|c|}{SLDG-RK3 $(Q_3)$} \\
\hline
  $M_1=M_2$ & $N$    &  error   &order & error    & order & error    & order\\ 
\hline
    10    &     10    & 6.66E-03 &   -  & 1.86E-04 &    -  & 2.20E-06 &    -   \\
    20    &     20    & 3.26E-03 & 1.02 & 4.52E-05 & 2.04  & 3.10E-07 & 2.83   \\
    40    &     40    & 1.61E-03 & 1.01 & 1.08E-05 & 2.06  & 3.20E-08 & 3.27   \\
    80    &     80    & 8.04E-04 & 1.00 & 2.69E-06 & 2.01  & 4.34E-09 & 2.88   \\
   160    &    160    & 4.01E-04 & 1.00 & 6.66E-07 & 2.01  & 4.90E-10 & 3.14   \\

\hline
\end{tabular}
\end{center}
\caption{\label{tab:7.2} \textit{\Exnin\ (2D diffusion equation),} 
error table with $\dt\sim \dx$ using $Q_k$ polynomials.
}
\end{table}

\if{
\begin{table}[h]
\begin{center}
\begin{tabular}{c c | c c|c c|c c}
\multicolumn{8}{c}{\uline{\mbox{"SLDG-RK1" ($k=4$)}}}\\
$M$ & $N$ & $L^1$-Error & order & $L^2$-Error & order & $L^\infty$-Error & order\\
\hline
\hline
    200    &     10    & 1.588E-02 &   -       & 6.671E-03 &   -       & 7.990E-03 &   -    \\
    200    &     20    & 7.832E-03 & 1.020     & 3.267E-03 & 1.030     & 3.913E-03 & 1.030  \\
    200    &     40    & 3.893E-03 & 1.009     & 1.617E-03 & 1.015     & 1.937E-03 & 1.015  \\
    200    &     80    & 1.941E-03 & 1.004     & 8.045E-04 & 1.007     & 9.635E-04 & 1.007  \\
    200    &    160    & 9.692E-04 & 1.002     & 4.012E-04 & 1.004     & 4.805E-04 & 1.004 \\
\\
\multicolumn{8}{c}{\uline{\mbox{"SLDG-RK2" ($k=4$)}}}\\
$M$ & $N$ & $L^1$-Error & order & $L^2$-Error & order & $L^\infty$-Error & order\\
\hline
\hline
    200    &     10    & 5.108E-04 &   -       & 1.869E-04 &   -       & 2.242E-04 &   -    \\
    200    &     20    & 1.180E-04 & 2.113     & 4.442E-05 & 2.073     & 5.325E-05 & 2.074  \\
    200    &     40    & 2.855E-05 & 2.048     & 1.084E-05 & 2.034     & 1.300E-05 & 2.035  \\
    200    &     80    & 7.029E-06 & 2.022     & 2.680E-06 & 2.017     & 3.212E-06 & 2.017  \\
    200    &    160    & 1.744E-06 & 2.011     & 6.660E-07 & 2.008     & 7.983E-07 & 2.008  \\
\\
\multicolumn{8}{c}{\uline{\mbox{"SLDG-RK3" ($k=4$)}}}\\
$M$ & $N$ & $L^1$-Error & order & $L^2$-Error & order & $L^\infty$-Error & order\\
\hline
\hline
    200    &     10    & 6.989E-06 &   -       & 2.208E-06 &   -       & 2.606E-06 &   -    \\
    200    &     20    & 8.228E-07 & 3.087     & 2.625E-07 & 3.072     & 3.100E-07 & 3.071  \\
    200    &     40    & 9.977E-08 & 3.044     & 3.200E-08 & 3.036     & 3.782E-08 & 3.035  \\
    200    &     80    & 1.228E-08 & 3.022     & 3.951E-09 & 3.018     & 4.670E-09 & 3.017  \\
    200    &    160    & 1.524E-09 & 3.011     & 4.908E-10 & 3.009     & 5.803E-10 & 3.009  \\
\\
\end{tabular}
\end{center}
\caption{\label{tab:6.1} \Exsix\ ($2D$ diffusion equation), 
testing the time order for "SLDG-RK1", "SLDG-RK2" and "SLDG-RK3" diffusion schemes,
with fixed spatial order $k=4$ and $M=200$ intervals.
}
\end{table}
}\fi

\medskip


\noindent{\bf \Exten\ (2D diffusion with non-constant coefficients)}
We consider the following two-dimensional diffusion equation:
\be\label{eq:2dvar}
  & & u_t - \frac{1}{2}  Tr(\ms\ms^T D^2u) = f(t,x), \quad  x\in \mO,\ t\in(0,T), \\
  & & u(0,x,y)=u_0(x,y), \quad (x,y)\in \mO 
\ee
set on $\mO=(-\pi,\pi)^2$ with periodic boundary conditions, $T=1.0$.
The diffusion matrix $A=\ms\ms^T$ is defined by  
$$
  \ms(x,y):=\MAT{cc}{cos(x) & cos(2x)\\ 0 & sin(y)}.
$$
In this test we have chosen $u(t,x,y):=cos(t)sin(2x)sin(x+y)$ and the source term $f(t,x)$ such that \eqref{eq:2dvar} holds.
(The initial datum is therefore $u_0(x,y)=u(0,x,y)$).

\MODIF{
The scheme is defined here by using either 
\begin{itemize}
\item
the weak Euler scheme for the diffusion part, combined with Trotter's splitting (and with $Q_1$ polynomials) 
and a first-order correction for the source tem (as in \eqref{eq:whx-approx}).
\item
the weak Platen scheme for the diffusion part, combined with Strang's splitting (and with $Q_2$ polynomials)
and a second-order correction for the source term \eqref{eq:whx-approx},
as explained in Section~\ref{sec:4.2}.
\end{itemize}
}



The results for $L^2$ errors are  given in Table~\ref{tab:2Dvar}, where we consider variable time steps and mesh steps $\dt\sim \dx$.
(see Section~\ref{sec:4.1}).
The schemes are numerically roughly of the expected orders $1$ and $2$. 

As mentioned in Remark~\ref{rem:vanish-sigma}, 
there is no need to assume strict positivity of the diffusion matrix in this approach.

\if{
\begin{table}[!hbtp]
\begin{center}
\begin{tabular}{|cc|cc|cc|cc|c|}
\hline
   &  & \multicolumn{2}{|c|}{$L^1$ error} & \multicolumn{2}{|c|}{$L^2$ error} & \multicolumn{2}{|c|}{$L^\infty$ error} & \\
\hline
  $M$ & $N$    &  error   &order & error    & order & error    & order & cpu(s)\\ 
\hline
     3 &     2  & 1.23E+01 &  -     & 2.47E+00 &  -     & 9.82E-01 &  -     &    0.01 \\
     6 &     4  & 2.29E+00 &  2.43  & 4.52E-01 &   2.45 & 1.65E-01 &  2.57  &    0.03 \\
    12 &     8  & 3.60E-01 &  2.67  & 7.60E-02 &   2.57 & 3.03E-02 &  2.45  &    0.09 \\
    24 &    16  & 7.24E-02 &  2.32  & 1.55E-02 &   2.29 & 5.96E-03 &  2.35  &    0.39 \\
    48 &    32  & 1.69E-02 &  2.10  & 3.54E-03 &   2.14 & 1.19E-03 &  2.32  &    2.08 \\
    96 &    64  & 4.61E-03 &  1.87  & 9.52E-04 &   1.89 & 3.25E-04 &  1.87  &   12.28 \\
\hline
\end{tabular}
\end{center}
\caption{\label{tab:2Dvar} \textit{\Exten} ($2D$ diffusion equation with variable coefficients), error table with $\dt\sim \dx$.
}
\end{table}
}\fi

\begin{table}[!hbtp]
\begin{center}
\MODIF{
\begin{tabular}{|cc|ccc|ccc|}
\hline
   \multicolumn{2}{|c|}{$L^2$ error} & \multicolumn{3}{|c|}{Euler/Trotter (with $Q_1$)} & \multicolumn{3}{|c|}{Platen/Strang (with $Q_2$)} \\
\hline
   $M_1=M_2$ & $N$    &  error   &  order &  cpu(s) &  error   & order  &  cpu(s)\\
\hline
     5 &    10  & 1.50E+00 &  -     & 0.01     & 2.96E-01 &  -     &  0.020 \\
    10 &    20  & 4.98E-01 &   1.59 & 0.02     & 3.14E-02 &   3.24 &  0.088 \\
    20 &    40  & 9.63E-02 &   2.37 & 0.11     & 3.40E-03 &   3.21 &  0.432 \\
    40 &    80  & 2.87E-02 &   1.75 & 0.74     & 7.10E-04 &   2.26 &  2.564 \\
    80 &   160  & 1.07E-02 &   1.43 & 5.44     & 1.66E-04 &   2.09 & 16.621 \\
\hline
\end{tabular}
}
\end{center}
\caption{\label{tab:2Dvar} \textit{\Exten} ($2D$ diffusion equation with variable coefficients)
}
\end{table}

\medskip

%
%


\appendix


\section{Instability of the direct scheme}\label{app:unstabilities}

Here we consider the "direct scheme", which defines  naively at each time iteration a new piecewise polynomial $u^{n+1}\in V_k$ such that, 
$$
   u^{n+1,i}_\ma := u^n(x^i_\ma-b\dt), \quad \mbox{for all Gauss points $x^i_\ma$.}
$$

In Figure~\ref{fig:1}, we consider again $v_t + v_x=0$
with periodic boundary conditions on $(0,1)$, and with the initial data $v_0(x)=\sin(2\pi x)$.
We have depicted two graphs with different choices of the parameter~$N$. 
In each graph we plotted the result of the direct scheme (green line) and of
the SLDG scheme (red line) at time $T=1$, with piecewise $P_1$ elements ($k=1$) and fixed spatial mesh using $M=46$ mesh steps.
In the left graph, $N=80$ time steps and both curves are confounded; in the right graph, $N=320$, and the direct scheme becomes unstable.
(We have found that the error behaves as $c^N \dx^{k+1}$ where $c>1$, when using $P_k$ elements.)


\begin{figure}[!hbtp]\label{fig:stab1}
\centering
\subfloat[$N=80$] {\includegraphics[scale=0.30]{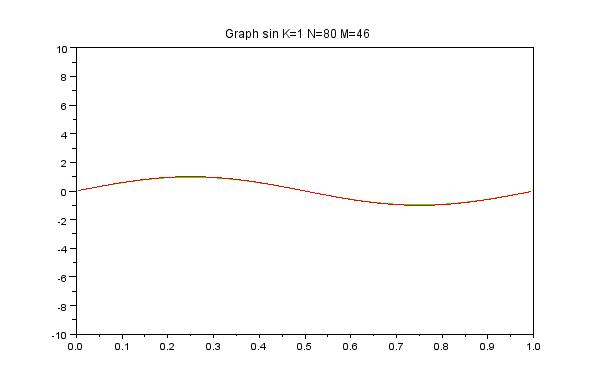}}
\subfloat[$N=320$]{\includegraphics[scale=0.30]{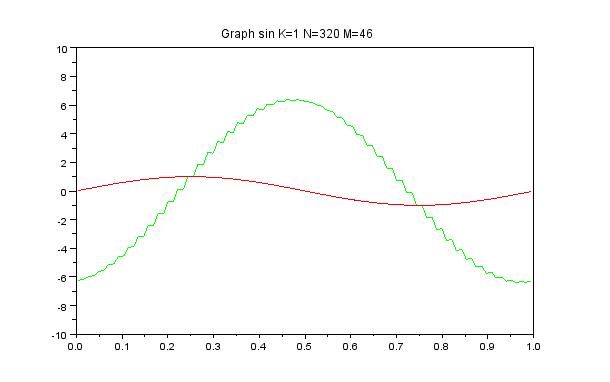}}
\caption{\label{fig:1} Results for $N=80$ (left) and $N=320$ (right), using $P_1$ elements with $M=46$ in both cases.
Instability appears on the right figure.}
\end{figure}

\medskip

\FULL{
We show also in Table~\ref{tab:unstable} 
some typical numerical results with the unstable schemes that can still appear to be good if the number
of time steps is not too large.

\begin{small}
\begin{table}[!hbtp]
\begin{center}
\begin{tabular}{|cc|cc|cc|cc|cc|}
\hline
\multicolumn{2}{|c|}{$L^2$ error}  
   & \multicolumn{2}{|c|}{k=1} & \multicolumn{2}{|c|}{k=2} & \multicolumn{2}{|c|}{k=3} & \multicolumn{2}{|c|}{k=4} \\
\hline
 Mesh $M$ & $N$ &  error   &order & error    & order & error    & order & error    & order \\
\hline
  23 & 10  & 8.17E-03 &  -   & 4.84E-05 &  -   & 2.27E-06 &    -  & 8.44E-09 &  -    \\
  46 & 20  & 4.16E-03 & 0.97 & 9.85E-06 & 2.29 & 2.54E-07 &  3.16 & 4.50E-10 & 4.23  \\
  92 & 40  & 2.11E-03 & 0.98 & 2.30E-06 & 2.09 & 4.04E-08 &  2.65 & 2.70E-11 & 4.05  \\
 184 & 80  & 1.06E-03 & 0.98 & 5.66E-07 & 2.02 & 1.54E-07 & -1.93 & 1.68E-12 & 4.00  \\
 368 & 160 & 5.33E-04 & 0.99 & 1.41E-07 & 2.00 & 7.44E+02 & -32.1 & 1.07E-13 & 3.97  \\
 736 & 320 & 2.67E-04 & 0.99 & 3.52E-08 & 2.00 & 4.30E+22 & -65.6 & 1.37E-14 & 2.97  \\
\hline
\end{tabular}
\end{center}
\caption{\label{tab:unstable}  Results for the Advection PDE, unstable scheme, CFL=2.3, $u_0(x)=sin(2\pi x)$, $T=1$.}
\end{table}
\end{small}

}

\if{
\begin{table}[!hbtp]
\begin{center}
\begin{tabular}{|c|c||c|c||c|c||c|c|}
\hline
\multicolumn{8}{|c|}{k=1}\\
\hline
 M &  N  & $L^\infty$ Error &  order & $L^1$ Error & order & $L^2$ Error & order\\
\hline
23 & 10 & 1.318E-02 &  -  & 7.277E-03 &  -  & 8.175E-03 &  -   \\
46 & 20 & 6.317E-03 & 1.061 & 3.739E-03 & 0.9608 & 4.165E-03 & 0.9730  \\
92 & 40 & 3.095E-03 & 1.029 & 1.898E-03 & 0.9778 & 2.110E-03 & 0.9809  \\
184 & 80 & 1.531E-03 & 1.015 & 9.568E-04 & 0.9885 & 1.063E-03 & 0.9892 \\  
368 & 160 & 7.616E-04 & 1.008 & 4.804E-04 & 0.9941 & 5.336E-04 & 0.9943 \\ 
736 & 320 & 3.798E-04 & 1.004 & 2.407E-04 & 0.9970 & 2.673E-04 & 0.9971 \\
\hline
\hline
\multicolumn{8}{|c|}{k=2}\\
\hline
 M &  N  & $L^\infty$ Error &  order & $L^1$ Error & order & $L^2$ Error & order\\
\hline
23 & 10 & 9.191E-05 &  -  & 4.022E-05 &  -  & 4.848E-05 &  -   \\
46 & 20 & 1.765E-05 & 2.381 & 8.433E-06 & 2.254 & 9.851E-06 & 2.299  \\
92 & 40 & 3.778E-06 & 2.224 & 2.042E-06 & 2.046 & 2.302E-06 & 2.097  \\
184 & 80 & 8.696E-07 & 2.119 & 5.077E-07 & 2.008 & 5.661E-07 & 2.024  \\
368 & 160 & 2.083E-07 & 2.061 & 1.269E-07 & 2.001 & 1.411E-07 & 2.005  \\
736 & 320 & 5.097E-08 & 2.031 & 3.173E-08 & 1.999 & 3.525E-08 & 2.001  \\
\hline
\hline
\multicolumn{8}{|c|}{k=3}\\
\hline
 M &  N  & $L^\infty$ Error &  order & $L^1$ Error & order & $L^2$ Error & order\\
\hline
23 & 10 & 4.720E-06 &  -  & 1.856E-06 &  -  & 2.278E-06 &  -   \\
46 & 20 & 5.548E-07 & 3.089 & 1.995E-07 & 3.217 & 2.548E-07 & 3.161  \\
92 & 40 & 9.127E-08 & 2.604 & 3.041E-08 & 2.714 & 4.043E-08 & 2.656  \\
184 & 80 & 5.548E-07  &- 2.604 & 1.212E-07 & - 1.995 & 1.549E-07 & - 1.938  \\
368 & 160 & 3.063E+03 & - 32.36 & 5.481E+02 & - 32.07 & 7.442E+02 & - 32.16  \\
736 & 320 & 1.588E+23  &- 65.49 & 3.354E+22 & - 65.73 & 4.307E+22 & - 65.65  \\
\hline
\hline
\multicolumn{8}{|c|}{k=4}\\
\hline
 M &  N  & $L^\infty$ Error &  order & $L^1$ Error & order & $L^2$ Error & order\\
\hline
 23 & 10  & 1.771E-08 &  -    & 6.670E-09 &  -    & 8.445E-09 &  -     \\
 46 & 20  & 9.093E-10 & 4.283 & 3.755E-10 & 4.151 & 4.500E-10 & 4.230  \\
 92 & 40  & 4.974E-11 & 4.192 & 2.368E-11 & 3.987 & 2.708E-11 & 4.055  \\
184 & 80  & 2.801E-12 & 4.150 & 1.494E-12 & 3.986 & 1.688E-12 & 4.004  \\
368 & 160 & 1.970E-13 & 3.830 & 9.443E-14 & 3.984 & 1.075E-13 & 3.973  \\
736 & 320 & 3.764E-14 & 2.388 & 1.117E-14 & 3.079 & 1.371E-14 & 2.970  \\
\hline
\hline
\multicolumn{8}{|c|}{k=5}\\
\hline
 M &  N  & $L^\infty$ Error &  order & $L^1$ Error & order & $L^2$ Error & order\\
\hline
23 & 10 & 8.068E-10 &  -  & 4.473E-10 &  -  & 5.053E-10 &  -   \\
46 & 20 & 2.376E-11 & 5.086 & 1.414E-11 & 4.984 & 1.577E-11 & 5.002  \\
92 & 40 & 7.180E-13 & 5.048 & 4.415E-13 & 5.001 & 4.909E-13 & 5.005  \\
184 & 80 & 1.621E-14 & 5.469 & 8.772E-15 & 5.654 & 9.778E-15 & 5.650  \\
368 & 160 & 2.909E-14  & - 0.8436 & 9.940E-15  & - 0.1803 & 1.143E-14  & - 0.2253  \\
736 & 320 & 5.649E-13  & - 4.279 & 1.002E-13  & - 3.334 & 1.336E-13  & - 3.547  \\
\hline
\end{tabular}
\end{center}
\caption{\label{tab:1} Results for the Advection PDE, unstable scheme, CFL$=2.3$, $u_0(x)=sin(2\pi x - \pi)$, $T=1$}
\end{table}
}\fi

\bibliographystyle{abbrv}
\def\cprime{$'$}


\end{document}